\documentclass[10pt]{article}
   \usepackage{graphicx}
   \usepackage{epsfig}
   \usepackage{amssymb}
   \usepackage{amsmath}
   \usepackage{amsthm}
   \newtheorem{lemma}{Lemma}[section]
   \newtheorem{theorem}{Theorem}[section]

   \newtheorem{remark}{Remark}[section]

   {\end{description}}

   {\hfill $\bullet$ \\}

   \newcommand{\be}{\begin{equation}}
   \newcommand{\ee}{\end{equation}}

\begin{document}
    \title{A Two-Level Fourth-Order Approach For Time-Fractional Convection-Diffusion-Reaction Equation With Variable Coefficients}
   \author{Eric Ngondiep$^{\text{\,a\,b}}$}
   \date{$^{\text{\,a\,}}$\small{Department of Mathematics and Statistics, College of Science, Imam Mohammad Ibn Saud\\ Islamic University
        (IMSIU), $90950$ Riyadh $11632,$ Saudi Arabia.}\\
     \text{\,}\\
       $^{\text{\,b\,}}$\small{Hydrological Research Centre, Institute for Geological and Mining Research, 4110 Yaounde-Cameroon.}\\
     \text{,}\\
        \textbf{Email addresses:} ericngondiep@gmail.com/engondiep@imamu.edu.sa}
   \maketitle

   \textbf{Abstract.}
   This paper develops a two-level fourth-order scheme for solving time-fractional convection-diffusion-reaction equation with variable coefficients
   subjected to suitable initial and boundary conditions. The basis properties of the new approach are investigated and both stability and error 
   estimates of the proposed numerical scheme are deeply analyzed in the $L^{\infty}(0,T;L^{2})$-norm. The theory indicates that the method is
   unconditionally stable with convergence of order $O(k^{2-\frac{\lambda}{2}}+h^{4})$, where $k$ and $h$ are time step and mesh size, respectively,
   and $\lambda\in(0,1)$. This result suggests that the two-level fourth-order technique is more efficient than a large class of numerical techniques
   widely studied in the literature for the considered problem. Some numerical evidences are provided to verify the unconditional stability and
   convergence rate of the proposed algorithm.\\
    \text{\,}\\

   \\
   {\bf AMS Subject Classification (MSC). 65M12, 65M06}.

  \section{Introduction}\label{sec1}
   In the last few decades, fractional calculus have played an important role in all areas in sciences and engineering \cite{3zzy,14zzy,17zzy}. Most
   recently, the great potential of the fractional partial differential equations (FPDEs) has motivated the development of efficient numerical schemes
   for solving both stationary and evolutionary FPDEs describing nonlinear phenomena in medical, biological, physical, financial, and geological systems.
   For more details, the readers can consult \cite{5jrs,16jrs,29jrs,39jrs}. Although the concepts and the calculus of fractional derivative are few
   centuries old, fractional advection-diffusion equations have received a great interest in recent years and have been used to model a broad range of
   problems in fluid flow, electrostatics, electricity, heat, electrodynamic and sound \cite{3zzy,5zzy}. Owing to the increasing applications, a particular
    attention is given to the analytical and numerical solutions of FPDEs. In the literature the big challenge with such equations is the design of
    efficient and accurate numerical methods and computational cost is the main issue to be considered for any numerical scheme. For classical integer
    order ordinary or partial differential equations (ODEs/PDEs) such as: systems of ODEs, Navier-Stokes equations, mixed Stokes-Darcy model, shallow
    water equations, advection-diffusions problems, convection-diffusion-reaction equations, heat conduction
   \cite{en1,en2,zl1984,en19,de2004,en3,en4,en5,en6,tg13,tg17,en7,en8,tg10,en10}, a wide class of numerical approaches have been deeply
   analyzed: finite difference techniques, two-level MacCormack procedure, spectral methods, full implicit finite difference schemes, two-level factored 
   approaches, compact ADI methods and multi-level finite difference formulations. For more details, we refer the readers to
    \cite{tg8,gv1983,en11,en12,en9,tg31,tg14,en13,9mc,en13,en14,22mc,en15,en16,23mc,en17,38mc,en18} and references therein. For
    fractional partial differential equations, a variety of numerical methods have been developed and their stability and accuracy have been widely 
   discussed. Such techniques for solving FPSEs considered implicit meshless schemes based on radial basis functions, finite element formulations, 
   spectral methods, finite difference procedures, meshless methods \cite{8zzy,11zzy,9zzy,15zzy,16zzy,18zzy,20zzy}. This work deals with a
   two-level fourth-order scheme applied to the time-fractional convection-diffusion-reaction equation with variable coefficients. Though the
    proposed approach is slightly more accurate (convergence order $O(k^{2-\frac{\lambda}{2}})$) than a large class of numerical methods widely studied
    in literature \cite{18mc,31mc,38mc,11mc}, it is also less time computing. The main motivation of this paper are the following: (a) we introduce
    a new parameter $\alpha$ (where $\alpha=1-\lambda$) in the discrete time when approximating the time-fractional derivative ($cD_{0t}^{\lambda}u$) at
    the grid point $(x_{j},t_{i+\alpha})$; (b) the considered problem has variable coefficients for convection, diffusion and reaction terms; (c) a new
    two-level method of order $O(k^{2-\frac{\lambda}{2}}+h^{4})$ is developed and (d) both stability and convergence rate of the proposed algorithm are 
   deeply analyzed in the $L^{\infty}(0,T;L^{2})$-norm (also $L^{2}(0,T;L^{2})$-norm for numerical examples) by introducing generalized sequences with 
   positive increasing terms. The use of generalized sequences instead of the ordinary ones in the study of the stability and accuracy of a numerical 
   method is an innovation since in our knowledge, there is not available works in the literature that use the generalized sequences in the analysis of 
   both stability and convergence.\\

     In this work, we propose a modified time discrete form combined with finite difference techniques for the time-fractional 
   convection-diffusion-reaction equation involving Caputo fractional derivative and describing by the following initial-boundary value problem

     \begin{equation}\label{1e}
      cD_{0t}^{\lambda}u(x,t)-q(t)u_{xx}+p(t)u_{x}+g(x,t)u(x,t)=s(x,t),
     \end{equation}
     where the coefficients $q$ and $p$ are functions depending on the time variable $t$, whereas $g$ and $s$ are functions that depend on both variables
      $x$ and $t$. Here $p(t)\geq0$, $g(x,t)\geq0,$ and $q(t)\geq\gamma>0$, where $\gamma$ is a constant. In \cite{18mc}, the Caputo fractional derivative
      $cD_{0t}^{\lambda}u$ $(0<\lambda<1)$ of the function $u$ is defined as
     \begin{equation*}
       cD_{0t}^{\lambda}u(x,t)=\frac{1}{\Gamma(1-\lambda)}\int_{0}^{t}\frac{\partial_{\tau}u(x,\tau)}{(t-\tau)^{\lambda}}d\tau,
     \end{equation*}
      where $\partial_{\tau}u$ denotes $\frac{\partial u}{\partial \tau}$. The initial condition for equation $(\ref{1e})$ is given by
     \begin{equation}\label{2e}
      u(x,0)=\psi_{1}(x),\text{\,\,\,\,on\,\,\,\,}(0,L_{1}),
     \end{equation}
      and the boundary conditions are defined by
     \begin{equation}\label{3e}
     u(0,t)=\psi_{2}(t),\text{\,\,\,\,and\,\,\,\,}u(L_{1},t)=\psi_{2}(t),\text{\,\,\,\,on\,\,\,\,}(0,T).
     \end{equation}
     Furthermore, we assume that the exact solution of problem $(\ref{1e})$-$(\ref{3e})$ is sufficiently smooth for the discretization and error
     estimates. We recall that the aim of this paper is to construct an efficient solution to the time-fractional equation $(\ref{1e})$ subjects to
     suitable initial and boundary conditions given by relations $(\ref{2e})$ and $(\ref{3e})$, respectively. More specifically, the attention
     is focused on the following three items:

     \begin{description}
      \item[(i1)] detailed description of a two-level fourth-order approach for time-fractional convection-diffusion-reaction equation $(\ref{1e})$ 
      with appropriate initial-boundary conditions $(\ref{2e})$-$(\ref{3e})$,
      \item[(i2)] analysis of the unconditional stability and error estimates of the proposed approach,
      \item[(i3)] a broad range of numerical examples that confirm the theoretical study.
     \end{description}

     In the following we proceed as follows. Section $\ref{sec2}$ considers a full description of the two-level fourth scheme for solving the
     given initial-boundary value problem $(\ref{1e})$-$(\ref{3e})$. In Section $\ref{sec3}$, we analyze both stability and error estimates of the
     new algorithm using the $L^{\infty}(0,T;L^{2})$-norm. A wide set of numerical evidences that confirm the theoretical results are presented
     and discussed in Section $\ref{sec4}$. Finally, we draw in Section $\ref{sec5}$ the general conclusion and provide our future works.

    \section{Full description of a two-level time-fractional method}\label{sec2}
    This section deals with a detailed description of the two-level time-fractional scheme for solving the initial-boundary value problem 
    $(\ref{1e})$-$(\ref{3e})$. The proposed algorithm is a two-step implicit method which approximates the time-fractional Caputo derivative
    using forward difference in each step and the convection and diffusion terms are approximated by the use of central difference. Since the aim of 
    this section is to develop the method, without loss of generality we should use a constant time step $k=\Delta t$ and space step $h=\Delta x$. 
    Let $\lambda$ be a
    real number satisfying $0<\lambda<1,$ $L_{1}>0$ and $T>0$ be the space interval length and time interval length, respectively. Suppose $N$ and $M$ 
    be two positive integers. Set $x_{j}=jh,$ $t_{i}=ik$ and let the superscript denoting the time level and space level of the approximation, where
     $k=\frac{T}{N}$ and $h=\frac{L_{1}}{M}$. Consider the uniform mesh space $\mathcal{Y}_{kh}=\{(x_{j},t_{i}),\text{\,}0\leq i\leq N;\text{\,}0\leq j
    \leq M\}$. Furthermore, we introduce the positive parameter $\alpha=1-\lambda$. For a function $u\in\mathcal{C}^{3}(0,T;H^{6}(0,L_{1}))$, the Caputo
     fractional derivatives of order $\lambda$ of the function $u$ at the grid points $(x_{j},t_{i+\frac{1}{2}+\alpha})$ and $(x_{j},t_{i+1+\alpha})$ are
     given by\\

      \begin{equation}\label{1}
      cD_{0t}^{\lambda}u(x_{j},t_{i+\frac{1}{2}+\alpha})=\frac{1}{\Gamma(1-\lambda)}\int_{0}^{t_{i+\frac{1}{2}+\alpha}}u_{\tau,j}(\tau)
      (t_{i+\frac{1}{2}+\alpha}-\tau)^{-\lambda}d\tau,
      \end{equation}
      and
      \begin{equation}\label{2}
      cD_{0t}^{\lambda}u(x_{j},t_{i+1+\alpha})=\frac{1}{\Gamma(1-\lambda)}\int_{0}^{t_{i+1+\alpha}}u_{\tau,j}(\tau)(t_{i+1+\alpha}-\tau)^{-\lambda}d\tau.
      \end{equation}

      For the convenient of writing, we should set in the following the notations
       \begin{equation*}
      cD_{0t}^{\lambda}u(x_{j},t_{i+\frac{1}{2}+\alpha})=cD_{0t}^{\lambda}u_{j}^{i+\frac{1}{2}+\alpha}\text{\,\,\,\,\,and\,\,\,\,\,}
      cD_{0t}^{\lambda}u(x_{j},t_{i+1+\alpha})=cD_{0t}^{\lambda}u_{j}^{i+1+\alpha},\text{\,\,\,\,for\,\,\,\,}i\geq0.
      \end{equation*}

       Using this, equation $(\ref{1})$ can be rewritten as
       \begin{equation*}
        cD_{0t}^{\lambda}u_{j}^{i+\frac{1}{2}+\alpha}=\frac{1}{\Gamma(1-\lambda)}\int_{0}^{t_{i+\frac{1}{2}+\alpha}}u_{\tau,j}(\tau)
      (t_{i+\frac{1}{2}+\alpha}-\tau)^{-\lambda}d\tau=\frac{1}{\Gamma(1-\lambda)}\underset{l=0}{\overset{i-1}\sum}\int_{t_{l+\frac{1}{2}}}^{t_{l+\frac{3}{2}
      }}u_{\tau,j}(\tau)(t_{i+\frac{1}{2}+\alpha}-\tau)^{-\lambda}d\tau
       \end{equation*}
       \begin{equation}\label{3}
      +\frac{1}{\Gamma(1-\lambda)}\left[\int_{0}^{t_{\frac{1}{2}}}u_{\tau,j}(\tau)(t_{i+\frac{1}{2}+\alpha}-\tau)^{-\lambda}d\tau
      +\int_{t_{i+\frac{1}{2}}}^{t_{i+\frac{1}{2}+\alpha}}u_{\tau,j}(\tau)(t_{i+\frac{1}{2}+\alpha}-\tau)^{-\lambda}d\tau\right].
       \end{equation}
       Furthermore, let introduce the following operators
       \begin{equation}\label{2a}
        \delta_{t} u_{j}^{i}=\frac{u_{j}^{i+\frac{1}{2}}-u_{j}^{i}}{k/2};\text{\,}\delta_{t}u_{j}^{i+\frac{1}{2}}=
        \frac{u_{j}^{i+1}-u_{j}^{i+\frac{1}{2}}}{k/2};\text{\,}\delta_{x}u_{j+\frac{1}{2}}^{i}=\frac{u_{j+1}^{i}
        -u_{j}^{i}}{h};\text{\,}\delta_{x}^{2}u_{j}^{i}=\frac{u_{j+1}^{i}-2u_{j}^{i}+u_{j-1}^{i}}{h^{2}}.
       \end{equation}
       We consider the following norms and inner product
       \begin{equation}\label{2aa}
        \|u^{i}\|_{L^{2}}=\left(h\underset{j=1}{\overset{M-1}\sum}|u_{j}^{i}|^{2}\right)^{\frac{1}{2}},\text{\,\,\,}
        \||u|\|_{\mathcal{C}^{6,3}_{D}}=\underset{0\leq r\leq 6}{\underset{0\leq s\leq 3}\max}\left\{\underset{0\leq i\leq
        N}{\max}\left\|\frac{\partial^{r+s}u}{\partial x^{r}\partial t^{s}}(t_{i})\right\|_{L^{2}}\right\}\text{\,\,\,and\,\,\,}(u^{i},v^{i})=
        h\underset{j=1}{\overset{M-1}\sum}u_{j}^{i}v_{j}^{i},
       \end{equation}
       where $D=[0,L_{1}]\times[0,T]$, $|\cdot|$ denotes the $\mathbb{C}$-norm. The spaces $L^{2}(0,L_{1})$ and $\mathcal{C}^{6,3}_{D}$ are endowed with
        the norms $\|\cdot\|_{L^{2}}$ and $\||\cdot|\|_{\mathcal{C}^{6,3}_{D}}$, respectively, whereas the Hilbert space $L^{2}(0,L_{1})$ is equipped with
        the inner product $(\cdot,\cdot)$.\\

       Let $P_{2}^{u_{j},l}$ be the second-order polynomial interpolating $u_{j}(t)$ at the points $(t_{l+\frac{1}{2}},u_{j}^{l+\frac{1}{2}}),$
       $(t_{l+1},u_{j}^{l+1})$ and $(t_{l+\frac{3}{2}},u_{j}^{l+\frac{3}{2}})$. Simple calculations give
       \begin{equation}\label{4}
        P_{2}^{u_{j},l}(t)=\frac{2}{k^{2}}\left[(t-t_{l+1})(t-t_{l+\frac{3}{2}})u_{j}^{l+\frac{1}{2}}-2(t-t_{l+\frac{1}{2}})(t-t_{l+\frac{3}{2}})u_{j}^{l+1}
        +(t-t_{l+\frac{1}{2}})(t-t_{l+1})u_{j}^{l+\frac{3}{2}}\right],
       \end{equation}
       and the corresponding error is defined as
       \begin{equation}\label{4a}
        E^{l}_{j}(t)=u_{j}(t)-P_{2}^{u_{j},l}(t)=\frac{1}{6}(t-t_{l+\frac{1}{2}})(t-t_{l+1})(t-t_{l+\frac{3}{2}})u_{3t,j}(t^{\epsilon}_{l}),
       \end{equation}
       where $t^{\epsilon}_{l}$ is between the minimum and maximum of $t_{l+\frac{1}{2}}$, $t_{l+1}$, $t_{l+\frac{3}{2}}$ and $t$. The derivative of
       $P_{2}^{u_{j},l}$ with respect to the time-variable provides
       \begin{equation}\label{5}
        P_{2,t}^{u_{j},l}(t)=\frac{2}{k^{2}}\left[2(u_{j}^{l+\frac{1}{2}}-2u_{j}^{l+1}+u_{j}^{l+\frac{3}{2}})t-(t_{l+1}+t_{l+\frac{3}{2}})u_{j}^{l+\frac{1}{2}}+
        2(t_{l+\frac{1}{2}}+t_{l+\frac{3}{2}})u_{j}^{l+1}-(t_{l+\frac{1}{2}}+t_{l+1})u_{j}^{l+\frac{3}{2}}\right].
       \end{equation}
       Setting
       \begin{equation*}
       I_{i}^{(\alpha)}=\frac{1}{\Gamma(1-\lambda)}\left[\int_{0}^{t_{\frac{1}{2}}}u_{\tau,j}(\tau)(t_{i+\frac{1}{2}+\alpha}-\tau)^{-\lambda}d\tau
      +\int_{t_{i+\frac{1}{2}}}^{t_{i+\frac{1}{2}+\alpha}}u_{\tau,j}(\tau)(t_{i+\frac{1}{2}+\alpha}-\tau)^{-\lambda}d\tau\right]+
       \end{equation*}
       \begin{equation}\label{6}
      \frac{1}{\Gamma(1-\lambda)}\underset{l=0}{\overset{i-1}\sum}\int_{t_{l+\frac{1}{2}}}^{t_{l+\frac{3}{2}}}E^{l}_{\tau,j}(\tau)(t_{i+\frac{1}{2}
      +\alpha}-\tau)^{-\lambda}d\tau.
       \end{equation}
       Since $(1-\lambda)\Gamma(1-\lambda)=\Gamma(2-\lambda)$, plugging equations $(\ref{3})$-$(\ref{6})$, direct computations yield
       \begin{equation*}
        cD_{0t}^{\lambda}u_{j}^{i+\frac{1}{2}+\alpha}=\frac{2}{k^{2}\Gamma(2-\lambda)}\underset{l=0}{\overset{i-1}\sum}\left[
        -2\left(u_{j}^{l+\frac{1}{2}}-2u_{j}^{l+1}+u_{j}^{l+\frac{3}{2}}\right)\left(\tau(t_{i+\frac{1}{2}+\alpha}-\tau)^{1-\lambda}+
        \frac{1}{2-\lambda}(t_{i+\frac{1}{2}+\alpha}-\tau)^{2-\lambda}\right)\right.
       \end{equation*}
       \begin{equation*}
        \left.+\left((t_{l+1}+t_{l+\frac{3}{2}})u_{j}^{l+\frac{1}{2}}-2(t_{l+\frac{1}{2}}+t_{l+\frac{3}{2}})u_{j}^{l+1}+(t_{l+\frac{1}{2}}+t_{l+1})
        u_{j}^{l+\frac{3}{2}}\right)(t_{i+\frac{1}{2}+\alpha}-\tau)^{1-\lambda}\right]_{t_{l+\frac{1}{2}}}^{t_{l+\frac{3}{2}}}+I_{i}^{(\alpha)}
       \end{equation*}
       \begin{equation*}
        =\frac{2k^{2-\lambda}}{k^{2}\Gamma(2-\lambda)}\underset{l=0}{\overset{i-1}\sum}\left[2\left(u_{j}^{l+\frac{1}{2}}-2u_{j}^{l+1}
        +u_{j}^{l+\frac{3}{2}}\right)\left((l+\frac{1}{2})(i+\alpha-l)^{1-\lambda}+\frac{1}{2-\lambda}(i+\alpha-l)^{2-\lambda}\right.\right.
       \end{equation*}
       \begin{equation*}
        \left.-(l+\frac{3}{2})(i+\alpha-l-1)^{1-\lambda}-\frac{1}{2-\lambda}(i+\alpha-l-1)^{2-\lambda}\right)+
       \end{equation*}
       \begin{equation*}
        \left.\left((2l+\frac{5}{2})u_{j}^{l+\frac{1}{2}}-2(2l+2)u_{j}^{l+1}+(2l+\frac{3}{2})
        u_{j}^{l+\frac{3}{2}}\right)\left((i+\alpha-l-1)^{1-\lambda}-(i+\alpha-l)^{1-\lambda}\right)\right]+I_{i}^{(\alpha)}.
       \end{equation*}
       \begin{equation*}
        =\frac{k^{1-\lambda}}{\Gamma(2-\lambda)}\underset{l=0}{\overset{i-1}\sum}\left[\left((i+\alpha-l)^{1-\lambda}-
        (i+\alpha-l-1)^{1-\lambda}\right)\delta_{t}u_{j}^{l+\frac{1}{2}}+\left[\frac{2}{2-\lambda}\left((i+\alpha-l)^{2-\lambda}-
        (i+\alpha-l-1)^{2-\lambda}\right)\right.\right.
       \end{equation*}
       \begin{equation*}
        \left.\left.-\frac{1}{2}(i+\alpha-l)^{1-\lambda}-\frac{3}{2}(i+\alpha-l-1)^{1-\lambda}\right](\delta_{t}u_{j}^{l+1}-
        \delta_{t}u_{j}^{l+\frac{1}{2}})\right]+I_{i}^{(\alpha)}
       \end{equation*}
       \begin{equation}\label{10}
        =\frac{k^{1-\lambda}}{\Gamma(2-\lambda)}\underset{l=0}{\overset{i-1}\sum}\left[\widetilde{f}^{\lambda\alpha}_{i+\frac{1}{2},l}
        \delta_{t}u_{j}^{l+1}+(\widetilde{d}^{\lambda\alpha}_{i+\frac{1}{2},l}-\widetilde{f}^{\lambda\alpha}_{i+\frac{1}{2},l})
        \delta_{t}u_{j}^{l+\frac{1}{2}}\right]+I_{i}^{(\alpha)},
       \end{equation}
       where
       \begin{equation}\label{11}
        \widetilde{d}^{\lambda\alpha}_{i+\frac{1}{2},l}=(i+\alpha-l)^{1-\lambda}-(i+\alpha-l-1)^{1-\lambda},
       \end{equation}
       and
       \begin{equation}\label{12}
        \widetilde{f}^{\lambda\alpha}_{i+\frac{1}{2},l}=\frac{2}{2-\lambda}\left[(i+\alpha-l)^{2-\lambda}-
        (i+\alpha-l-1)^{2-\lambda}\right]-\frac{1}{2}[(i+\alpha-l)^{1-\lambda}+3(i+\alpha-l-1)^{1-\lambda}],
       \end{equation}
       for $0\leq l\leq i-1,$ $i\geq1.$ In addition, we should find an explicit expression of the term $cD_{0t}^{\lambda}u_{j}^{\frac{1}{2}+\alpha}$.
        \begin{equation*}
        cD_{0t}^{\lambda}u_{j}^{\frac{1}{2}+\alpha}=\frac{1}{\Gamma(1-\lambda)}\int_{0}^{t_{\frac{1}{2}+\alpha}}u_{\tau,j}(\tau)
      (t_{\frac{1}{2}+\alpha}-\tau)^{-\lambda}d\tau=\frac{1}{\Gamma(1-\lambda)}\left[\int_{0}^{t_{\frac{1}{2}+\alpha}}\delta_{t}u_{j}(t_{0})
      (t_{\frac{1}{2}+\alpha}-\tau)^{-\lambda}d\tau+\right.
       \end{equation*}
       \begin{equation}\label{14}
      \left.\int_{0}^{t_{\frac{1}{2}+\alpha}}(u_{\tau,j}(\tau)-\delta_{t}u_{j}^{0})(t_{\frac{1}{2}
      +\alpha}-\tau)^{-\lambda}d\tau\right]=c\Delta_{0t}^{\lambda}u_{j}^{\frac{1}{2}+\alpha}+\frac{1}{\Gamma(1-\lambda)}\int_{0}^{t_{\frac{1}{2}
      +\alpha}}(u_{\tau,j}(\tau)-P_{1,\tau}^{u_{j},0})(t_{\frac{1}{2}+\alpha}-\tau)^{-\lambda}d\tau,
       \end{equation}
       where
        \begin{equation}\label{15}
      c\Delta_{0t}^{\lambda}u_{j}^{\frac{1}{2}+\alpha}=\frac{k^{1-\lambda}(\frac{1}{2}+\alpha)^{1-\lambda}}{\Gamma(2-\lambda)}\delta_{t}u_{j}^{0}.
       \end{equation}
       In a similar manner, the term $\frac{1}{\Gamma(1-\lambda)}\int_{t_{i+\frac{1}{2}}}^{t_{i+\frac{1}{2}+\alpha}}u_{\tau,j}(\tau)
      (t_{i+\frac{1}{2}+\alpha}-\tau)^{-\lambda}d\tau$ can be rewritten as
      \begin{equation}\label{16}
      \frac{1}{\Gamma(1-\lambda)}\int_{t_{i+\frac{1}{2}}}^{t_{i+\frac{1}{2}+\alpha}}u_{\tau,j}(\tau)
      (t_{i+\frac{1}{2}+\alpha}-\tau)^{-\lambda}d\tau=\frac{k^{1-\lambda}\alpha^{1-\lambda}}{\Gamma(2-\lambda)}\delta_{t}u_{j}^{i}+
      \frac{1}{\Gamma(1-\lambda)}\int_{t_{i+\frac{1}{2}}}^{t_{i+\frac{1}{2}+\alpha}}\frac{u_{\tau,j}(\tau)-\delta_{t}u_{j}^{i}}{(t_{i+\frac{1}{2}+\alpha}
      -\tau)^{\lambda}}d\tau.
      \end{equation}
      Furthermore, it is easy to observe that
      \begin{equation}\label{17}
      \frac{1}{\Gamma(1-\lambda)}\int_{0}^{t_{\frac{1}{2}}}\frac{u_{\tau,j}(\tau)}{(t_{i+\frac{1}{2}+\alpha}-\tau)^{\lambda}}d\tau=
      \frac{\delta_{t}u_{j}^{0}}{\Gamma(2-\lambda)}[(i+\frac{1}{2}+\alpha)^{1-\lambda}-(i+\alpha)^{1-\lambda}]+\frac{1}{\Gamma(1-\lambda)}
      \int_{0}^{t_{\frac{1}{2}}}\frac{u_{\tau,j}(\tau)-\delta_{t}u_{j}^{0}}{(t_{i+\frac{1}{2}+\alpha}-\tau)^{\lambda}}d\tau.
      \end{equation}
      Plugging equations $(\ref{4})$, $(\ref{10})$, $(\ref{14})$ and $(\ref{15})$, it is not hard to observe that
      \begin{equation*}
      cD_{0t}^{\lambda}u_{j}^{\frac{1}{2}+\alpha}=c\Delta_{0t}^{\lambda}u_{j}^{\frac{1}{2}+\alpha}+\frac{1}{\Gamma(1-\lambda)}\int_{0}^{t_{\frac{1}{2}
      +\alpha}}(u_{\tau,j}(\tau)-P_{1,\tau}^{u_{j},0})(t_{\frac{1}{2}+\alpha}-\tau)^{-\lambda}d\tau=
      \end{equation*}
       \begin{equation}\label{18}
       f_{\frac{1}{2},0}^{\lambda\alpha}\delta_{t}u_{j}^{0}+\frac{1}{\Gamma(1-\lambda)}\int_{0}^{t_{\frac{1}{2}
      +\alpha}}(u_{\tau,j}(\tau)-P_{1,\tau}^{u_{j},0})(t_{\frac{1}{2}+\alpha}-\tau)^{-\lambda}d\tau,\text{\,\,\,\,\,if\,\,\,\,\,}i=0,
       \end{equation}
       and
       \begin{equation*}
        cD_{0t}^{\lambda}u_{j}^{i+\frac{1}{2}+\alpha}=c\Delta_{0t}^{\lambda}u_{j}^{i+\frac{1}{2}+\alpha}+\frac{1}{\Gamma(1-\lambda)}\left\{
        \underset{l=0}{\overset{i-1}\sum}\int_{t_{l+\frac{1}{2}}}^{t_{l+\frac{3}{2}}}E^{l}_{\tau,j}(\tau)(t_{i+\frac{1}{2}
      +\alpha}-\tau)^{-\lambda}d\tau+\int_{0}^{t_{\frac{1}{2}}}\frac{u_{\tau,j}(\tau)-\delta_{t}u_{j}^{0}}{(t_{i+\frac{1}{2}
      +\alpha}-\tau)^{\lambda}}d\tau\right.
       \end{equation*}
       \begin{equation}\label{19}
      \left.+\int_{t_{i+\frac{1}{2}}}^{t_{i+\frac{1}{2}+\alpha}}(u_{\tau,j}(\tau)-\delta_{t}u_{j}^{i})(t_{i+\frac{1}{2}+\alpha}
      -\tau)^{-\lambda}d\tau\right\},\text{\,\,\,\,\,for\,\,\,\,\,}i\geq1,
       \end{equation}
       where
       \begin{equation}\label{20}
       c\Delta_{0t}^{\lambda}u_{j}^{\frac{1}{2}+\alpha}=f_{\frac{1}{2},0}^{\lambda\alpha}\delta_{t}u_{j}^{0},\text{\,\,\,\,\,if\,\,\,\,\,}i=0,
       \end{equation}
       \begin{equation}\label{21}
       c\Delta_{0t}^{\lambda}u_{j}^{i+\frac{1}{2}+\alpha}=\text{\,\.{f}}_{i+\frac{1}{2},0}^{\lambda\alpha}\delta_{t}u_{j}^{0}+
       \underset{l=0}{\overset{i-1}\sum}\left[f^{\lambda\alpha}_{i+\frac{1}{2},l}\delta_{t}u_{j}^{l+1}+(d^{\lambda\alpha}_{i+\frac{1}{2},l}
        -f^{\lambda\alpha}_{i+\frac{1}{2},l})\delta_{t}u_{j}^{l+\frac{1}{2}}\right]+f^{\lambda\alpha}_{i+\frac{1}{2},i}\delta_{t}u_{j}^{i}
        ,\text{\,\,\,\,\,if\,\,\,\,\,}i\geq1,
       \end{equation}
       where
       \begin{equation*}
        f^{\lambda\alpha}_{\frac{1}{2},0}=\frac{k^{1-\lambda}}{\Gamma(2-\lambda)}\widetilde{f}^{\lambda\alpha}_{\frac{1}{2},0};\text{\,\,}
       \text{\,\.{f}}^{\lambda\alpha}_{i+\frac{1}{2},0}=\frac{k^{1-\lambda}}{\Gamma(2-\lambda)}\widetilde{\text{\,\.{f}}}^{\lambda\alpha}_{i+\frac{1}{2},0};
       \text{\,\,}
       f^{\lambda\alpha}_{i+\frac{1}{2},i}=\frac{k^{1-\lambda}}{\Gamma(2-\lambda)}\widetilde{f}^{\lambda\alpha}_{i+\frac{1}{2},i};\text{\,\,}
       f^{\lambda\alpha}_{i+\frac{1}{2},l}=\frac{k^{1-\lambda}}{\Gamma(2-\lambda)}\widetilde{f}^{\lambda\alpha}_{i+\frac{1}{2},l};\text{\,\,}
       \end{equation*}
       \begin{equation}\label{22}
        d^{\lambda\alpha}_{i+\frac{1}{2},l}=\frac{k^{1-\lambda}}{\Gamma(2-\lambda)}\widetilde{d}^{\lambda\alpha}_{i+\frac{1}{2},l};
        \text{\,\,\,\,for\,\,\,\,}0\leq l\leq i-1,
       \end{equation}
       with
       \begin{equation}\label{22a}
        \widetilde{f}^{\lambda\alpha}_{i+\frac{1}{2},i}=\alpha^{1-\lambda};\text{\,\,}
        \widetilde{f}^{\lambda\alpha}_{\frac{1}{2},0}=(\frac{1}{2}+\alpha)^{1-\lambda};\text{\,\,}
        \widetilde{\text{\,\.{f}}}^{\lambda\alpha}_{i+\frac{1}{2},0}=(i+\frac{1}{2}+\alpha)^{1-\lambda}-(i+\alpha)^{1-\lambda},
       \end{equation}
       where $\widetilde{d}^{\lambda\alpha}_{i+\frac{1}{2},l}$ and $\widetilde{f}^{\lambda\alpha}_{i+\frac{1}{2},l}$, for $0\leq l\leq i-1$, are given by
       relations $(\ref{11})$ and $(\ref{12})$, respectively.\\

       To construct the first-level of the proposed approach, we should find a similar approximation for the function $q(t)u_{2x}-p(t)u_{x}-g(x,t)u$ at the
        point $(x_{j},t_{i+\frac{1}{2}+\alpha})$. Firstly,
        \begin{equation}\label{23}
        [q(t)u_{2x}-p(t)u_{x}-g(x,t)u+s(x,t)](x_{j},t_{i+\frac{1}{2}+\alpha})=q^{i+\frac{1}{2}+\alpha}u_{2x,j}^{i+\frac{1}{2}+\alpha}-p^{i+\frac{1}{2}+\alpha}
        u_{x,j}^{i+\frac{1}{2}+\alpha}-g_{j}^{i+\frac{1}{2}+\alpha}u_{j}^{i+\frac{1}{2}+\alpha}+s_{j}^{i+\frac{1}{2}+\alpha}.
       \end{equation}
       Expanding the Taylor series for $u$ about the point $(x_{j},t_{i+\frac{1}{2}+\alpha})$ with step size $h$ using both forward and backward
        representations provides
       \begin{equation}\label{24}
       u_{j+2}^{i+\frac{1}{2}+\alpha}=u_{j}^{i+\frac{1}{2}+\alpha}+2hu_{x,j}^{i+\frac{1}{2}+\alpha}+2h^{2}u_{2x,j}^{i+\frac{1}{2}+\alpha}+
       \frac{4h^{3}}{3}u_{3x,j}^{i+\frac{1}{2}+\alpha}+\frac{2h^{4}}{3}u_{4x,j}^{i+\frac{1}{2}+\alpha}+\frac{4h^{5}}{15}u_{5x,j}^{i+\frac{1}{2}+\alpha}+
       \frac{4h^{6}}{45}u_{6x}^{i+\frac{1}{2}+\alpha}(\epsilon_{j}^{7}),
       \end{equation}
       \begin{equation}\label{25}
       u_{j-2}^{i+\frac{1}{2}+\alpha}=u_{j}^{i+\frac{1}{2}+\alpha}-2hu_{x,j}^{i+\frac{1}{2}+\alpha}+2h^{2}u_{2x,j}^{i+\frac{1}{2}+\alpha}-
       \frac{4h^{3}}{3}u_{3x,j}^{i+\frac{1}{2}+\alpha}+\frac{2h^{4}}{3}u_{4x,j}^{i+\frac{1}{2}+\alpha}-\frac{4h^{5}}{15}u_{5x,j}^{i+\frac{1}{2}+\alpha}+
       \frac{4h^{6}}{45}u_{6x}^{i+\frac{1}{2}+\alpha}(\epsilon_{j}^{6}),
       \end{equation}
       \begin{equation}\label{26}
       u_{j+1}^{i+\frac{1}{2}+\alpha}=u_{j}^{i+\frac{1}{2}+\alpha}+hu_{x,j}^{i+\frac{1}{2}+\alpha}2+\frac{h^{2}}{2}u_{2x,j}^{i+\frac{1}{2}+\alpha}+
       \frac{h^{3}}{6}u_{3x,j}^{i+\frac{1}{2}+\alpha}+\frac{h^{4}}{24}u_{4x,j}^{i+\frac{1}{2}+\alpha}+\frac{h^{5}}{120}u_{5x,j}^{i+\frac{1}{2}+\alpha}+
       \frac{h^{6}}{720}u_{6x}^{i+\frac{1}{2}+\alpha}(\epsilon_{j}^{5}),
       \end{equation}
       \begin{equation}\label{27}
       u_{j-1}^{i+\frac{1}{2}+\alpha}=u_{j}^{i+\frac{1}{2}+\alpha}-hu_{x,j}^{i+\frac{1}{2}+\alpha}2+\frac{h^{2}}{2}u_{2x,j}^{i+\frac{1}{2}+\alpha}-
       \frac{h^{3}}{6}u_{3x,j}^{i+\frac{1}{2}+\alpha}+\frac{h^{4}}{24}u_{4x,j}^{i+\frac{1}{2}+\alpha}-\frac{h^{5}}{120}u_{5x,j}^{i+\frac{1}{2}+\alpha}+
       \frac{h^{6}}{720}u_{6x}^{i+\frac{1}{2}+\alpha}(\epsilon_{j}^{4}),
       \end{equation}
       where
       \begin{equation}\label{28}
       \epsilon_{j}^{4}\in(x_{j-1},x_{j}),\text{\,\,}\epsilon_{j}^{5}\in(x_{j},x_{j+1}),\text{\,\,}\epsilon_{j}^{6}\in(x_{j-2},x_{j}),\text{\,\,}
        \epsilon_{j}^{7}\in(x_{j},x_{j+2}).
       \end{equation}
       Combining equations $(\ref{24})$ and $(\ref{25})$, direct computations give
       \begin{equation}\label{29}
       u_{j+2}^{i+\frac{1}{2}+\alpha}+u_{j-2}^{i+\frac{1}{2}+\alpha}=2u_{j}^{i+\frac{1}{2}+\alpha}+4h^{2}u_{2x,j}^{i+\frac{1}{2}+\alpha}+
       \frac{4h^{4}}{3}u_{4x,j}^{i+\frac{1}{2}+\alpha}+\frac{4h^{6}}{45}[u_{6x}^{i+\frac{1}{2}+\alpha}(\epsilon_{j}^{7})+u_{6x}^{i+\frac{1}{2}+
       \alpha}(\epsilon_{j}^{6})].
       \end{equation}
       In a similar manner, plugging equations $(\ref{26})$ and $(\ref{27})$ to get
       \begin{equation}\label{30}
       u_{j+1}^{i+\frac{1}{2}+\alpha}+u_{j-1}^{i+\frac{1}{2}+\alpha}=2u_{j}^{i+\frac{1}{2}+\alpha}+h^{2}u_{2x,j}^{i+\frac{1}{2}+\alpha}+
       \frac{h^{4}}{12}u_{4x,j}^{i+\frac{1}{2}+\alpha}+\frac{h^{6}}{360}[u_{6x}^{i+\frac{1}{2}+\alpha}(\epsilon_{j}^{4})+u_{6x}^{i+\frac{1}{2}+
       \alpha}(\epsilon_{j}^{5})].
       \end{equation}
       Multiplying both sides of equations $(\ref{30})$ and $(\ref{29})$ by $12$ and $\frac{-3}{4}$, respectively, and adding the resulting equations
       to obtain
       \begin{equation*}
       12(u_{j+1}^{i+\frac{1}{2}+\alpha}+u_{j-1}^{i+\frac{1}{2}+\alpha})-\frac{3}{4}(u_{j+2}^{i+\frac{1}{2}+\alpha}+u_{j-2}^{i+\frac{1}{2}
       +\alpha})=\frac{45}{2}u_{j}^{i+\frac{1}{2}+\alpha}+9h^{2}u_{2x,j}^{i+\frac{1}{2}+\alpha}+
       \frac{h^{6}}{30}[u_{6x}^{i+\frac{1}{2}+\alpha}(\epsilon_{j}^{4})+u_{6x}^{i+\frac{1}{2}+\alpha}(\epsilon_{j}^{5})-
       \end{equation*}
       \begin{equation*}
       18u_{6x}^{i+\frac{1}{2}+\alpha}(\epsilon_{j}^{6})-18u_{6x}^{i+\frac{1}{2}+\alpha}(\epsilon_{j}^{7})].
       \end{equation*}
        Solving this equation for $u_{2x,j}^{i+\frac{1}{2}+\alpha}$, we obtain
        \begin{equation*}
       u_{2x,j}^{i+\frac{1}{2}+\alpha}=\frac{1}{12h^{2}}\left[-u_{j+2}^{i+\frac{1}{2}+\alpha}+16u_{j+1}^{i+\frac{1}{2}+\alpha}-
      30u_{j}^{i+\frac{1}{2}+\alpha}+16u_{j-1}^{i+\frac{1}{2}+\alpha}-u_{j-2}^{i+\frac{1}{2}+\alpha}\right]+
       \end{equation*}
        \begin{equation}\label{31}
        \frac{h^{4}}{270}\left[18u_{6x}^{i+\frac{1}{2}+\alpha}(\epsilon_{j}^{6})+18u_{6x}^{i+\frac{1}{2}+\alpha}(\epsilon_{j}^{7})-
      u_{6x}^{i+\frac{1}{2}+\alpha}(\epsilon_{j}^{4})-u_{6x}^{i+\frac{1}{2}+\alpha}(\epsilon_{j}^{5})\right].
      \end{equation}
      In a similar way, one easily shows that
       \begin{equation*}
       u_{x,j}^{i+\frac{1}{2}+\alpha}=\frac{1}{12h}\left[-u_{j+2}^{i+\frac{1}{2}+\alpha}+8u_{j+1}^{i+\frac{1}{2}+\alpha}-
      8u_{j-1}^{i+\frac{1}{2}+\alpha}+u_{j-2}^{i+\frac{1}{2}+\alpha}\right]+
       \end{equation*}
        \begin{equation}\label{32}
        \frac{h^{4}}{180}\left[4u_{5x}^{i+\frac{1}{2}+\alpha}(\epsilon_{j}^{10})+4u_{5x}^{i+\frac{1}{2}+\alpha}(\epsilon_{j}^{11})-
      u_{5x}^{i+\frac{1}{2}+\alpha}(\epsilon_{j}^{8})-u_{5x}^{i+\frac{1}{2}+\alpha}(\epsilon_{j}^{9})\right],
      \end{equation}
      where
       \begin{equation}\label{33}
       \epsilon_{j}^{9}\in(x_{j-1},x_{j}),\text{\,\,}\epsilon_{j}^{8}\in(x_{j},x_{j+1}),\text{\,\,}\epsilon_{j}^{11}\in(x_{j-2},x_{j}),\text{\,\,}
      \epsilon_{j}^{10}\in(x_{j},x_{j+2}).
       \end{equation}
       On the other hand, the application of the Taylor series formulation for the function $u$ about the point $(x_{j},t_{i+\frac{1}{2}+\alpha})$ with
      time step $k$ results in
        \begin{equation}\label{28a}
       u_{j}^{i+\frac{1}{2}}=u_{j}^{i+\frac{1}{2}+\alpha}-\alpha ku_{t,j}^{i+\frac{1}{2}+\alpha}+\frac{1}{2}\alpha^{2}k^{2}u_{2t,j}(\epsilon_{i}^{4}),
       \text{\,\,}u_{j}^{i}=u_{j}^{i+\frac{1}{2}+\alpha}-(\frac{1}{2}+\alpha)ku_{t,j}^{i+\frac{1}{2}+\alpha}+\frac{1}{2}(\frac{1}{2}+\alpha)^{2}k^{2}u_{2t,j}
       (\epsilon_{i}^{5}),
       \end{equation}
       where $\epsilon_{i}^{4}\in(t_{i+\frac{1}{2}},t_{i+\frac{1}{2}+\alpha})$ and $\epsilon_{i}^{5}\in(t_{i},t_{i+\frac{1}{2}+\alpha})$. Combining
       both equations in $(\ref{28a})$, simple computations yield
        \begin{equation}\label{28c}
       u_{j}^{i+\frac{1}{2}+\alpha}=(1+2\alpha)u_{j}^{i+\frac{1}{2}}-2\alpha u_{j}^{i}-\alpha(\frac{1}{2}+\alpha)k^{2}H^{i}_{1,j},
       \end{equation}
       where
        \begin{equation}\label{28d}
       H^{i}_{1,j}=\alpha u_{2t,j}(\epsilon_{i}^{4})-(\frac{1}{2}+\alpha)u_{2t,j}(\epsilon_{i}^{5}).
       \end{equation}
       Setting $u_{j}^{\alpha_{i}}=(1+2\alpha)u_{j}^{i+\frac{1}{2}}-2\alpha u_{j}^{i}$, equation $(\ref{28c})$ becomes
       \begin{equation*}
        u_{j}^{i+\frac{1}{2}+\alpha}=u_{j}^{\alpha_{i}}-\alpha(\frac{1}{2}+\alpha)k^{2}H^{i}_{1,j}.
       \end{equation*}
       Utilizing this, equations $(\ref{31})$ and $(\ref{32})$ can be rewritten as
        \begin{equation}\label{40}
         u_{2x,j}^{i+\frac{1}{2}+\alpha}=\frac{1}{12h^{2}}\left[-u_{j+2}^{\alpha_{i}}+16u_{j+1}^{\alpha_{i}}-
      30u_{j}^{\alpha_{i}}+16u_{j-1}^{\alpha_{i}}-u_{j-2}^{\alpha_{i}}\right]+\psi_{1,j}^{i},
      \end{equation}
      and
       \begin{equation}\label{41}
         u_{x,j}^{i+\frac{1}{2}+\alpha}=\frac{1}{12h}\left[-u_{j+2}^{\alpha_{i}}+8u_{j+1}^{\alpha_{i}}-
      8u_{j-1}^{\alpha_{i}}+u_{j-2}^{\alpha_{i}}\right]+\psi_{2,j}^{i},
      \end{equation}
      where
       \begin{equation*}
       \psi_{1,j}^{i}=\frac{h^{4}}{270}\left[18u_{6x}^{i+\frac{1}{2}+\alpha}(\epsilon_{j}^{6})+18u_{6x}^{i+\frac{1}{2}+\alpha}(\epsilon_{j}^{7})-
      u_{6x}^{i+\frac{1}{2}+\alpha}(\epsilon_{j}^{4})-u_{6x}^{i+\frac{1}{2}+\alpha}(\epsilon_{j}^{5})\right]+
       \end{equation*}
        \begin{equation}\label{42}
       \frac{\alpha(\frac{1}{2}+\alpha)k^{2}}{12h^{2}}\left[H_{1,j+2}^{i}-16H_{1,j+1}^{i}+30H_{1,j}^{i}-16H_{1,j-1}^{i}+H_{1,j-2}^{i}\right],
      \end{equation}
      and
       \begin{equation*}
       \psi_{2,j}^{i}=\frac{h^{4}}{180}\left[4u_{5x}^{i+\frac{1}{2}+\alpha}(\epsilon_{j}^{10})+4u_{5x}^{i+\frac{1}{2}+\alpha}(\epsilon_{j}^{11})-
      u_{5x}^{i+\frac{1}{2}+\alpha}(\epsilon_{j}^{8})-u_{5x}^{i+\frac{1}{2}+\alpha}(\epsilon_{j}^{9})\right]+
       \end{equation*}
        \begin{equation}\label{43}
       \frac{\alpha(\frac{1}{2}+\alpha)k^{2}}{12h}\left[H_{1,j+2}^{i}-8H_{1,j+1}^{i}+8H_{1,j-1}^{i}-H_{1,j-2}^{i}\right].
      \end{equation}
      To develop the first-level of the proposed approach, we should combine equations $(\ref{1e})$, $(\ref{19})$, $(\ref{23})$ and
      $(\ref{40})$-$(\ref{43})$. This provides
      \begin{equation*}
      c\Delta_{0t}^{\lambda}u_{j}^{i+\frac{1}{2}+\alpha}+\frac{1}{\Gamma(1-\lambda)}\left\{
        \underset{l=0}{\overset{i-1}\sum}\int_{t_{l+\frac{1}{2}}}^{t_{l+\frac{3}{2}}}E^{l}_{\tau,j}(\tau)(t_{i+\frac{1}{2}
      +\alpha}-\tau)^{-\lambda}d\tau+\int_{0}^{t_{\frac{1}{2}}}\frac{u_{\tau,j}(\tau)-\delta_{t}u_{j}^{0}}{(t_{i+\frac{1}{2}
      +\alpha}-\tau)^{\lambda}}d\tau\right.
       \end{equation*}
       \begin{equation*}
      \left.+\int_{t_{i+\frac{1}{2}}}^{t_{i+\frac{1}{2}+\alpha}}(u_{\tau,j}(\tau)-\delta_{t}u_{j}^{i})(t_{i+\frac{1}{2}+\alpha}
      -\tau)^{-\lambda}d\tau\right\}=\frac{q^{i+\frac{1}{2}+\alpha}}{12h^{2}}\left[-u_{j+2}^{\alpha_{i}}+16u_{j+1}^{\alpha_{i}}-
      30u_{j}^{\alpha_{i}}+16u_{j-1}^{\alpha_{i}}-u_{j-2}^{\alpha_{i}}\right]-
       \end{equation*}
      \begin{equation*}
      \frac{p^{i+\frac{1}{2}+\alpha}}{12h}\left[-u_{j+2}^{\alpha_{i}}+8u_{j+1}^{\alpha_{i}}-8u_{j-1}^{\alpha_{i}}+u_{j-2}^{\alpha_{i}}\right]-
      g_{j}^{i+\frac{1}{2}+\alpha}u_{j}^{\alpha_{i}}+s_{j}^{i+\frac{1}{2}+\alpha}+q^{i+\frac{1}{2}+\alpha}\psi_{1,j}^{i}-p^{i+\frac{1}{2}
      +\alpha}\psi_{2,j}^{i}+
       \end{equation*}
       \begin{equation}\label{43a}
       \alpha(\frac{1}{2}+\alpha)k^{2}g^{i+\frac{1}{2}+\alpha}_{j}H_{1,j}^{i}.
       \end{equation}
       Suppose $U_{j}^{i}$ be the approximate solution at the grid point $(x_{j},t_{i})$. Tracking the error terms in both sides, $(\ref{43a})$ becomes
       \begin{equation*}
      c\Delta_{0t}^{\lambda}U_{j}^{i+\frac{1}{2}+\alpha}=\frac{q^{i+\frac{1}{2}+\alpha}}{12h^{2}}\left[-U_{j+2}^{\alpha_{i}}+16U_{j+1}^{\alpha_{i}}-
      30U_{j}^{\alpha_{i}}+16U_{j-1}^{\alpha_{i}}-U_{j-2}^{\alpha_{i}}\right]-
       \end{equation*}
       \begin{equation*}
      \frac{p^{i+\frac{1}{2}+\alpha}}{12h}\left[-U_{j+2}^{\alpha_{i}}+8U_{j+1}^{\alpha_{i}}-8U_{j-1}^{\alpha_{i}}+U_{j-2}^{\alpha_{i}}\right]-
      g_{j}^{i+\frac{1}{2}+\alpha}U_{j}^{\alpha_{i}}+s_{j}^{i+\frac{1}{2}+\alpha}.
       \end{equation*}
       Using equation $(\ref{21})$, expanding and rearranging terms, this approximation is equivalent to
       \begin{equation*}
       U_{j}^{i+\frac{1}{2}}=U_{j}^{i}-\frac{k}{2}\left(f^{\lambda\alpha}_{i+\frac{1}{2},i}+f^{\lambda\alpha}_{i+\frac{1}{2},i-1}\right)^{-1}
       \left\{\text{\,\.{f}}_{i+\frac{1}{2},0}^{\lambda\alpha}\delta_{t}u_{j}^{0}+(d^{\lambda\alpha}_{i+\frac{1}{2},i-1}
        -f^{\lambda\alpha}_{i+\frac{1}{2},i-1})\delta_{t}U_{j}^{i-\frac{1}{2}}+
       \underset{l=0}{\overset{i-2}\sum}\left[f^{\lambda\alpha}_{i+\frac{1}{2},l}\delta_{t}U_{j}^{l+1}+\right.\right.
       \end{equation*}
        \begin{equation*}
       \left.(d^{\lambda\alpha}_{i+\frac{1}{2},l}-f^{\lambda\alpha}_{i+\frac{1}{2},l})\delta_{t}U_{j}^{l+\frac{1}{2}}\right]-
       \frac{q^{i+\frac{1}{2}+\alpha}}{12h^{2}}\left[-U_{j+2}^{\alpha_{i}}+16U_{j+1}^{\alpha_{i}}-
      30U_{j}^{\alpha_{i}}+16U_{j-1}^{\alpha_{i}}-U_{j-2}^{\alpha_{i}}\right]+
       \end{equation*}
      \begin{equation}\label{44}
      \left. \frac{p^{i+\frac{1}{2}+\alpha}}{12h}\left[-U_{j+2}^{\alpha_{i}}+8U_{j+1}^{\alpha_{i}}-8U_{j-1}^{\alpha_{i}}+U_{j-2}^{\alpha_{i}}\right]+
      g_{j}^{i+\frac{1}{2}+\alpha}U_{j}^{\alpha_{i}}-s_{j}^{i+\frac{1}{2}+\alpha}\right\},\text{\,\,\,\,for\,\,\,\,}i\geq1.
       \end{equation}
       Here the sum equals zero is the upper summation index is less than the lower one.\\

       For $i=0,$ combining equations $(\ref{1e})$, $(\ref{18})$, $(\ref{23})$ and $(\ref{40})$-$(\ref{43})$, it is not difficult
     to observe that
       \begin{equation*}
       f_{\frac{1}{2},0}^{\lambda\alpha}\delta_{t}u_{j}^{0}+\frac{1}{\Gamma(1-\lambda)}\int_{0}^{t_{\frac{1}{2}
      +\alpha}}\frac{u_{\tau,j}(\tau)-\delta_{t}u_{j}^{0}}{(t_{\frac{1}{2}+\alpha}-\tau)^{\lambda}}d\tau=\frac{q^{\frac{1}{2}+\alpha}}{12h^{2}}
      \left[-u_{j+2}^{\alpha_{0}}+16u_{j+1}^{\alpha_{0}}-30u_{j}^{\alpha_{0}}+16u_{j-1}^{\alpha_{0}}-u_{j-2}^{\alpha_{0}}\right]-
       \end{equation*}
       \begin{equation}\label{44a}
      \frac{p^{\frac{1}{2}+\alpha}}{12h}\left[-u_{j+2}^{\alpha_{0}}+8u_{j+1}^{\alpha_{0}}-8u_{j-1}^{\alpha_{0}}+u_{j-2}^{\alpha_{0}}\right]-
      g_{j}^{\frac{1}{2}+\alpha}u_{j}^{\alpha_{0}}+s_{j}^{\frac{1}{2}+\alpha}+q^{\frac{1}{2}+\alpha}\psi_{1,j}^{0}-p^{\frac{1}{2}+\alpha}\psi_{2,j}^{0}+
      \alpha(\frac{1}{2}+\alpha)k^{2}g^{\frac{1}{2}+\alpha}_{j}H_{1,j}^{0}.
       \end{equation}
       Omitting the error terms in both sides of equation $(\ref{44a})$ to obtain
        \begin{equation*}
       f_{\frac{1}{2},0}^{\lambda\alpha}\delta_{t}U_{j}^{0}=\frac{q^{\frac{1}{2}+\alpha}}{12h^{2}}
      \left[-U_{j+2}^{\alpha_{0}}+16U_{j+1}^{\alpha_{0}}-30U_{j}^{\alpha_{0}}+16U_{j-1}^{\alpha_{0}}-U_{j-2}^{\alpha_{0}}\right]-
       \end{equation*}
       \begin{equation}\label{44aa}
      \frac{p^{\frac{1}{2}+\alpha}}{12h}\left[-U_{j+2}^{\alpha_{0}}+8U_{j+1}^{\alpha_{0}}-8U_{j-1}^{\alpha_{0}}+U_{j-2}^{\alpha_{0}}\right]-
      g_{j}^{\frac{1}{2}+\alpha}U_{j}^{\alpha_{0}}+s_{j}^{\frac{1}{2}+\alpha}.
       \end{equation}
       Expanding and rearranging terms, it is easy to see that
        \begin{equation*}
       U_{j}^{\frac{1}{2}}=U_{j}^{0}+\frac{k}{2}\left(f_{\frac{1}{2},0}^{\lambda\alpha}\right)^{-1}\left\{\frac{q^{\frac{1}{2}+\alpha}}{12h^{2}}
      \left[-U_{j+2}^{\alpha_{0}}+16U_{j+1}^{\alpha_{0}}-30U_{j}^{\alpha_{0}}+16U_{j-1}^{\alpha_{0}}-U_{j-2}^{\alpha_{0}}\right]-\right.
       \end{equation*}
       \begin{equation}\label{46}
      \left.\frac{p^{\frac{1}{2}+\alpha}}{12h}\left[-U_{j+2}^{\alpha_{0}}+8U_{j+1}^{\alpha_{0}}-8U_{j-1}^{\alpha_{0}}+U_{j-2}^{\alpha_{0}}\right]-
      g_{j}^{\frac{1}{2}+\alpha}U_{j}^{\alpha_{0}}+s_{j}^{\frac{1}{2}+\alpha}\right\}.
       \end{equation}
      It worth noticing to mention that equation $(\ref{44})$ defines the first-level of the new method subject to the appropriate initial condition
      $(\ref{46})$.\\

      Now, we should describe the second step of the two-level formulation for solving the considered time-fractional convection-diffusion-reaction
     equation.\\

       Considering equation $(\ref{2})$, it is not hard to see that
       \begin{equation*}
        cD_{0t}^{\lambda}u_{j}^{i+1+\alpha}=\frac{1}{\Gamma(1-\lambda)}\int_{0}^{t_{i+1+\alpha}}u_{\tau,j}(\tau)
      (t_{i+1+\alpha}-\tau)^{-\lambda}d\tau=\frac{1}{\Gamma(1-\lambda)}\underset{l=0}{\overset{i}\sum}\int_{t_{l}}^{t_{l+1
      }}u_{\tau,j}(\tau)(t_{i+1+\alpha}-\tau)^{-\lambda}d\tau
       \end{equation*}
       \begin{equation}\label{47}
      \left.+\int_{t_{i+1}}^{t_{i+1+\alpha}}u_{\tau,j}(\tau)(t_{i+1+\alpha}-\tau)^{-\lambda}d\tau\right].
       \end{equation}

       Let $\widehat{P}_{2}^{u_{j},l}$ be the quadratic polynomial interpolating the function $u_{j}(t)$ at the points $(t_{l},u_{j}^{l}),$
      $(t_{l+\frac{1}{2}},u_{j}^{l+\frac{1}{2}})$ and $(t_{l+1},u_{j}^{l+1})$. Direct computations result in
       \begin{equation}\label{47a}
        \widehat{P}_{2}^{u_{j},l}(t)=\frac{2}{k^{2}}\left[(t-t_{l+\frac{1}{2}})(t-t_{l+1})u_{j}^{l}-2(t-t_{l})(t-t_{l+1})u_{j}^{l+\frac{1}{2}}
        +(t-t_{l})(t-t_{l+\frac{1}{2}})u_{j}^{l+1}\right].
       \end{equation}
       The corresponding error is given by
       \begin{equation}\label{48}
        \widehat{E}^{l}_{j}(t)=u_{j}(t)-\widehat{P}_{2}^{u_{j},l}(t)=\frac{1}{6}(t-t_{l})(t-t_{l+\frac{1}{2}})(t-t_{l+1})u_{3t,j}(\epsilon_{l}^{13}),
       \end{equation}
       where $\epsilon_{l}^{13}$ is between the minimum and maximum of $t_{l}$, $t_{l+\frac{1}{2}}$, $t_{l+1}$ and $t$. Furthermore
       \begin{equation}\label{49}
        \widehat{P}_{2,t}^{u_{j},l}(t)=\frac{2}{k^{2}}\left[2(u_{j}^{l}-2u_{j}^{l+\frac{1}{2}}+u_{j}^{l+1})t-(t_{l+\frac{1}{2}}+t_{l+1})u_{j}^{l}+
        2(t_{l}+t_{l+1})u_{j}^{l+\frac{1}{2}}-(t_{l}+t_{l+\frac{1}{2}})u_{j}^{l+1}\right].
       \end{equation}
       Setting
        \begin{equation}\label{50}
       J_{j}^{(i,\alpha)}=\frac{1}{\Gamma(1-\lambda)}\left[\int_{t_{i+1}}^{t_{i+1+\alpha}}u_{\tau,j}(\tau)(t_{i+1+\alpha}-\tau)^{-\lambda}d\tau+
       \underset{l=0}{\overset{i}\sum}\int_{t_{l}}^{t_{l+1}}\widehat{E}^{l}_{\tau,j}(\tau)(t_{i+1+\alpha}-\tau)^{-\lambda}d\tau\right].
       \end{equation}
       Combining equations $(\ref{47})$-$(\ref{50})$, simple calculations give
       \begin{equation*}
        cD_{0t}^{\lambda}u_{j}^{i+1+\alpha}=\frac{1}{\Gamma(1-\lambda)}\left[\underset{l=0}{\overset{i}\sum}\int_{t_{l}}^{t_{l+1}}
        \widehat{P}_{2,\tau}^{u_{j},l}(\tau)(t_{i+1+\alpha}-\tau)^{-\lambda}d\tau+\delta_{t}u_{j}^{l+\frac{1}{2}}\int_{t_{i+1}}^{t_{i+1+\alpha}}
        (t_{i+1+\alpha}-\tau)^{-\lambda}d\tau\right]+J_{j}^{(i,\alpha)}
       \end{equation*}
       \begin{equation*}
        =\frac{k^{1-\lambda}}{\Gamma(2-\lambda)}\underset{l=0}{\overset{i}\sum}\left[2\left(\delta_{t}u_{j}^{l+\frac{1}{2}}-\delta_{t}u_{j}^{l}\right)
        \left(l(i+1+\alpha-l)^{1-\lambda}-(l+1)(i+\alpha-l)^{1-\lambda}+\frac{1}{2-\lambda}(i+1+\alpha-l)^{2-\lambda}\right.\right.
       \end{equation*}
       \begin{equation*}
        \left.\left.-\frac{1}{2-\lambda}(i+\alpha-l)^{2-\lambda}\right)+\left((2l+\frac{1}{2})\delta_{t}u_{j}^{l+\frac{1}{2}}-(2l+\frac{3}{2})
        \delta_{1}u_{j}^{l}\right)\left((i+\alpha-l)^{1-\lambda}-(i+1+\alpha-l)^{1-\lambda}\right)\right]+J_{j}^{(i,\alpha)}
       \end{equation*}
       \begin{equation}\label{51}
        =\underset{l=0}{\overset{i}\sum}\left[f^{\lambda\alpha}_{i+1,l}\delta_{t}u_{j}^{l+\frac{1}{2}}+(d^{\lambda\alpha}_{i+1,l}
        -f^{\lambda\alpha}_{i+1,l})\delta_{t}u_{j}^{l}\right]+f^{\lambda\alpha}_{i+1,i+1}\delta_{t}u_{j}^{i+\frac{1}{2}}+J_{j}^{(i,\alpha)}=
        c\Delta_{0t}^{\lambda}u^{i+1+\alpha}_{j}+J_{j}^{(i,\alpha)},
       \end{equation}
       where
        \begin{equation}\label{51a}
        c\Delta_{0t}^{\lambda}u^{i+1+\alpha}_{j}=\underset{l=0}{\overset{i}\sum}\left[f^{\lambda\alpha}_{i+1,l}\delta_{t}u_{j}^{l+\frac{1}{2}}
        +(d^{\lambda\alpha}_{i+1,l}-f^{\lambda\alpha}_{i+1,l})\delta_{t}u_{j}^{l}\right]+f^{\lambda\alpha}_{i+1,i+1}\delta_{t}u_{j}^{i+\frac{1}{2}},
       \end{equation}
       \begin{equation}\label{52}
        d^{\lambda\alpha}_{i+1,l}=\frac{k^{1-\lambda}}{\Gamma(2-\lambda)}\widetilde{d}^{\lambda\alpha}_{i+1,l},\text{\,\,}f^{\lambda\alpha}_{i+1,l}=
        \frac{k^{1-\lambda}}{\Gamma(2-\lambda)}\widetilde{f}^{\lambda\alpha}_{i+1,l},\text{\,\,for\,\,}0\leq l\leq i,\text{\,\,\,}
        f^{\lambda\alpha}_{i+1,i+1}=\frac{k^{1-\lambda}}{\Gamma(2-\lambda)}\widetilde{f}^{\lambda\alpha}_{i+1,i+1},
       \end{equation}
       where
       \begin{equation*}
        \widetilde{d}^{\lambda\alpha}_{i+1,l}=(i+1+\alpha-l)^{1-\lambda}-(i+\alpha-l)^{1-\lambda},\text{\,\,\,}
        \widetilde{f}^{\lambda\alpha}_{i+1,i+1}=\alpha^{1-\lambda},
       \end{equation*}
       \begin{equation}\label{53}
        \widetilde{f}^{\lambda\alpha}_{i+1,l}=\frac{2}{2-\lambda}\left[(i+1+\alpha-l)^{2-\lambda}-
        (i+\alpha-l)^{2-\lambda}\right]-\frac{1}{2}[(i+1+\alpha-l)^{1-\lambda}+3(i+\alpha-l)^{1-\lambda}],
       \end{equation}
       for $0\leq l\leq i,$ with $i\geq0.$ Applying the Taylor approximation, it is not difficult to show that
       \begin{equation}\label{55}
       u_{j}^{i+1+\alpha}=(1+2\alpha)u_{j}^{i+1}-2\alpha u_{j}^{i+\frac{1}{2}}-\alpha(\frac{1}{2}+\alpha)k^{2}H^{i}_{2,j},
       \end{equation}
       where
        \begin{equation}\label{56}
       H^{i}_{2,j}=\alpha u_{2t,j}(\epsilon_{6}^{i})-(\frac{1}{2}+\alpha)u_{2t,j}(\epsilon_{7}^{i}),
       \end{equation}
        where $\epsilon_{6}^{i}\in(t_{i+1},t_{i+1+\alpha})$ and $\epsilon_{7}^{i}\in(t_{i+\frac{1}{2}},t_{i+1+\alpha})$. Replacing
       $t_{i+\frac{1}{2}+\alpha}$ by $t_{i+1+\alpha}$ in equation $(\ref{40})$-$(\ref{41})$ to obtain
        \begin{equation}\label{57}
         u_{2x,j}^{i+1+\alpha}=\frac{1}{12h^{2}}\left[-u_{j+2}^{\theta_{i}}+16u_{j+1}^{\theta_{i}}-30u_{j}^{\theta_{i}}+16u_{j-1}^{\theta_{i}}
         -u_{j-2}^{\theta_{i}}\right]+\psi_{3,j}^{i},
      \end{equation}
      and
       \begin{equation}\label{58}
         u_{x,j}^{i+1+\alpha}=\frac{1}{12h}\left[-u_{j+2}^{\theta_{i}}+8u_{j+1}^{\theta_{i}}-8u_{j-1}^{\theta_{i}}+u_{j-2}^{\theta_{i}}\right]
         +\psi_{4,j}^{i},
      \end{equation}
      where we set
      \begin{equation}\label{59}
      u_{j}^{\theta_{i}}=(1+2\alpha)u_{j}^{i+1}-2\alpha u_{j}^{i+\frac{1}{2}},
      \end{equation}
       \begin{equation*}
       \psi_{3,j}^{i}=\frac{h^{4}}{270}\left[18u_{6x}^{i+1+\alpha}(\epsilon_{j}^{12})+18u_{6x}^{i+1+\alpha}(\epsilon_{j}^{13})-
      u_{6x}^{i+1+\alpha}(\epsilon_{j}^{14})-u_{6x}^{i+1+\alpha}(\epsilon_{j}^{15})\right]+
       \end{equation*}
        \begin{equation}\label{60}
       \frac{\alpha(\frac{1}{2}+\alpha)k^{2}}{12h^{2}}\left[H_{2,j+2}^{i}-16H_{2,j+1}^{i}+30H_{2,j}^{i}-16H_{2,j-1}^{i}+H_{2,j-2}^{i}\right],
      \end{equation}
      and
       \begin{equation*}
       \psi_{4,j}^{i}=\frac{h^{4}}{180}\left[4u_{5x}^{i+1+\alpha}(\epsilon_{j}^{18})+4u_{5x}^{i+1+\alpha}(\epsilon_{j}^{19})-
      u_{5x}^{i+1+\alpha}(\epsilon_{j}^{16})-u_{5x}^{i+1+\alpha}(\epsilon_{j}^{17})\right]+
       \end{equation*}
        \begin{equation}\label{61}
       \frac{\alpha(\frac{1}{2}+\alpha)k^{2}}{12h}\left[H_{2,j+2}^{i}-8H_{2,j+1}^{i}+8H_{2,j-1}^{i}-H_{2,j-2}^{i}\right],
      \end{equation}
      where $\epsilon_{j}^{8},\epsilon_{j}^{14},\epsilon_{j}^{16}\in(x_{j-1},x_{j})$, $\epsilon_{j}^{9},\epsilon_{j}^{15},\epsilon_{j}^{17}
      \in(x_{j},x_{j+1})$, $\epsilon_{j}^{10},\epsilon_{j}^{12},\epsilon_{j}^{18}\in(x_{j-2},x_{j})$ and
       $\epsilon_{j}^{11},\epsilon_{j}^{13},\epsilon_{j}^{19}\in(x_{j},x_{j+2})$. Here the terms $H_{2,m}^{i}$ are defined as in equation $(\ref{56})$.\\

      Combining equations $(\ref{1e})$, $(\ref{23})$, $(\ref{51})$, $(\ref{55})$, $(\ref{57})$ and $(\ref{58})$, direct computations result in

       \begin{equation*}
      c\Delta_{0t}^{\lambda}u^{i+1+\alpha}_{j}+J_{j}^{(i,\alpha)}=\frac{q^{i+1+\alpha}}{12h^{2}}\left[-u_{j+2}^{\theta_{i}}+16u_{j+1}^{\theta_{i}}-
      30u_{j}^{\theta_{i}}+16u_{j-1}^{\theta_{i}}-u_{j-2}^{\theta_{i}}\right]-\frac{p^{i+1+\alpha}}{12h}\left[-u_{j+2}^{\theta_{i}}+8u_{j+1}^{\theta_{i}}
      \right.
       \end{equation*}
      \begin{equation}\label{62}
      \left.-8u_{j-1}^{\theta_{i}}+u_{j-2}^{\theta_{i}}\right]-g_{j}^{i+1+\alpha}u_{j}^{\theta_{i}}+s_{j}^{i+1+\alpha}+q^{i+1+\alpha}\psi_{3,j}^{i}
      -p^{i+1+\alpha}\psi_{4,j}^{i}+\alpha(\frac{1}{2}+\alpha)k^{2}g^{i+1+\alpha}_{j}H_{2,j}^{i}.
       \end{equation}
       Tracking the error terms $J_{j}^{(i,\alpha)}$, $q^{i+1+\alpha}\psi_{3,j}^{i}$, $p^{i+1+\alpha}\psi_{4,j}^{i}$ and
       $\alpha(\frac{1}{2}+\alpha)k^{2}g^{i+1+\alpha}_{j}H_{2,j}^{i},$ plugging equations $(\ref{51a})$ and $(\ref{62})$ and rearranging terms, this yields
       \begin{equation*}
       U_{j}^{i+1}=U_{j}^{i+\frac{1}{2}}-\frac{k}{2}\left(f^{\lambda\alpha}_{i+1,i}+f^{\lambda\alpha}_{i+1,i+1}\right)^{-1}
       \left\{(d^{\lambda\alpha}_{i+1,i}-f^{\lambda\alpha}_{i+1,i})\delta_{t}U_{j}^{i}+
       \underset{l=0}{\overset{i-1}\sum}\left[f^{\lambda\alpha}_{i+1,l}\delta_{t}U_{j}^{l+\frac{1}{2}}+
       (d^{\lambda\alpha}_{i+1,l}-f^{\lambda\alpha}_{i+1,l})\delta_{t}U_{j}^{l}\right]-\right.
       \end{equation*}
        \begin{equation*}
      \frac{q^{i+1+\alpha}}{12h^{2}}\left[-U_{j+2}^{\theta_{i}}+16U_{j+1}^{\theta_{i}}-30U_{j}^{\theta_{i}}+
      16U_{j-1}^{\theta_{i}}-U_{j-2}^{\theta_{i}}\right]+
       \end{equation*}
      \begin{equation}\label{63}
       \left. \frac{p^{i+1+\alpha}}{12h}\left[-U_{j+2}^{\theta_{i}}+8U_{j+1}^{\theta_{i}}-8U_{j-1}^{\theta_{i}}+U_{j-2}^{\theta_{i}}\right]+
      g_{j}^{i+1+\alpha}U_{j}^{\theta_{i}}-s_{j}^{i+1+\alpha}\right\},\text{\,\,\,\,for\,\,\,\,}i\geq1.
       \end{equation}
       We recall that $U_{m}^{\theta_{i}}=(1+2\alpha)U_{m}^{i+1}-2\alpha U_{m}^{i+\frac{1}{2}}$, for $i\geq0$. Equation $(\ref{63})$ denotes the
      second-level of the proposed procedure applied to the time-fractional parabolic equations $(\ref{1e})$-$(\ref{3e})$.\\

       Now, utilizing $(\ref{63})$ and assuming that the sum equals zero when the upper summation index is less than the lower one, we obtain
        \begin{equation*}
       U_{j}^{1}=U_{j}^{\frac{1}{2}}-\frac{k}{2}\left(f^{\lambda\alpha}_{1,0}+f^{\lambda\alpha}_{1,1}\right)^{-1}
       \left\{(d^{\lambda\alpha}_{1,0}-f^{\lambda\alpha}_{1,0})\delta_{t}U_{j}^{0}-\frac{q^{1+\alpha}}{12h^{2}}\left[-U_{j+2}^{\theta_{0}}
       +16U_{j+1}^{\theta_{0}}-30U_{j}^{\theta_{0}}+16U_{j-1}^{\theta_{0}}-U_{j-2}^{\theta_{0}}\right]+\right.
       \end{equation*}
      \begin{equation}\label{64n}
       \left. \frac{p^{1+\alpha}}{12h}\left[-U_{j+2}^{\theta_{0}}+8U_{j+1}^{\theta_{0}}-8U_{j-1}^{\theta_{0}}+U_{j-2}^{\theta_{0}}\right]+
      g_{j}^{1+\alpha}U_{j}^{\theta_{0}}-s_{j}^{1+\alpha}\right\}.
       \end{equation}
       An assembly of equations $(\ref{46})$, $(\ref{44})$, $(\ref{64n})$ and $(\ref{63})$ after rearranging terms provides a two-level
      fourth-order approach for solving the initial-boundary value problem $(\ref{1e})$-$(\ref{3e})$. That is, for $i=1,2,...,N-1,$ and $j=2,3,...,M-2,$
      \begin{equation*}
       \frac{(1+2\alpha)k}{24h^{2}}\left[\left(q^{\frac{1}{2}+\alpha}+hp^{\frac{1}{2}+\alpha}\right)U_{j-2}^{\alpha_{0}}-\left(16q^{\frac{1}{2}+\alpha}
     +8hp^{\frac{1}{2}+\alpha}\right)U_{j-1}^{\alpha_{0}}-\left(16q^{\frac{1}{2}+\alpha}-8hp^{\frac{1}{2}+\alpha}\right)U_{j+1}^{\alpha_{0}}\right.
       \end{equation*}
     \begin{equation}\label{s1}
       \left.+\left(q^{\frac{1}{2}+\alpha}-hp^{\frac{1}{2}+\alpha}\right)U_{j+2}^{\alpha_{0}}\right]+\left[f^{\lambda\alpha}_{\frac{1}{2},0}+\frac{(1+2\alpha)k}{24h^{2}}
       \left(30q^{\frac{1}{2}+\alpha}+12h^{2}g_{j}^{\frac{1}{2}+\alpha}\right)\right]U_{j}^{\alpha_{0}}=f^{\lambda\alpha}_{\frac{1}{2},0}U_{j}^{0}+
       \frac{(1+2\alpha)k}{2}s_{j}^{\frac{1}{2}+\alpha},
       \end{equation}
       where $U_{l}^{\alpha_{0}}=(1+2\alpha)U_{l}^{\frac{1}{2}}-2\alpha U_{l}^{0}$, for $l=2,3,...,M-2$,
        \begin{equation*}
       \frac{(1+2\alpha)k}{24h^{2}}\left[
      \left(q^{1+\alpha}+hp^{1+\alpha}\right)U_{j-2}^{\theta_{0}}-\left(16q^{1+\alpha}+8hp^{1+\alpha}\right)U_{j-1}^{\theta_{0}}-
      \left(16q^{1+\alpha}-8hp^{1+\alpha}\right)U_{j+1}^{\theta_{0}}+\right.
       \end{equation*}
     \begin{equation*}
       \left.\left(q^{1+\alpha}-hp^{1+\alpha}\right)U_{j+2}^{\theta_{0}}\right]+\left[\left(f^{\lambda\alpha}_{1,1}+f^{\lambda\alpha}_{1,0}\right)
       +\frac{(1+2\alpha)k}{24h^{2}}\left(30q^{1+\alpha}+12h^{2}g_{j}^{1+\alpha}\right)\right]U_{j}^{\theta_{0}}=
       \end{equation*}
       \begin{equation}\label{s2}
        \left(f^{\lambda\alpha}_{1,1}+f^{\lambda\alpha}_{1,0}\right)U_{j}^{\frac{1}{2}}-
       \frac{(1+2\alpha)k}{2}\left[(d^{\lambda\alpha}_{1,0}-f^{\lambda\alpha}_{1,0})\delta_{t}U_{j}^{0}-s_{j}^{1+\alpha}\right],
       \end{equation}
      where $U_{r}^{\theta_{0}}=(1+2\alpha)U_{r}^{1}-2\alpha U_{r}^{\frac{1}{2}}$, for $r=2,3,...,M-2$,
      \begin{equation*}
       \frac{(1+2\alpha)k}{24h^{2}}\left[\left(q^{i+\frac{1}{2}+\alpha}+hp^{i+\frac{1}{2}+\alpha}\right)U_{j-2}^{\alpha_{i}}-\left(16q^{i+\frac{1}{2}+\alpha}
       +8hp^{i+\frac{1}{2}+\alpha}\right)U_{j-1}^{\alpha_{i}}-\right.
       \end{equation*}
     \begin{equation*}
       \left.\left(16q^{i+\frac{1}{2}+\alpha}-8hp^{i+\frac{1}{2}+\alpha}\right)U_{j+1}^{\alpha_{i}}+\left(q^{i+\frac{1}{2}+\alpha}-hp^{i+\frac{1}{2}+\alpha}
     \right)U_{j+2}^{\alpha_{i}}\right]+\left[\left(f^{\lambda\alpha}_{i+\frac{1}{2},i}+f^{\lambda\alpha}_{i+\frac{1}{2},i-1}\right)+\right.
       \end{equation*}
       \begin{equation*}
       \left.\frac{(1+2\alpha)k}{24h^{2}}\left(30q^{i+\frac{1}{2}+\alpha}+12h^{2}g_{j}^{i+\frac{1}{2}+\alpha}\right)\right]U_{j}^{\alpha_{i}}=
      \left(f^{\lambda\alpha}_{i+\frac{1}{2},i}+f^{\lambda\alpha}_{i+\frac{1}{2},i-1}\right)U_{j}^{i}-\frac{(1+2\alpha)k}{2}
     \left\{\text{\,\.{f}}_{i+\frac{1}{2},0}^{\lambda\alpha}\delta_{t}u_{j}^{0}\right.
       \end{equation*}
        \begin{equation}\label{s3}
        \left.+(d^{\lambda\alpha}_{i+\frac{1}{2},i-1}-f^{\lambda\alpha}_{i+\frac{1}{2},i-1})\delta_{t}U_{j}^{i-\frac{1}{2}}+\underset{l=0}{\overset{i-2}\sum}
       \left[f^{\lambda\alpha}_{i+\frac{1}{2},l}\delta_{t}U_{j}^{l+1}+(d^{\lambda\alpha}_{i+\frac{1}{2},l}-f^{\lambda\alpha}_{i+\frac{1}{2},l})
      \delta_{t}U_{j}^{l+\frac{1}{2}}\right]-s_{j}^{i+\frac{1}{2}+\alpha}\right\},
       \end{equation}
       where $U_{m}^{\alpha_{i}}=(1+2\alpha)U_{m}^{i+\frac{1}{2}}-2\alpha U_{m}^{i}$, for $m=2,3,...,M-2$,
        \begin{equation*}
       \frac{(1+2\alpha)k}{24h^{2}}\left[\left(q^{i+1+\alpha}+hp^{i+1+\alpha}\right)U_{j-2}^{\theta_{i}}-\left(16q^{i+1+\alpha}+
      8hp^{i+1+\alpha}\right)U_{j-1}^{\theta_{i}}-\left(16q^{i+1+\alpha}-8hp^{i+1+\alpha}\right)U_{j+1}^{\theta_{i}}\right.
       \end{equation*}
     \begin{equation*}
       \left.+\left(q^{i+1+\alpha}-hp^{i+1+\alpha}\right)U_{j+2}^{\theta_{i}}\right]+
      \left[\left(f^{\lambda\alpha}_{i+1,i+1}+f^{\lambda\alpha}_{i+1,i}\right)+\frac{(1+2\alpha)k}{24h^{2}}\left(30q^{i+1+\alpha}
      +12h^{2}g_{j}^{i+1+\alpha}\right)\right]U_{j}^{\theta_{i}}
       \end{equation*}
       \begin{equation*}
       =\left(f^{\lambda\alpha}_{i+1,i+1}+f^{\lambda\alpha}_{i+1,i}\right)U_{j}^{i+\frac{1}{2}}-\frac{(1+2\alpha)k}{2}
     \left\{(d^{\lambda\alpha}_{i+1,i}-f^{\lambda\alpha}_{i+1,i})\delta_{t}U_{j}^{i}\right.
       \end{equation*}
        \begin{equation}\label{s4}
        \left.+\underset{l=0}{\overset{i-1}\sum}\left[f^{\lambda\alpha}_{i+1,l}\delta_{t}U_{j}^{l+\frac{1}{2}}+
      (d^{\lambda\alpha}_{i+1,l}-f^{\lambda\alpha}_{i+1,l})\delta_{t}U_{j}^{l}\right]-s_{j}^{i+1+\alpha}\right\},
       \end{equation}
      $U_{m}^{\theta_{i}}=(1+2\alpha)U_{m}^{i+1}-2\alpha U_{m}^{i+\frac{1}{2}}$, for $m=2,3,...,M-2$. With initial and boundary conditions
       \begin{equation}\label{s5}
       U_{j}^{0}=\psi_{1,j},\text{\,\,}j=0,1,...,M;\text{\,\,and\,\,}U_{r}^{\alpha_{i}}=\psi_{2,r}^{\alpha_{i}},\text{\,\,}
       U_{r}^{\theta_{i}}=\psi_{2,r}^{\theta_{i}},\text{\,\,\,for\,\,\,}i=0,1,...,N-1;\text{\,\,}r\in\{0,\text{\,}1,\text{\,}M-1,\text{\,}M\}.
       \end{equation}
     We recall that the sum equals zero if the upper summation index is less than the lower ones.

     \section{Stability analysis and convergence rate of the new algorithm $(\ref{s1})$-$(\ref{s5})$}\label{sec3}
    This section considers the analysis of stability and convergence rate of the two-level fourth-order method $(\ref{s1})$-$(\ref{s5})$ applied to
    the time-fractional convection-diffusion-reaction equation $(\ref{1e})$ subjects to initial and boundary conditions $(\ref{2e})$ and $(\ref{3e})$,
    respectively. In this study we assume that the parameters $\lambda$ and $\alpha$ satisfy $0<\lambda<\frac{2}{3}$ and $\alpha=1-\lambda$.
    This requirement plays a crucial role in the proof of Lemma $(\ref{l6})$ and Theorem $(\ref{tt1})$. Firstly, it is worth noticing to mention that 
    the stability analysis and convergence rate of the proposed approach requires some intermediate results (namely Lemmas $\ref{l1}$-$\ref{l9}$).

    \begin{lemma}\label{l1}
   For any $\lambda\in(0,1)$ and $\alpha=1-\lambda$, suppose the function $u(\cdot,\cdot)\in\mathcal{C}^{6,3}(D):=\mathcal{C}^{6,3}_{D}$,
   where $D=[0,L_{1}]\times[0,T]$, so
   \begin{equation}\label{1l}
   \underset{0\leq i\leq N-1}{\max}\|cD_{0t}^{\lambda}u^{i+r+\alpha}_{j}-c\Delta_{0t}^{\lambda}u^{i+r+\alpha}_{j}\|_{L^{2}}\leq C_{r}k^{2-\lambda}
   \end{equation}
   where $r\in\{\frac{1}{2},1\}$. $cD_{0t}^{\lambda}u^{i+\frac{1}{2}+\alpha}_{j}$, (for $i=0$ and $1\leq i\leq N-1$) are given by equations $(\ref{18})$
   and $(\ref{19})$, respectively; $c\Delta_{0t}^{\lambda}u^{i+\frac{1}{2}+\alpha}_{j}$, (for $i=0$ and $1\leq i\leq N-1$) are defined by equations
    $(\ref{20})$ and $(\ref{21})$, respectively. $cD_{0t}^{\lambda}u^{i+1+\alpha}_{j}$ and $c\Delta_{0t}^{\lambda}u^{i+1+\alpha}_{j}$, (for $0\leq i\leq
    N-1$) are given by equations $(\ref{51})$ and $(\ref{52})$, respectively. $C_{r}$ are positive constants which do not depend on the time step $k$ and
    space step $h$.
   \end{lemma}

   \begin{proof}
   A combination of equations $(\ref{18})$ and $(\ref{20})$ yields
   \begin{equation}\label{67}
   |cD_{0t}^{\lambda}u^{\frac{1}{2}+\alpha}_{j}-c\Delta_{0t}^{\lambda}u^{\frac{1}{2}+\alpha}_{j}|=\frac{1}{\Gamma(1-\lambda)^{2}}
   \left|\int_{0}^{t_{\frac{1}{2}+\alpha}}(u_{\tau,j}(\tau)-\delta_{t}u_{j}^{0})(t_{\frac{1}{2}+\alpha}-\tau)^{-\lambda}d\tau\right|.
   \end{equation}
   The application of the Taylor series gives
   \begin{equation}\label{67a}
   u_{\tau,j}(\tau)=u_{\tau,j}(t_{\frac{1}{4}})+(\tau-t_{\frac{1}{4}})u_{2\tau,j}(t_{\frac{1}{4}})+\frac{1}{2}(\tau-
   t_{\frac{1}{4}})^{2}u_{3\tau,j}(\tau_{1}),
   \end{equation}
   where $\tau_{1}$ is between the minimum and maximum of $\tau$ and $t_{\frac{1}{4}}$. Substituting $(\ref{67a})$ into $(\ref{67})$ to obtain
   \begin{equation*}
    |cD_{0t}^{\lambda}u^{\frac{1}{2}+\alpha}_{j}-c\Delta_{0t}^{\lambda}u^{\frac{1}{2}+\alpha}_{j}|=\frac{1}{\Gamma(1-\lambda)}
   \left|(u_{\tau,j}(t_{\frac{1}{4}})-\delta_{t}u_{j}^{0})\int_{0}^{t_{\frac{1}{2}+\alpha}}(t_{\frac{1}{2}+\alpha}-\tau)^{-\lambda}d\tau+\right.
   \end{equation*}
   \begin{equation}\label{68}
    \left.u_{2\tau,j}(t_{\frac{1}{4}})\int_{0}^{t_{\frac{1}{2}+\alpha}}(\tau-t_{\frac{1}{4}})(t_{\frac{1}{2}+\alpha}-\tau)^{-\lambda}d\tau+
    \frac{1}{2}\int_{0}^{t_{\frac{1}{2}+\alpha}}(\tau-t_{\frac{1}{4}})^{2}u_{3\tau,j}(\tau)(t_{\frac{1}{2}+\alpha}-\tau)^{-\lambda}d\tau\right|.
   \end{equation}
   Performing simple calculations, it is easy to see that
   \begin{equation*}
    u_{j}^{i+\frac{1}{2}}=u_{j}^{i+\frac{1}{4}}+\frac{k}{4}u_{\tau,j}^{i+\frac{1}{4}}+\frac{k^{2}}{8}u_{2\tau,j}^{i+\frac{1}{4}}+O(k^{3});\text{\,\,\,}
     u_{j}^{i}=u_{j}^{i+\frac{1}{4}}-\frac{k}{4}u_{\tau,j}^{i+\frac{1}{4}}+\frac{k^{2}}{8}u_{2\tau,j}^{i+\frac{1}{4}}+O(k^{3}).
   \end{equation*}
   Using this, simple calculations result in
   \begin{equation}\label{68b}
    u_{\tau,j}^{i+\frac{1}{4}}-\delta_{t}u_{j}^{i}=O(k^{2}),\text{\,\,\,for\,\,\,}i\geq0.
    \end{equation}
    We recall that $\delta_{t}u_{j}^{i}=\frac{2}{k}(u_{j}^{i+\frac{1}{2}}-u_{j}^{i})$. So
    \begin{equation}\label{69}
    (u_{\tau,j}(t_{\frac{1}{4}})-\delta_{t}u_{j}^{0})\int_{0}^{t_{\frac{1}{2}+\alpha}}(t_{\frac{1}{2}+\alpha}-\tau)^{-\lambda}d\tau=
    \frac{-1}{1-\lambda}(u_{\tau,j}(t_{\frac{1}{4}})-\delta_{t}u_{j}^{0})\left[(t_{\frac{1}{2}+\alpha}-
    \tau)^{1-\lambda}\right]_{0}^{t_{\frac{1}{2}+\alpha}}=O(k^{3-\lambda}).
    \end{equation}
    On the other hand
    \begin{equation*}
    u_{2\tau,j}(t_{\frac{1}{4}})\int_{0}^{t_{\frac{1}{2}+\alpha}}(\tau-t_{\frac{1}{4}})(t_{\frac{1}{2}+\alpha}-\tau)^{-\lambda}d\tau=
    \frac{-u_{2\tau,j}(t_{\frac{1}{4}})}{1-\lambda}\left[(\tau-t_{\frac{1}{4}})(t_{\frac{1}{2}+\alpha}-\tau)^{1-\lambda}+
    \frac{1}{2-\lambda}(t_{\frac{1}{2}+\alpha}-\tau)^{2-\lambda}\right]_{0}^{t_{\frac{1}{2}+\alpha}}
   \end{equation*}
    \begin{equation}\label{70}
    =\frac{k^{2-\lambda}(\frac{1}{2}+\alpha)^{1-\lambda}}{1-\lambda}\left(-\frac{1}{4}+\frac{\frac{1}{2}+\alpha}{2-\lambda}\right)=O(k^{2-\lambda}).
    \end{equation}
    Setting $m(\tau)=(\tau-t_{\frac{1}{4}})^{2},$ so $\frac{d}{d\tau}m(\tau)=0,$ implies $\tau=t_{\frac{1}{4}}.$ So $\underset{0\leq\tau\leq t_{\frac{1}{2}
    +\alpha}}{\max}|m(\tau)|=(\frac{1}{4}+\alpha)^{2}k^{2}.$ Thus
    \begin{equation*}
    \left|\frac{1}{2}\int_{0}^{t_{\frac{1}{2}+\alpha}}(\tau-t_{\frac{1}{4}})^{2}u_{3\tau,j}(\tau)(t_{\frac{1}{2}+\alpha}-\tau)^{-\lambda}d\tau\right|
    \leq \frac{1}{2(1-\lambda)}(\frac{1}{4}+\alpha)^{2}k^{2}\underset{0\leq\tau\leq t_{\frac{1}{2}+\alpha}}{\max}|u_{3\tau,j}(\tau)|*
   \end{equation*}
    \begin{equation}\label{71}
    \left|\left[(t_{\frac{1}{2}+\alpha}-\tau)^{1-\lambda}\right]_{0}^{t_{\frac{1}{2}+\alpha}}\right|=
    \frac{1}{2(1-\lambda)}(\frac{1}{2}+\alpha)^{1-\lambda}(\frac{1}{4}+\alpha)^{2}k^{3-\lambda}\underset{0\leq\tau\leq t_{\frac{1}{2}+\alpha}}{\max}
   |u_{3\tau,j}(\tau)|=O(k^{3-\lambda}).
    \end{equation}
    Plugging estimates $(\ref{68})$-$(\ref{71})$ to get $|cD_{0t}^{\lambda}u^{\frac{1}{2}+\alpha}_{j}-c\Delta_{0t}^{\lambda}u^{\frac{1}{2}+\alpha}_{j}|
   =O(k^{2-\lambda})$. This fact, together with the definition of $L^{2}$-norm give
     \begin{equation}\label{71a}
   \|cD_{0t}^{\lambda}u^{\frac{1}{2}+\alpha}_{j}-c\Delta_{0t}^{\lambda}u^{\frac{1}{2}+\alpha}_{j}\|_{L^{2}}\leq C_{1}^{0}k^{2-\lambda},
   \end{equation}
   where $C_{1}^{0}$ is a positive constant. Furthermore, for $1\leq i\leq N-1,$ a combination of equations $(\ref{19})$ and $(\ref{21})$ provides
    \begin{equation*}
        cD_{0t}^{\lambda}u_{j}^{i+\frac{1}{2}+\alpha}-c\Delta_{0t}^{\lambda}u_{j}^{i+\frac{1}{2}+\alpha}=\frac{1}{\Gamma(1-\lambda)}\left\{
        \underset{l=0}{\overset{i-1}\sum}\int_{t_{l+\frac{1}{2}}}^{t_{l+\frac{3}{2}}}E^{l}_{\tau,j}(\tau)(t_{i+\frac{1}{2}
      +\alpha}-\tau)^{-\lambda}d\tau+\int_{0}^{t_{\frac{1}{2}}}\frac{u_{\tau,j}(\tau)-\delta_{t}u_{j}^{0}}{(t_{i+\frac{1}{2}
      +\alpha}-\tau)^{\lambda}}d\tau\right.
       \end{equation*}
       \begin{equation}\label{72}
      \left.+\int_{t_{i+\frac{1}{2}}}^{t_{i+\frac{1}{2}+\alpha}}(u_{\tau,j}(\tau)-\delta_{t}u_{j}^{i})(t_{i+\frac{1}{2}+\alpha}
      -\tau)^{-\lambda}d\tau\right\},\text{\,\,\,\,\,for\,\,\,\,\,}i\geq1.
       \end{equation}
       By straightforward computations, it is not hard to show that
        \begin{equation}\label{74}
      \left|\int_{t_{i+\frac{1}{2}}}^{t_{i+\frac{1}{2}+\alpha}}\frac{u_{\tau,j}(\tau)-\delta_{t}u_{j}^{i}}{(t_{i+\frac{1}{2}+\alpha}
      -\tau)^{\lambda}}d\tau\right|\leq \frac{(\alpha k)^{1-\lambda}}{1-\lambda}\left\{|u_{\tau,j}^{i+\frac{1}{2}}-\delta_{t}u_{j}^{i}|+
      \frac{\alpha k}{2-\lambda}\underset{t_{i+\frac{1}{2}}\leq\tau\leq t_{i+\frac{1}{2}+\alpha}}{\max}|u_{2\tau,j}(\tau)|\right\}
      =O(k^{2-\lambda}),
       \end{equation}
       where $u_{\tau,j}^{i+\frac{1}{2}}-\delta_{t}u_{j}^{i}=O(k^{2})$ is given by $(\ref{68b})$. In addition, using equation $(\ref{67a})$, it is
     not difficult to prove that
        \begin{equation}\label{73}
      \left|\int_{0}^{t_{\frac{1}{2}}}\frac{u_{\tau,j}(\tau)-\delta_{t}u_{j}^{0}}{(t_{i+\frac{1}{2}+\alpha}-\tau)^{\lambda}}d\tau\right|\leq
      \frac{k}{2}|u_{\tau,j}^{\frac{1}{4}}-\delta_{t}u_{j}^{0}|t_{i+\alpha}^{-\lambda}+\frac{k^{2}}{8}t_{i+\alpha}^{-\lambda}+
      \frac{k^{3}}{128}\cdot\underset{0\leq\tau\leq t_{\frac{1}{2}}}{\max}|u_{3\tau,j}(\tau)|= O(k^{2-\lambda}),
       \end{equation}
       since $u_{\tau,j}^{\frac{1}{4}}-\delta_{t}u_{j}^{0}=O(k^{2})$ and $t_{i+\alpha}^{-\lambda}=k^{-\lambda}(i+\alpha)^{-\lambda}\leq
      k^{-\lambda}\alpha^{-\lambda}$. Now, setting $m_{l}(\tau)=\frac{1}{6}(\tau-t_{l+\frac{1}{2}})(\tau-t_{l+1})(\tau-t_{l+\frac{3}{2}})$, simple
      calculations result in
       \begin{equation}\label{73b}
      \underset{t_{l+\frac{1}{2}}\leq\tau\leq t_{l+\frac{3}{2}}}{\max}|m_{l}(\tau)|=\frac{k^{3}}{72\sqrt{3}}.
       \end{equation}
       Combining of relations $(\ref{4a})$ and $(\ref{73b})$, direct computations provide
       \begin{equation*}
       \left|\int_{t_{l+\frac{1}{2}}}^{t_{l+\frac{3}{2}}}E^{l}_{\tau,j}(\tau)(t_{i+\frac{1}{2}+\alpha}-\tau)^{-\lambda}d\tau\right|
       \leq \frac{k^{3-\lambda}}{72\sqrt{3}}[(i+\alpha-l-1)^{-\lambda}-(i+\alpha-l)^{-\lambda}]\underset{t_{l+\frac{1}{2}}\leq\tau\leq
     t_{l+\frac{3}{2}}}{\max}|u_{3\tau,j}(\tau)|.
       \end{equation*}
       Since $i\geq1$, summing this up from $l=0,1,...,i-1,$ to obtain
       \begin{equation}\label{75}
       \underset{l=0}{\overset{i-1}\sum}\left|\int_{t_{l+\frac{1}{2}}}^{t_{l+\frac{3}{2}}}E^{l}_{\tau,j}(\tau)
       (t_{i+\frac{1}{2}+\alpha}-\tau)^{-\lambda}d\tau\right|\leq \frac{k^{3-\lambda}}{72\sqrt{3}}[\alpha^{-\lambda}-(i+\alpha)^{-\lambda}]
       \underset{0\leq\tau\leq t_{i-1}}{\max}|u_{3\tau,j}(\tau)|=O(k^{3-\lambda}).
       \end{equation}
      Plugging equation $(\ref{72})$ and estimate $(\ref{74})$-$(\ref{75})$, it is easy to see that
       \begin{equation}\label{75a}
      \|cD_{0t}^{\lambda}u^{i+\frac{1}{2}+\alpha}_{j}-c\Delta_{0t}^{\lambda}u^{i+\frac{1}{2}+\alpha}_{j}\|_{L^{2}}\leq C_{\frac{1}{2}}k^{2-\lambda},
      \text{\,\,\,for\,\,\,}i\geq1,
       \end{equation}
      where $C_{\frac{1}{2}}$ is a positive constant. In a similar manner, one easily proves the following inequality
       \begin{equation}\label{75b}
      \|cD_{0t}^{\lambda}u^{i+1+\alpha}_{j}-c\Delta_{0t}^{\lambda}u^{i+1+\alpha}_{j}\|_{L^{2}}\leq C_{1}k^{2-\lambda},
      \text{\,\,\,for\,\,\,}i\geq0,
       \end{equation}
        where $C_{1}>0$ is a constant. The proof of Lemma $\ref{l1}$ is completed by taking the maximum over $i$ ($0\leq i\leq N-1$) of estimates
    $(\ref{71a})$ and $(\ref{75a})$ (respectively, estimate $(\ref{75b})$).
   \end{proof}

   \begin{lemma}\label{l2}
   Define the following linear operators
   \begin{equation*}
      L_{h}u_{j}^{\gamma_{i}}=\frac{q^{\widetilde{\gamma}_{i}}}{12h^{2}}\left[-u_{j+2}^{\gamma_{i}}+16u_{j+1}^{\gamma_{i}}-30u_{j}^{\gamma_{i}}+
      16u_{j-1}^{\gamma_{i}}-u_{j-2}^{\gamma_{i}}\right]-
       \end{equation*}
      \begin{equation}\label{79}
       \frac{p^{\widetilde{\gamma}_{i}}}{12h}\left[-u_{j+2}^{\gamma_{i}}+8u_{j+1}^{\gamma_{i}}-8u_{j-1}^{\gamma_{i}}+u_{j-2}^{\gamma_{i}}\right]-
      g_{j}^{\widetilde{\gamma}_{i}}u_{j}^{\gamma_{i}},
       \end{equation}
   for $j=2,3,...,M-2$, and
    \begin{equation}\label{79a}
    Lu_{j}^{\widetilde{\gamma}_{i}}=[q(t)u_{2x}-p(t)u_{x}-g(x,t)u]|_{(x_{j},t_{\widetilde{\gamma}_{i}})},
    \end{equation}
   where $\widetilde{\gamma}_{i}\in\{i+\frac{1}{2}+\alpha,\text{\,\,}i+1+\alpha\}$ and $u_{j}^{\gamma_{i}}=u_{j}^{\alpha_{i}},$ $u_{j}^{\theta_{i}}$,
   for $0\leq i\leq N$. So, it holds
   \begin{equation}\label{80}
    \| L_{h}u^{\gamma_{i}}-Lu^{\widetilde{\gamma}_{i}}\|_{L^{2}}\leq A_{g}^{pq}\left[7\alpha(1+2\alpha)(\frac{1}{2}+2\alpha)k^{2}+\frac{53}{270}(12
    +\alpha(\frac{1}{2}+\alpha)(\frac{1}{2}+2\alpha)k^{2})h^{4}\right]\||u|\|_{\mathcal{C}^{6,3}_{D}},
   \end{equation}
   where
   \begin{equation}\label{80a}
     A_{g}^{pq}=\frac{\sqrt{3}}{12}\max\{\|q\|_{\mathcal{C}^{0}_{\Gamma}},\|p\|_{\mathcal{C}^{0}_{\Gamma}},\||g|\|_{\mathcal{C}^{6,3}_{D}}\},
   \end{equation}
   where $\Gamma=[0,T]$, $\|v\|_{\mathcal{C}^{0}_{\Gamma}}=\underset{0\leq i\leq N}{\max}|v^{i}|$ and $\||\cdot|\|_{\mathcal{C}^{6,3}_{D}}$ is defined
   in relation $(\ref{2aa})$.
   \end{lemma}

   \begin{proof}
   Considering relations $(\ref{40})$-$(\ref{41})$, $(\ref{55})$, $(\ref{57})$-$(\ref{58})$ and $(\ref{79})$-$(\ref{79a})$, simple computations give
   \begin{equation}\label{81}
    Lu_{j}^{\widetilde{\gamma}_{i}}-L_{h}u_{j}^{\gamma_{i}}=q^{\widetilde{\gamma}_{i}}\psi_{l,j}^{i}-p^{\widetilde{\gamma}_{i}}\psi_{r,j}^{i}
    +\alpha(\frac{1}{2}+\alpha)k^{2}g_{j}^{\widetilde{\gamma}_{i}}H_{m_{rl},j}^{i},
   \end{equation}
   where $(l,r,\widetilde{\gamma}_{i},m_{rl})\in\{(1,2,i+\frac{1}{2}+\alpha,1),(3,4,i+1+\alpha,2)\}$, $\psi_{l,j}^{i}$ are given by equations $(\ref{42})$
    and $(\ref{60})$, $\psi_{r,j}^{i}$ are given by equations $(\ref{43})$ and $(\ref{61})$. Working as in Section $\ref{sec2}$ to approximate the terms
   $u_{2x}(x_{j},t_{i})$ and $u_{x}(x_{j},t_{i})$, one easily shows that
     \begin{equation*}
       \partial_{2x}H_{l,j}^{i}=\frac{1}{12h^{2}}\left[-H_{l,j+2}^{i}+16H_{l,j+1}^{i}-30H_{l,j}^{i}+16H_{l,j-1}^{i}-H_{l,j-2}^{i}\right]+
       \end{equation*}
        \begin{equation}\label{82}
        \frac{h^{4}}{270}\left[18\partial_{6x}H_{l}^{i}(\epsilon_{l,j}^{20})+18\partial_{6x}H_{l}^{i}(\epsilon_{l,j}^{21})-
      \partial_{6x}H_{l}^{i}(\epsilon_{l,j}^{22})-\partial_{6x}H_{l}^{i}(\epsilon_{l,j}^{23})\right],
      \end{equation}
      and
       \begin{equation*}
      \partial_{x}H_{r,j}^{i}=\frac{1}{12h}\left[-H_{r,j+2}^{i}+8H_{r,j+1}^{i}-8H_{r,j-1}^{i}+H_{r,j-2}^{i}\right]+
       \end{equation*}
        \begin{equation}\label{83}
        \frac{h^{4}}{180}\left[4\partial_{5x}H_{r}^{i}(\epsilon_{r,j}^{20})+4\partial_{5x}H_{r}^{i}(\epsilon_{r,j}^{21})-
      \partial_{5x}H_{r}^{i}(\epsilon_{r,j}^{22})-\partial_{5x}H_{r}^{i}(\epsilon_{r,j}^{23})\right],
      \end{equation}
      where
       \begin{equation*}
       \epsilon_{l,j}^{22},\epsilon_{r,j}^{22}\in(x_{j-1},x_{j}),\text{\,\,}\epsilon_{l,j}^{23},\epsilon_{r,j}^{23}\in(x_{j},x_{j+1}),
       \text{\,\,}\epsilon_{l,j}^{20},\epsilon_{r,j}^{20}\in(x_{j-2},x_{j}),\text{\,\,} \epsilon_{l,j}^{21},\epsilon_{r,j}^{21}\in(x_{j},x_{j+2}),
       \end{equation*}
       where $\partial_{mx}H_{s}^{i}$, for $m,s=1,2$, denote $\frac{\partial^{m}H_{s}^{i}}{\partial x^{m}}.$ The functions $H_{s}^{i}$ are defined by
      relations $(\ref{28d})$ and $(\ref{56})$.\\

       Plugging equations $(\ref{42})$, $(\ref{60})$ and $(\ref{82})$, straightforward calculations provide
       \begin{equation*}
      \psi_{l,j}^{i}=-\frac{\alpha(\frac{1}{2}+\alpha)k^{2}}{12}\partial_{2x}H_{l,j}^{i}+\frac{h^{4}}{270}
      \left[18\left(u_{6x}^{\widetilde{\gamma}_{i}}(\epsilon_{l,j}^{1})+\frac{\alpha(\frac{1}{2}+\alpha)k^{2}}{12}\partial_{6x}H_{l}^{i}(\epsilon_{l,j}^{20})\right)
      +18\left(\frac{\alpha(\frac{1}{2}+\alpha)k^{2}}{12}\partial_{6x}H_{l}^{i}(\epsilon_{l,j}^{21})\right.\right.
       \end{equation*}
        \begin{equation}\label{84}
      \left.\left.+u_{6x}^{\widetilde{\gamma_{i}}}(\epsilon_{l,j}^{2})\right)-\left(u_{6x}^{\widetilde{\gamma_{i}}}(\epsilon_{l,j}^{3})+\frac{\alpha(\frac{1}{2}+\alpha)
      k^{2}}{12}\partial_{6x}H_{l}^{i}(\epsilon_{l,j}^{22})\right)-\left(u_{6x}^{\widetilde{\gamma_{i}}}(\epsilon_{l,j}^{4})+\frac{\alpha(\frac{1}{2}+\alpha)k^{2}}{12}
      \partial_{6x}H_{l}^{i}(\epsilon_{l,j}^{23})\right)\right],
      \end{equation}
      where
        \begin{equation*}
       \epsilon_{l,j}^{1}\in\{\epsilon_{j}^{6},\epsilon_{j}^{12}\},\text{\,\,}\epsilon_{l,j}^{2}\in\{\epsilon_{j}^{7},\epsilon_{j}^{13}\},
       \text{\,\,}\epsilon_{l,j}^{3}\in\{\epsilon_{j}^{4},\epsilon_{j}^{14}\},\text{\,\,}\epsilon_{l,j}^{4}\in\{\epsilon_{j}^{5},\epsilon_{j}^{15}\}.
       \end{equation*}
      Analogously, considering equations $(\ref{43})$, $(\ref{61})$ and $(\ref{83})$, simple computations yield
       \begin{equation*}
      \psi_{r,j}^{i}=-\frac{\alpha(\frac{1}{2}+\alpha)k^{2}}{12}\partial_{x}H_{r,j}^{i}+\frac{h^{4}}{180}
      \left[4\left(u_{5x}^{\widetilde{\gamma_{i}}}(\epsilon_{r,j}^{1})+\frac{\alpha(\frac{1}{2}+\alpha)k^{2}}{12}\partial_{5x}H_{r}^{i}(\epsilon_{r,j}^{20})\right)
      +4\left(\frac{\alpha(\frac{1}{2}+\alpha)k^{2}}{12}\partial_{5x}H_{r}^{i}(\epsilon_{r,j}^{21})\right.\right.
       \end{equation*}
        \begin{equation}\label{85}
      \left.\left.+u_{5x}^{\widetilde{\gamma_{i}}}(\epsilon_{r,j}^{2})\right)-\left(u_{5x}^{\widetilde{\gamma_{i}}}(\epsilon_{r,j}^{3})+\frac{\alpha(\frac{1}{2}+\alpha)
      k^{2}}{12}\partial_{5x}H_{r}^{i}(\epsilon_{r,j}^{22})\right)-\left(u_{5x}^{\widetilde{\gamma_{i}}}(\epsilon_{r,j}^{4})+\frac{\alpha(\frac{1}{2}+\alpha)k^{2}}{12}
      \partial_{5x}H_{r}^{i}(\epsilon_{r,j}^{23})\right)\right],
      \end{equation}
      where
        \begin{equation*}
       \epsilon_{r,j}^{1}=\epsilon_{j}^{10},\epsilon_{j}^{18},\text{\,\,}\epsilon_{r,j}^{2}=\epsilon_{j}^{11},\epsilon_{j}^{19},
       \text{\,\,}\epsilon_{r,j}^{3}=\epsilon_{j}^{8},\epsilon_{j}^{16},\text{\,\,}\epsilon_{r,j}^{4}=\epsilon_{j}^{9},\epsilon_{j}^{17},
       \end{equation*}
     with $(l,r,\widetilde{\gamma}_{i})\in\{(1,2,i+\frac{1}{2}+\alpha),(3,4,i+1+\alpha)\}$.\\

     Now, using equations $(\ref{28d})$ and $(\ref{56})$, it is not hard to show that
       \begin{equation}\label{88}
     \||\partial_{mx}H_{s}^{i}|\|_{\mathcal{C}_{D}^{6,3}}\leq(\frac{1}{2}+2\alpha)\||u|\|_{\mathcal{C}_{D}^{6,3}},\text{\,\,\,for\,\,\,}s=1,2.
      \end{equation}
      Substituting equations $(\ref{28d})$ and $(\ref{56})$ into equations $(\ref{84})$ and $(\ref{85})$, respectively, performing simple calculations and
      utilizing estimate equations $(\ref{88})$, it is not difficult to observe that
       \begin{equation}\label{89}
     \||\psi_{l}^{i}|\|_{\mathcal{C}_{D}^{6,3}}\leq \frac{1}{12}\left\{\alpha(\frac{1}{2}+\alpha)(\frac{1}{2}+2\alpha)k^{2}+\frac{19}{135}
     [12+\alpha(\frac{1}{2}+\alpha)(\frac{1}{2}+2\alpha)k^{2}]h^{4}\right\}\||u|\|_{\mathcal{C}_{D}^{6,3}},\text{\,\,\,for\,\,\,}l=1,3,
      \end{equation}
      and
      \begin{equation}\label{90}
     \||\psi_{r}^{i}|\|_{\mathcal{C}_{D}^{6,3}}\leq \frac{1}{12}\left\{\alpha(\frac{1}{2}+\alpha)(\frac{1}{2}+2\alpha)k^{2}+\frac{1}{18}
     [12+\alpha(\frac{1}{2}+\alpha)(\frac{1}{2}+2\alpha)k^{2}]h^{4}\right\}\||u|\|_{\mathcal{C}_{D}^{6,3}},\text{\,\,\,for\,\,\,}r=2,4.
      \end{equation}
      Taking the square of both sides of equation $(\ref{81})$ and applying the Cauchy-Schwarz inequality, it holds
   \begin{equation*}
    [L_{h}u^{\gamma_{i}}-Lu^{\widetilde{\gamma}_{i}}]^{2}=3[(q^{\widetilde{\gamma}_{i}}\psi_{l,j}^{i})^{2}+(p^{\widetilde{\gamma}_{i}}\psi_{r,j}^{i})^{2}
    +\alpha^{2}(\frac{1}{2}+\alpha)^{2}k^{4}(g_{j}^{\widetilde{\gamma}_{i}}H_{m_{rl},j}^{i})^{2}]\leq3\left[\|q\|_{\mathcal{C}_{\Gamma}^{0}}^{2}
    \||\psi_{l}|\|_{\mathcal{C}_{D}^{6,3}}^{2}+\right.
   \end{equation*}
   \begin{equation*}
    \left.\|p\|_{\mathcal{C}_{\Gamma}^{0}}^{2}\||\psi_{r}|\|_{\mathcal{C}_{D}^{6,3}}^{2}+\alpha^{2}(\frac{1}{2}+\alpha)^{2}k^{4}
    \||g|\|_{\mathcal{C}_{D}^{6,3}}^{2}\||H_{m_{rl}}|\|_{\mathcal{C}_{D}^{6,3}}^{2}\right].
   \end{equation*}
   Substituting estimates $(\ref{88})$-$(\ref{89})$ into this inequality, summing up from $j=1,2,...,M-1$, and multiplying both sides of the obtained
   estimate by $h$, direct calculations result in
   \begin{equation*}
    \|L_{h}u^{\gamma_{i}}-Lu^{\widetilde{\gamma}_{i}}\|^{2}_{L^{2}}\leq\frac{3}{144}\max\{\|q\|_{\mathcal{C}_{\Gamma}^{0}}^{2},
    \|p\|_{\mathcal{C}_{\Gamma}^{0}}^{2},\||g|\|_{\mathcal{C}_{D}^{6,3}}^{2}\}\left[7\alpha(\frac{1}{2}+\alpha)(1+2\alpha)k^{2}+\right.
   \end{equation*}
   \begin{equation*}
    \left.\frac{53}{270}(12+\alpha(\frac{1}{2}+\alpha)(\frac{1}{2}+2\alpha)k^{2})h^{4}\right]^{2}\||u|\|_{\mathcal{C}_{D}^{6,3}}^{2}.
   \end{equation*}
   Taking the square root of both sides of this estimate to get relation $(\ref{80})$. This ends the proof of Lemma $\ref{l2}$.
   \end{proof}

   \begin{lemma}\label{l3}
   Let $u\in\mathcal{C}^{2,0}_{D}$, be a function defined on $D=[0,L_{1}]\times[0,T]$, such that $u(0,t)=u(L_{1},t)=0,$ for every $t\in[0,T]$. Suppose
    $U(t)\in\mathcal{Y}_{h}(t)=\{U_{j}(t);\text{\,}j=0,1,...,M\}$ be a grid function satisfying $U_{j}(t)=u(x_{j},t),$ for $j=0,1,...,M$. So, it holds
   \begin{equation*}
   (-Lu(t),u(t))\geq\gamma\|u_{x}(t)\|_{L^{2}}^{2}+\beta\|u(t)\|_{L^{2}}^{2}\text{\,\,\,and\,\,\,}
    (-L_{h}U(t),U(t))\geq\gamma\|\delta_{x}U(t)\|_{L^{2}}^{2}+\beta\|U(t)\|_{L^{2}}^{2},
   \end{equation*}
   for any $t\in[0,T]$, where $\gamma=\underset{t\in[0,T]}{\inf}q(t)>0$ and $\beta=\underset{(x,t)\in D}{\inf}g(x,t)\geq0$.
   \end{lemma}

   \begin{proof}
    In this proof, we should prove that the linear operator: $-Lu=-q(t)u_{2x}+p(t)u_{x}+g(x,t)u,$ satisfies $(-Lu,u)\geq\gamma\|u_{x}\|_{L^{2}}^{2}
    +\beta\|u\|_{L^{2}}^{2},$ and then use the definition of the discrete $L^{2}$-norm given in relation $(\ref{2aa})$ to conclude. Since
     $\delta_{x}u_{j-\frac{1}{2}}(t)=\frac{u_{j}(t)-u_{j-1}(t)}{h}$ and $\delta_{x}^{2}u_{j}(t)=\frac{u_{j+1}(t)-2u_{j}(t)
    +u_{j-1}(t)}{h^{2}}=\frac{\delta_{x}u_{j+\frac{1}{2}}(t)-\delta_{x}u_{j-\frac{1}{2}}(t)}{h}$, applying the Taylor series expansion, it is easy to show
      that $u_{x}(x_{j},t)=\delta_{x}u_{j-\frac{1}{2}}(t)+O(h)$ and $u_{2x}(x_{j},t)=\delta_{x}^{2}u_{j}(t)+O(h^{2})$. Using the conditions
      $u_{0}(t)=u_{M}(t)=0$, for every $t\in[0,T]$ together with the summation by parts and the equality $a(a-b)=\frac{1}{2}[(a-b)^{2}+a^{2}-b^{2}]$, for
     every real numbers $a$ and $b$, simple calculations provide
    \begin{equation*}
    (-Lu(t),u(t))=-h\underset{j=1}{\overset{M-1}\sum}[q(t)(\delta_{x}^{2}u_{j}(t)+O(h^{2}))-p(t)(\delta_{x}u_{j-\frac{1}{2}}(t)+O(h))
    -g_{j}(t)u_{j}(t)]u_{j}(t)=
   \end{equation*}
   \begin{equation*}
    -hq(t)\underset{j=1}{\overset{M-1}\sum}(\delta_{x}^{2}u_{j}(t))u_{j}(t)+O(h^{2})+hp(t)\underset{j=1}{\overset{M-1}\sum}
    (\delta_{x}u_{j-\frac{1}{2}}(t))u_{j}(t)+O(h)+h\underset{j=1}{\overset{M-1}\sum}g_{j}(t)(u_{j}(t))^{2}=
   \end{equation*}
   \begin{equation*}
    -q(t)\left[(\delta_{x}u_{M-\frac{1}{2}}(t))u_{M-1}(t)-(\delta_{x}u_{\frac{1}{2}}(t))u_{1}(t)-\underset{j=1}{\overset{M-2}\sum}
    (\delta_{x}u_{j+\frac{1}{2}}(t))^{2}\right]+O(h^{2})+\frac{1}{2}p(t)\left[h^{2}\underset{j=1}{\overset{M-1}\sum}
    (\delta_{x}u_{j-\frac{1}{2}}(t))^{2}\right.
   \end{equation*}
   \begin{equation*}
    \left.+(u_{M-1}(t))^{2}-(u_{0}(t))^{2}\right]+O(h)+h\underset{j=1}{\overset{M-1}\sum}g_{j}(t)(u_{j}(t))^{2}=
    q(t)h\underset{j=0}{\overset{M-1}\sum}(\delta_{x}u_{j+\frac{1}{2}}(t))^{2}+O(h^{2})
   \end{equation*}
   \begin{equation}\label{91a}
    +\frac{1}{2}p(t)h^{2}\underset{j=1}{\overset{M}\sum}(\delta_{x}u_{j-\frac{1}{2}}(t))^{2}+O(h)+h\underset{j=1}{\overset{M-1}\sum}g_{j}(t)(u_{j}(t))^{2}.
   \end{equation}
   The last equality follows from $h\delta_{x}u_{M-\frac{1}{2}}(t)=-u_{M-1}(t)$ and $\delta_{x}u_{\frac{1}{2}}(t)=\frac{1}{h}u_{1}(t)$, since
    $u_{M}(t)=u_{0}(t)=0$. For small values of $h$, it holds
   \begin{equation}\label{91b}
    h\underset{j=0}{\overset{M-1}\sum}(\delta_{x}u_{j+\frac{1}{2}}(t))^{2}+O(h^{2})\approx h\underset{j=1}{\overset{M-1}\sum}(u_{x,j}(t))^{2}.
   \end{equation}
   Since $q(t)\geq\gamma>0$, $p(t)\geq0$, for any $t\in[0,T]$ and $\beta=\underset{(x,t)\in D}{\inf}g(x,t)\geq0,$ using this and substituting approximation
    $(\ref{91b})$ into relation $(\ref{91a})$ to obtain
   \end{proof}
    \begin{equation*}
    (-Lu(t),u(t))\geq\gamma h\underset{j=1}{\overset{M-1}\sum}(u_{x,j}(t))^{2}+\beta h\underset{j=1}{\overset{M-1}\sum}(u_{j}(t))^{2}=
    \gamma\|u_{x}(t)\|_{L^{2}}^{2}+\beta\|u(t)\|_{L^{2}}^{2}.
   \end{equation*}
   This ends the proof of the first estimate in Lemma $\ref{l3}$. Furthermore, since $u_{j}(t)=U_{j}(t)$, for $j=0,1,...,M$, for $h$ sufficiently small,
   neglecting the infinitesimal terms $O(h^{2})$ and $O(h)$, equation $(\ref{91a})$ implies
   \begin{equation*}
    (-Lu(t),u(t))\geq\gamma h\underset{j=1}{\overset{M-1}\sum}(\delta_{x}U_{j+\frac{1}{2}}(t))^{2}+\beta h\underset{j=1}{\overset{M-1}
    \sum}(U_{j}(t))^{2}=\gamma\|\delta_{x}U(t)\|_{L^{2}}^{2}+\beta\|U(t)\|_{L^{2}}^{2}.
   \end{equation*}
   The proof of Lemma $\ref{l3}$ is completed thanks to the definition of the discrete $L^{2}$-norm given in relation $(\ref{2aa})$.\\

   The following Lemmas (namely Lemmas $\ref{l4}$-$\ref{l5}$ taken in \cite{aaa}) help in proving Lemmas $\ref{l6}$.
   \begin{lemma}\cite{aaa}\label{l4}
    For every $i=1,2,3,...,$ and any $\lambda$ satisfying $0<\lambda<1$, the following inequalities are satisfied
    \begin{equation*}
    \frac{1}{2}<\mathcal{Z}_{i}<\frac{1}{2-\lambda}
   \end{equation*}
   where
   \begin{equation*}
    \mathcal{Z}_{i}=\frac{(i+\alpha)^{2-\lambda}-(i-1+\alpha)^{2-\lambda}-(2-\lambda)(i-1+\alpha)^{1-\lambda}}{(2-\lambda)[(i+\alpha)^{1-\lambda}
    -(i-1+\alpha)^{1-\lambda}]}.
   \end{equation*}
   \end{lemma}

   \begin{lemma}\cite{aaa}\label{l5}
   Suppose $\lambda$ be a positive number which is less than one. So, it holds
    \begin{equation*}
    \frac{1}{2-\lambda}\left[(l+\alpha)^{2-\lambda}-(l-1+\alpha)^{2-\lambda}\right]-\frac{1}{2}\left[(l+\alpha)^{1-\lambda}+(l-1+\alpha)^{1-\lambda}\right]>0,
   \end{equation*}
   for every integer $l\geq1$.
   \end{lemma}

   \begin{lemma}\label{l6}
   Let $\lambda$ be a positive number such that $0<\lambda<\frac{2}{3}$. For any positive integer $i$, we introduce the generalized sequences
    $(a_{i+\frac{1}{2},l}^{\lambda\alpha})_{l}$ and $(a_{i+1,l}^{\lambda\alpha})_{l}$ with step size equals $\frac{1}{2}$ (that is,
    $l=\frac{1}{2},1,\frac{3}{2},2,...$) defined by:

   \begin{equation}\label{94}
    a_{\frac{1}{2},\frac{1}{2}}^{\lambda\alpha}=\widetilde{f}_{\frac{1}{2},0}^{\lambda\alpha},\text{\,\,\,and\,\,\,for\,\,\,}i\geq1,
    \text{\,\,\,\,}a_{i+\frac{1}{2},\frac{1}{2}}^{\lambda\alpha}=\widetilde{\text{\,\.{f}}}_{i+\frac{1}{2},0}^{\lambda\alpha},
   \end{equation}
   \begin{equation}\label{95}
    a_{i+\frac{1}{2},l}^{\lambda\alpha}=\begin{array}{c}
                                          \left\{
                                            \begin{array}{ll}
                                              \widetilde{d}_{i+\frac{1}{2},l-1}^{\lambda\alpha}-\widetilde{f}_{i+\frac{1}{2},l-1}^{\lambda\alpha},
                                              & \hbox{if $l=1,2,3,...,i$,} \\
                                              \text{\,}\\
                                              \widetilde{f}_{i+\frac{1}{2},l-\frac{1}{2}}^{\lambda\alpha}, & \hbox{if
                                                $l=\frac{1}{2},\frac{3}{2},\frac{5}{2},...,i-\frac{1}{2}$,} \\
                                               \text{\,}\\
                                              \widetilde{f}_{i+\frac{1}{2},i-1}^{\lambda\alpha}+\widetilde{f}_{i+\frac{1}{2},i}^{\lambda\alpha},
                                              & \hbox{if $l=i+\frac{1}{2}$,} \\
                                            \end{array}
                                          \right.
                                        \end{array}
   \end{equation}
    where the terms $\widetilde{f}_{\frac{1}{2},0}^{\lambda\alpha}$, $\widetilde{\text{\,\.{f}}}_{i+\frac{1}{2},0}^{\lambda\alpha}$,
    $\widetilde{f}_{i+\frac{1}{2},s}^{\lambda\alpha}$ and $\widetilde{d}_{i+\frac{1}{2},s}^{\lambda\alpha}$ (for $s=0,1,2,...,i$) are given by equations
    $(\ref{11})$, $(\ref{12})$ and $(\ref{22a})$. Furthermore,

    \begin{equation}\label{96a}
    a_{1,\frac{1}{2}}^{\lambda\alpha}=\widetilde{d}_{1,0}^{\lambda\alpha}-\widetilde{f}_{1,0}^{\lambda\alpha}\text{\,\,\,and\,\,\,}a_{1,1}^{\lambda\alpha}=
    \widetilde{f}_{1,0}^{\lambda\alpha}+\widetilde{f}_{1,1}^{\lambda\alpha},
   \end{equation}
   and for $i\geq1$,
   \begin{equation}\label{96}
    a_{i+1,l+\frac{1}{2}}^{\lambda\alpha}=\begin{array}{c}
                                          \left\{
                                            \begin{array}{ll}
                                              \widetilde{d}_{i+1,l}^{\lambda\alpha}-\widetilde{f}_{i+1,l}^{\lambda\alpha},
                                              & \hbox{if $l=0,1,2,3,...,i$,} \\
                                              \text{\,}\\
                                              \widetilde{f}_{i+1,l-\frac{1}{2}}^{\lambda\alpha}, & \hbox{if
                                                $l=\frac{1}{2},\frac{3}{2},\frac{5}{2},...,i-\frac{1}{2}$,} \\
                                               \text{\,}\\
                                              \widetilde{f}_{i+1,i}^{\lambda\alpha}+\widetilde{f}_{i+1,i+1}^{\lambda\alpha},
                                              & \hbox{if $l=i+\frac{1}{2}$,} \\
                                            \end{array}
                                          \right.
                                        \end{array}
   \end{equation}
    where the terms $\widetilde{d}_{i+1,s}^{\lambda\alpha}$ and $\widetilde{f}_{i+1,s}^{\lambda\alpha}$ (for $i\geq0$ and $s=0,1,2,...,i+1$) are defined by
    relation $(\ref{53})$. Thus, for $i\geq1$ the following estimates are satisfied, for $s=\frac{1}{2},1$,
    \begin{equation}\label{97}
     \widetilde{f}_{\frac{1}{2},0}^{\lambda\alpha},\text{\,}\widetilde{f}_{1,0}^{\lambda\alpha},\text{\,}\widetilde{d}_{1,0}^{\lambda\alpha},
     \text{\,}\widetilde{f}_{1,1}^{\lambda\alpha}, \widetilde{\text{\,\.{f}}}_{i+\frac{1}{2},0}^{\lambda\alpha},\text{\,}
      \text{\,}\widetilde{f}_{i+s,j}^{\lambda\alpha},\text{\,}\widetilde{d}_{i+s,j}^{\lambda\alpha},
     \text{\,}\widetilde{d}_{i+s,j}^{\lambda\alpha}-\widetilde{f}_{i+s,j}^{\lambda\alpha}>0,\text{\,\,}j=0,1,2,...,i\text{\,\,}(\text{\,resp.,\,\,}i+1),
   \end{equation}
    \begin{equation}\label{98}
     \text{\,}\widetilde{f}_{i+s,j-1}^{\lambda\alpha}+\widetilde{f}_{i+s,j}^{\lambda\alpha}-\widetilde{d}_{i+s,j}^{\lambda\alpha}<0,
      \text{\,}j=1,2,...,i\text{\,}(\text{\,resp.\,\,}i+1),\text{\,\,}2\widetilde{f}_{i+s,j}^{\lambda\alpha}-\widetilde{d}_{i+s,j}^{\lambda\alpha}>0,
      \text{\,\,}j=0,1,2,...,i\text{\,}(\text{\,resp.,\,\,}i+1).
   \end{equation}
    Furthermore,
     \begin{equation}\label{99}
     a_{i+s,l}^{\lambda\alpha}<a_{i+s,l+\frac{1}{2}}^{\lambda\alpha},\text{\,\,}l=\frac{1}{2},1,\frac{3}{2},...,i\text{\,\,}
     (\text{\,resp.,\,\,\,}i+\frac{1}{2}),
   \end{equation}
     and
    \begin{equation}\label{100}
     a_{i+s,l}^{\lambda\alpha}>\frac{(2-3\lambda)(1-\lambda)}{2(2-\lambda)}(i+s+\alpha-l)^{-\lambda},\text{\,\,}
     l=\frac{1}{2},1,\frac{3}{2},...,i+\frac{1}{2}\text{\,\,}(\text{\,resp.,\,\,\,}i+1).
   \end{equation}
   \end{lemma}

    \begin{proof}
     It follows from relations $(\ref{22a})$ and $(\ref{53})$ that
     \begin{equation*}
     \widetilde{f}_{\frac{1}{2},0}^{\lambda\alpha}=(\frac{1}{2}+\alpha)^{1-\lambda}>0,\text{\,}\widetilde{\text{\,\.{f}}}_{i+\frac{1}{2},0}^{\lambda\alpha}
     =(i+\frac{1}{2}+\alpha)^{1-\lambda}-(i+\alpha)^{1-\lambda}>0,\text{\,\,}\widetilde{f}_{i+\frac{1}{2},i}^{\lambda\alpha}=
      \widetilde{f}_{i+1,i+1}^{\lambda\alpha}=\alpha^{1-\lambda}>0.
   \end{equation*}
    In addition, it comes from equations $(\ref{11})$, $(\ref{12})$ and $(\ref{53})$ that
    \begin{equation*}
     0<\widetilde{d}_{1,0}^{\lambda\alpha}=(1+\alpha)^{1-\lambda}-\alpha^{1-\lambda},\text{\,\,} 0<\widetilde{d}_{i+1,l}^{\lambda\alpha}=
     (i+1+\alpha-l)^{1-\lambda}-(i+\alpha-l)^{1-\lambda}=\widetilde{d}_{i+\frac{1}{2},l-1}^{\lambda\alpha},
    \text{\,\,}l=1,2,3,...,i,
   \end{equation*}
    \begin{equation}\label{101}
     \widetilde{f}_{i+1,l}^{\lambda\alpha}=\frac{2}{2-\lambda}[(i+1+\alpha-l)^{2-\lambda}-(i+\alpha-l)^{2-\lambda}]-
    \frac{1}{2}[(i+1+\alpha-l)^{1-\lambda}+3(i+\alpha-l)^{1-\lambda}]=\widetilde{f}_{i+\frac{1}{2},l-1}^{\lambda\alpha},\text{\,\,}l=1,2,...,i+1,
   \end{equation}
    \begin{equation}\label{102}
    \widetilde{d}_{i+1,0}^{\lambda\alpha}=(i+1+\alpha)^{1-\lambda}-(i+\alpha)^{1-\lambda}>0.
   \end{equation}
    We should show that $\widetilde{f}_{i+1,l}^{\lambda\alpha}>0$ and $\widetilde{d}_{i+1,l}^{\lambda\alpha}-\widetilde{f}_{i+1,l}^{\lambda\alpha}>0$,
    for $0\leq l\leq i+1$, and use this to end the proof of estimates  given in relation $(\ref{97})$.
    \begin{equation*}
    2\mathcal{Z}_{i+1-l}-\frac{1}{2}=2\frac{(i+1+\alpha-l)^{2-\lambda}-(i+\alpha-l)^{2-\lambda}-(2-\lambda)(i+\alpha-l)^{1-\lambda}}{(2-\lambda)
    [(i+1+\alpha-l)^{1-\lambda}-(i+\alpha-l)^{1-\lambda}]}-\frac{1}{2}=
   \end{equation*}
    \begin{equation*}
    \frac{2}{2-\lambda}\frac{(i+1+\alpha-l)^{2-\lambda}-(i+\alpha-l)^{2-\lambda}}{(i+1+\alpha-l)^{1-\lambda}-(i+\alpha-l)^{1-\lambda}}
    -\frac{(i+1+\alpha-l)^{1-\lambda}+3(i+\alpha-l)^{1-\lambda}}{2[(i+1+\alpha-l)^{1-\lambda}-(i+\alpha-l)^{1-\lambda}]}=
   \end{equation*}
    \begin{equation}\label{106}
    [(i+1+\alpha-l)^{1-\lambda}-(i+\alpha-l)^{1-\lambda}]^{-1}\widetilde{f}_{i+1,l}^{\lambda\alpha}.
   \end{equation}
    But it comes from Lemma $\ref{l4}$ that $\mathcal{Z}_{i+1-l}>\frac{1}{2}$, this implies $2\mathcal{Z}_{i+1-l}-\frac{1}{2}>0$. Thus,
    $\widetilde{f}_{1,0}^{\lambda\alpha}>0$, $\widetilde{f}_{1,1}^{\lambda\alpha}>0$, $\widetilde{f}_{i+1,0}^{\lambda\alpha}>0$ and
    $\widetilde{f}_{i+1,l}^{\lambda\alpha}=\widetilde{f}_{i+\frac{1}{2},l-1}^{\lambda\alpha}>0$, for $l=1,2,...,i+1$.\\

    In a similar way, it is easy to show that
    \begin{equation}\label{105}
    \frac{3}{2}-2\mathcal{Z}_{i+1-l}=[(i+1+\alpha-l)^{1-\lambda}-(i+\alpha-l)^{1-\lambda}]^{-1}(\widetilde{d}_{i+1,l}^{\lambda\alpha}-
    \widetilde{f}_{i+1,l}^{\lambda\alpha}).
   \end{equation}
    Since $\frac{3}{2}-\frac{2}{2-\lambda}=\frac{2-3\lambda}{2(2-\lambda)}>0$, for $0<\lambda<\frac{2}{3}$. This fact together with from Lemma
     $\ref{l4}$ provide $\frac{3}{2}-2\mathcal{Z}_{i+1-l}>\frac{2}{2-\lambda}-2\mathcal{Z}_{i+1-l}>0$. Hence,
     \begin{equation*}
    \widetilde{d}_{i+1,0}^{\lambda\alpha}-\widetilde{f}_{i+1,0}^{\lambda\alpha}>0,\text{\,\,and\,\,}\widetilde{d}_{i+1,l}^{\lambda\alpha}-
    \widetilde{f}_{i+1,l}^{\lambda\alpha}=\widetilde{d}_{i+\frac{1}{2},l-1}^{\lambda\alpha}-\widetilde{f}_{i+\frac{1}{2},l-1}^{\lambda\alpha}>0,\text{\,\,}
     l=1,2,...,i+1.
   \end{equation*}
    This ends the proof of relation $(\ref{97})$. Now, we should prove the inequalities in $(\ref{98})$.\\

    Firstly, it is not hard to observe that
     \begin{equation*}
    -\lambda(1-\lambda)\int_{0}^{1}\int_{0}^{1}(i+1+y_{1}+y_{2}+\alpha-j)^{-1-\lambda}dy_{1}dy_{2}=(i+2+\alpha-j)^{1-\lambda}
    -2(i+1+\alpha-j)^{1-\lambda}+(i+\alpha-j)^{1-\lambda}.
   \end{equation*}
    Using this together with equation $(\ref{53})$, simple calculations give
    \begin{equation*}
    \widetilde{f}_{i+1,j-1}^{\lambda\alpha}+\widetilde{f}_{i+1,j}^{\lambda\alpha}-\widetilde{d}_{i+1,j}^{\lambda\alpha}=
    \frac{2}{2-\lambda}[(i+2+\alpha-j)^{2-\lambda}-(i+\alpha-j)^{2-\lambda}]-\frac{1}{2}[(i+2+\alpha-j)^{1-\lambda}-2(i+1+\alpha-j)^{1-\lambda}+
   \end{equation*}
    \begin{equation*}
    (i+\alpha-j)^{1-\lambda}]-4(i+1+\alpha-j)^{1-\lambda}=2\int_{0}^{2}(i+z_{1}+\alpha-j)^{1-\lambda}dz_{1}+\frac{\lambda(1-\lambda)}{2}\int_{0}^{1}
    \int_{0}^{1}(i+y_{1}+y_{2}+\alpha-j)^{-1-\lambda}dy_{1}dy_{2}
   \end{equation*}
    \begin{equation*}
    -4(i+1+\alpha-j)^{1-\lambda}\leq 2(i+2+\alpha-j)^{1-\lambda}-\frac{7}{2}(i+1+\alpha-j)^{1-\lambda}+\frac{\lambda(1-\lambda)}{2}
     (i+2+\alpha-j)^{-1-\lambda}-\frac{1}{2}(i+1+\alpha-j)^{1-\lambda}<0,
   \end{equation*}
    since $2(i+2+\alpha-j)^{1-\lambda}-\frac{7}{2}(i+1+\alpha-j)^{1-\lambda}$, $\frac{\lambda(1-\lambda)}{2}(i+2+\alpha-j)^{-1-\lambda}
    -\frac{1}{2}(i+1+\alpha-j)^{1-\lambda}<0$. Indeed,
    \begin{equation*}
    \frac{4}{7}\left(\frac{i+2+\alpha-j}{i+1+\alpha-j}\right)^{1-\lambda}=\frac{4}{7}\left(1+\frac{1}{i+1+\alpha-j}\right)^{1-\lambda}<
    \frac{4}{7}\left(1+\frac{1-\lambda}{i+1+\alpha-j}\right)\leq \frac{4}{7}\left(1+\frac{1-\lambda}{1+\alpha}\right)<1,
   \end{equation*}
   since $j\leq i$, $\alpha=1-\lambda$ and $0<\lambda<\frac{2}{3}$. So, estimate $2(i+2+\alpha-j)^{1-\lambda}-\frac{7}{2}(i+1+\alpha-j)^{1-\lambda}<0,$
   is holds for every $j=0,1,...,i$. On the other hand, $\frac{\lambda(1-\lambda)}{2}(i+2+\alpha-j)^{-1-\lambda}\leq\frac{\lambda(1-\lambda)}{2}
   (2+\alpha)^{-1-\lambda}<\frac{1}{2}$ and $\frac{1}{2}(i+1+\alpha-j)^{1-\lambda}>\frac{1}{2}(1+\alpha)^{1-\lambda}>\frac{1}{2}$, for $j=0,1,...,i$. Then
    $\frac{\lambda(1-\lambda)}{2}(i+2+\alpha-j)^{-1-\lambda}-\frac{1}{2}(i+1+\alpha-j)^{1-\lambda}<0$, for all $j=0,1,...,i$. Thus,
    \begin{equation}\label{107}
    \widetilde{f}_{i+1,j-1}^{\lambda\alpha}+\widetilde{f}_{i+1,j}^{\lambda\alpha}-\widetilde{d}_{i+1,j}^{\lambda\alpha}<0,\text{\,\,\,for\,\,\,}
    j=0,1,...,i.
   \end{equation}
    But, for $j=1,2,...,i$, plugging equation $(\ref{101})$ and estimate $(\ref{107})$ yields
    \begin{equation*}
    \widetilde{f}_{i+\frac{1}{2},j-1}^{\lambda\alpha}+\widetilde{f}_{i+\frac{1}{2},j}^{\lambda\alpha}-\widetilde{d}_{i+\frac{1}{2},j}^{\lambda\alpha}<0,
    \text{\,\,\,for\,\,\,}j=1,2,...,i.
   \end{equation*}
    Furthermore, combining $(\ref{53})$ and Lemma $\ref{l5}$, we obtain
   \begin{equation*}
    2\widetilde{f}_{i+1,j}^{\lambda\alpha}-\widetilde{d}_{i+1,j}^{\lambda\alpha}=\frac{4}{2-\lambda}[(i+1+\alpha-j)^{2-\lambda}-(i+\alpha-j)^{2-\lambda}]
    -2[(i+1+\alpha-j)^{1-\lambda}+(i+\alpha-j)^{1-\lambda}]
   \end{equation*}
    \begin{equation*}
    =4\left\{\frac{1}{2-\lambda}[(i+1+\alpha-j)^{2-\lambda}-(i+\alpha-j)^{2-\lambda}]-\frac{1}{2}[(i+1+\alpha-j)^{1-\lambda}
    +(i+\alpha-j)^{1-\lambda}]\right\}>0,\text{\,}l=0,1,...,i.
   \end{equation*}
   The last estimate comes from Lemma $\ref{l5}$. Utilizing this inequality together with relation $(\ref{101})$, it holds
    \begin{equation*}
    2\widetilde{f}_{i+\frac{1}{2},j}^{\lambda\alpha}-\widetilde{d}_{i+\frac{1}{2},j}^{\lambda\alpha}>0,\text{\,\,\,for\,\,\,}l=0,1,...,i-1.
   \end{equation*}
    This completes the proof of relation $(\ref{98})$.\\

    Finally, we should prove the first estimate in relation $(\ref{99})$. The proof of the second one is similar.\\

    Firstly, it comes from the expression of $a_{i+\frac{1}{2},\frac{1}{2}}^{\lambda\alpha}$ and $a_{i+\frac{1}{2},1}^{\lambda\alpha}$, that
    \begin{equation*}
    a_{i+\frac{1}{2},\frac{1}{2}}^{\lambda\alpha}-a_{i+\frac{1}{2},1}^{\lambda\alpha}=\widetilde{\dot{f}}_{i+\frac{1}{2},0}^{\lambda\alpha}+
    \widetilde{f}_{i+\frac{1}{2},0}^{\lambda\alpha}-\widetilde{d}_{i+\frac{1}{2},0}^{\lambda\alpha}.
   \end{equation*}
    Utilizing equation $(\ref{12})$ and $(\ref{22a})$, this becomes
    \begin{equation*}
    a_{i+\frac{1}{2},\frac{1}{2}}^{\lambda\alpha}-a_{i+\frac{1}{2},1}^{\lambda\alpha}=\frac{-2}{2-\lambda}[(i+\alpha)^{2-\lambda}-(i+\alpha-1)^{2-\lambda}]
    +\frac{1}{2}[(i+\alpha)^{1-\lambda}+3(i+\alpha-1)^{1-\lambda}]+(i+\frac{1}{2}+\alpha)^{1-\lambda}-(i+\alpha)^{1-\lambda}
   \end{equation*}
    \begin{equation*}
    =-2\int_{0}^{1}(i+\alpha+y_{1}-1)^{1-\lambda}dy_{1}-\frac{1}{2}[(i+\alpha)^{1-\lambda}-3(i+\alpha-1)^{1-\lambda}]+(i+\frac{1}{2}+\alpha)^{1-\lambda}.
   \end{equation*}
    Applying the integral mean value theorem, there exists $y^{\alpha}\in(0,1)$ such that, $\int_{0}^{1}(i+\alpha+y_{1}-1)^{1-\lambda}dy_{1}=
    (i+\alpha+y^{\alpha}-1)^{1-\lambda}$. Using this and assuming that $4y^{\alpha}\leq 3$, it holds
    \begin{equation*}
    a_{i+\frac{1}{2},\frac{1}{2}}^{\lambda\alpha}-a_{i+\frac{1}{2},1}^{\lambda\alpha}=-2(i+\alpha+y^{\alpha}-1)^{1-\lambda}+\frac{3}{2}(i+\alpha-1)^{1-\lambda}
    -\frac{1}{2}(i+\alpha)^{1-\lambda}+(i+\frac{1}{2}+\alpha)^{1-\lambda}=
   \end{equation*}
    \begin{equation*}
    (1-\lambda)\left[\int_{y^{\alpha}}^{\frac{3}{2}}(i+\alpha+y_{1}-1)^{-\lambda}dy_{1}-\int_{0}^{y^{\alpha}}(i+\alpha+y_{1}-1)^{-\lambda}dy_{1}-\frac{1}{2}
    \int_{0}^{1}(i+\alpha+y_{1}-1)^{-\lambda}dy_{1}\right]\leq
   \end{equation*}
    \begin{equation*}
    (1-\lambda)\left[(\frac{3}{2}-y^{\alpha})(i+\alpha+y^{\alpha}-1)^{-\lambda}-y^{\alpha}(i+\alpha+y^{\alpha}-1)^{-\lambda}-\frac{1}{2}
    (i+\alpha)^{-\lambda}\right]=
   \end{equation*}
    \begin{equation*}
    (1-\lambda)\left[(\frac{3}{2}-2y^{\alpha})(i+\alpha+y^{\alpha}-1)^{-\lambda}-\frac{1}{2}(i+\alpha)^{-\lambda}\right]=
    \frac{1-\lambda}{2(i+\alpha)^{\lambda}}\left[(3-4y^{\alpha})\left(\frac{i+\alpha}{i+\alpha+y^{\alpha}-1}\right)^{\lambda}-1\right]=
   \end{equation*}
   \begin{equation*}
    \frac{1-\lambda}{2(i+\alpha)^{\lambda}}\left[\underset{>0}{(\underbrace{3-4y^{\alpha}})}\left(1+\frac{1-y^{\alpha}}{i+\alpha+
    y^{\alpha}-1}\right)^{\lambda}-1\right]<\frac{1-\lambda}{2(i+\alpha)^{\lambda}}\left[\underset{>0}{(\underbrace{3-4y^{\alpha}})}
    \left(1+\frac{\lambda(1-y^{\alpha})}{i+\alpha+(y^{\alpha}-1}\right)-1\right]=
   \end{equation*}
    \begin{equation*}
    -\frac{\alpha(\alpha+y^{\alpha})}{2(i+\alpha)^{\lambda}}\left[4\alpha(y^{\alpha})^{2}+(5-3\alpha)y^{\alpha}-3+\alpha\right]\leq0,
   \end{equation*}
    for values of $y^{\alpha}$ satisfying
    \begin{equation}\label{107a}
    \max\left\{0,\frac{1}{8\alpha}\left(-5+3\alpha+\sqrt{-7\alpha^{2}+18\alpha+25}\right)\right\}\leq y^{\alpha}\leq \frac{3}{4},
   \end{equation}
    where $\alpha=1-\lambda$ and $\lambda\in(0,\frac{2}{3})$. In fact, since $y^{\alpha}\in(0,1)$, without loss of this generality, we can assume that
    $y^{\alpha}$ satisfies estimate $(\ref{107a})$. Thus
    \begin{equation*}
    a_{i+\frac{1}{2},\frac{1}{2}}^{\lambda\alpha}<a_{i+\frac{1}{2},1}^{\lambda\alpha}.
   \end{equation*}
    For $l=1,\frac{3}{2},2,\frac{5}{2},...,i-\frac{1}{2}$, if $l$ is an integer, it follows from relation $(\ref{95})$ that
    \begin{equation*}
    a_{i+\frac{1}{2},l}^{\lambda\alpha}=\widetilde{d}_{i+\frac{1}{2},l-1}^{\lambda\alpha}-\widetilde{f}_{i+\frac{1}{2},l-1}^{\lambda\alpha}\text{\,\,\,and\,\,\,}
    a_{i+\frac{1}{2},l+\frac{1}{2}}^{\lambda\alpha}=\widetilde{f}_{i+\frac{1}{2},l-1}^{\lambda\alpha}.
   \end{equation*}
   Utilizing the second inequality in $(\ref{98})$, it is easy to see that
   \begin{equation*}
    a_{i+\frac{1}{2},l}^{\lambda\alpha}-a_{i+\frac{1}{2},l+\frac{1}{2}}^{\lambda\alpha}=\widetilde{d}_{i+\frac{1}{2},l-1}^{\lambda\alpha}-
    2\widetilde{f}_{i+\frac{1}{2},l-1}^{\lambda\alpha}<0.
   \end{equation*}
   If $l$ is not an integer, relation $(\ref{95})$ provides
   \begin{equation*}
    a_{i+\frac{1}{2},l}^{\lambda\alpha}=\widetilde{f}_{i+\frac{1}{2},l-\frac{3}{2}}^{\lambda\alpha}\text{\,\,\,and\,\,\,}a_{i+\frac{1}{2},l+\frac{1}{2}}^{\lambda\alpha}=
    \widetilde{d}_{i+\frac{1}{2},l-\frac{1}{2}}^{\lambda\alpha}-\widetilde{f}_{i+\frac{1}{2},l-\frac{1}{2}}^{\lambda\alpha}.
   \end{equation*}
    Hence
    \begin{equation*}
    a_{i+\frac{1}{2},l}^{\lambda\alpha}-a_{i+\frac{1}{2},l+\frac{1}{2}}^{\lambda\alpha}=\widetilde{f}_{i+\frac{1}{2},l-\frac{3}{2}}^{\lambda\alpha}+
    \widetilde{f}_{i+\frac{1}{2},l-\frac{1}{2}}^{\lambda\alpha}-\widetilde{d}_{i+\frac{1}{2},l-\frac{1}{2}}^{\lambda\alpha}<0.
   \end{equation*}
    The last estimate comes from the first inequality in $(\ref{98})$. Furthermore,
    \begin{equation*}
    a_{i+\frac{1}{2},i}^{\lambda\alpha}-a_{i+\frac{1}{2},i+\frac{1}{2}}^{\lambda\alpha}=\widetilde{d}_{i+\frac{1}{2},i-1}^{\lambda\alpha}
    -2\widetilde{f}_{i+\frac{1}{2},i-1}^{\lambda\alpha}-\widetilde{f}_{i+\frac{1}{2},-}^{\lambda\alpha}<0.
   \end{equation*}
    In a similar manner, Using relations $(\ref{96})$ and $(\ref{98})$, one easily shows that,
    \begin{equation*}
    a_{i+1,l}^{\lambda\alpha}-a_{i+1,l+\frac{1}{2}}^{\lambda\alpha}<0,
   \end{equation*}
    for $l=\frac{1}{2},1,\frac{3}{2},2,\frac{5}{2},...,i+\frac{1}{2}.$ This completes the proof of relation $(\ref{99})$. To end, we should prove the
    second inequality in relation $(\ref{100})$. The proof of the first one is similar.\\

    Combining equation $(\ref{105})$ and the last equation in $(\ref{97})$, it holds
    \begin{equation*}
    \left(\frac{3}{2}-2\mathcal{Z}_{i+1-l}\right)[(i+1+\alpha-l)^{1-\lambda}-(i+\alpha-l)^{1-\lambda}]=\widetilde{d}_{i+1,l}^{\lambda\alpha}-
    \widetilde{f}_{i+1,l}^{\lambda\alpha}=a_{i+1,l+\frac{1}{2}}^{\lambda\alpha},
   \end{equation*}
   if $l$ is an integer satisfying $0\leq l\leq i$. This fact together with Lemma $\ref{l4}$ result in
   \begin{equation}\label{109}
    a_{i+1,l+\frac{1}{2}}^{\lambda\alpha}>\left(\frac{3}{2}-\frac{2}{2-\lambda}\right)(1-\lambda)\int_{0}^{1}(i+\alpha+y_{1}-l)^{-\lambda}dy_{1}=
    \frac{(2-3\lambda)(1-\lambda)}{2(2-\lambda)}(i+1+\alpha-l)^{-\lambda},
   \end{equation}
    for $l=0,1,...,i$. Now, if $l=\frac{1}{2},\frac{3}{2},\frac{5}{2},...,i+\frac{1}{2}$, plugging equation $(\ref{106})$ and Lemma $\ref{l4}$, it is not
    difficult to observe that
     \begin{equation*}
    \left(2\mathcal{Z}_{i+1-(l-\frac{1}{2})}-\frac{1}{2}\right)[(i+1+\alpha-(l-\frac{1}{2}))^{1-\lambda}-(i+\alpha-(l-\frac{1}{2}))^{1-\lambda}]=
    \widetilde{f}_{i+1,l-\frac{1}{2}}^{\lambda\alpha}=a_{i+1,l+\frac{1}{2}}^{\lambda\alpha}.
   \end{equation*}
    Since $2\mathcal{Z}_{i+1-(l-\frac{1}{2})}>1$, this implies
    \begin{equation*}
    a_{i+1,l+\frac{1}{2}}^{\lambda\alpha}>\frac{1-\lambda}{2}\int_{0}^{1}(i+\alpha+y_{1}-(l-\frac{1}{2}))^{-\lambda}dy_{1}\geq
    \frac{1-\lambda}{2}(i+1+\alpha-(l-\frac{1}{2}))^{-\lambda}>\frac{(2-3\lambda)(1-\lambda)}{2(2-\lambda)}(i+\frac{3}{2}+\alpha-l)^{-\lambda}.
   \end{equation*}
   The last estimate follows from $1>\frac{2-3\lambda}{2-\lambda}>0$, for any $0<\lambda<\frac{2}{3}.$ Thus,
    \begin{equation}\label{110}
    \widetilde{f}_{i+1,l-\frac{1}{2}}^{\lambda\alpha}=a_{i+1,l+\frac{1}{2}}^{\lambda\alpha}>\frac{(2-3\lambda)(1-\lambda)}{2(2-\lambda)}(i+\frac{3}{2}+\alpha-l)^{-\lambda},
   \end{equation}
    $l=\frac{1}{2},\frac{3}{2},\frac{5}{2},...,i+\frac{1}{2}$. Since $\alpha=1-\lambda$ and $1>\frac{2-3\lambda}{2-\lambda}>0$, for any
    $0<\lambda<\frac{2}{3}$, a combination of $(\ref{53})$, $(\ref{96})$ and $(\ref{110})$ gives
     \begin{equation*}
    a_{i+1,i+1}^{\lambda\alpha}=\widetilde{f}_{i+1,i}^{\lambda\alpha}+\widetilde{f}_{i+1,i+1}^{\lambda\alpha}>
    \frac{(2-3\lambda)(1-\lambda)}{2(2-\lambda)}(1+\alpha)^{-\lambda}+\alpha^{1-\lambda}>\alpha\alpha^{-\lambda}>\frac{1-\lambda}{2}\alpha^{-\lambda}>
   \end{equation*}
    \begin{equation}\label{112}
     \frac{(2-3\lambda)(1-\lambda)}{2(2-\lambda)}\alpha^{-\lambda}.
   \end{equation}
    An assembly of estimates $(\ref{109})$-$(\ref{112})$ ends the proof of the second inequality in relation $(\ref{100})$. In a similar way,
    one easily show the first estimate in $(\ref{100})$. This completes the proof of Lemma $\ref{l6}$.
    \end{proof}

    The following Lemma (Lemma $\ref{l7}$) plays a crucial role when proving Lemma $\ref{l8}$.
   \begin{lemma}\label{l7}
    Consider the generalized sequences $(a_{\cdot,l}^{\lambda\alpha})_{l}$ defined by relations $(\ref{94})$-$(\ref{96})$. For any mesh function $w$
    defined on the grid space $\mathcal{Y}_{kh}$, it holds
     \begin{equation*}
    \underset{l=l_{0}}{\overset{m}\sum}(a_{\cdot,l+\frac{1}{2}}^{\lambda\alpha})^{-1}[(w^{l+\frac{1}{2}})^{2}-(w^{l})^{2}]=
    (a_{\cdot,m}^{\lambda\alpha})^{-1}(w^{m})^{2}-(a_{\cdot,l_{0}}^{\lambda\alpha})^{-1}(w^{l_{0}})^{2}+
     \underset{l=l_{0}}{\overset{m-\frac{1}{2}}\sum}[(a_{\cdot,l+\frac{1}{2}}^{\lambda\alpha})^{-1}-(a_{\cdot,l+1}^{\lambda\alpha})^{-1}]
     (w^{l+\frac{1}{2}})^{2},
   \end{equation*}
    for $m\in\{i,i+\frac{1}{2}\}$ and $l=l_{0},l_{0}+\frac{1}{2},l_{0}+1,l_{0}+\frac{3}{2},...,m$, where $l_{0}$ is a nonnegative integer that satisfies
    $l_{0}\leq m$.
   \end{lemma}

    \begin{proof}
     Expanding the left side of this equality and rearranging terms to obtain the result.
    \end{proof}

   \begin{lemma}\label{l8}
    Given the generalized sequences $(a_{\cdot,l}^{\lambda\alpha})_{l}$ defined by relations $(\ref{94})$-$(\ref{96})$, respectively, for every grid
    function $u(\cdot,\cdot)$ defined on the mesh space $\mathcal{Y}_{kh}$, the following estimates are satisfied
    \begin{equation*}
    u_{j}^{i+\frac{1}{2}}(c\Delta_{0t}^{\lambda}u_{j}^{i+\frac{1}{2}+\alpha})=\frac{1}{2}c\Delta_{0t}^{\lambda}(u_{j}^{i+\frac{1}{2}+\alpha})^{2}+
    \frac{k^{2-\lambda}}{4\Gamma(2-\lambda)}\left\{(a_{i+\frac{1}{2},i+\frac{1}{2}}^{\lambda\alpha})^{-1}\left(\underset{l=0}{\overset{i}\sum}
    a_{i+\frac{1}{2},l+\frac{1}{2}}^{\lambda\alpha}\delta_{t}u_{j}^{l}\right)^{2}+\right.
   \end{equation*}
    \begin{equation}\label{113}
    \left.\left[a_{i+\frac{1}{2},\frac{1}{2}}^{\lambda\alpha}-(a_{i+\frac{1}{2},1}^{\lambda\alpha})^{-1}(a_{i+\frac{1}{2},\frac{1}{2}}^{\lambda\alpha})^{2}\right]
    \left(\delta_{t}u_{j}^{0}\right)^{2}+\underset{l=\frac{1}{2}}{\overset{i-\frac{1}{2}}\sum}\left[(a_{i+\frac{1}{2},l+\frac{1}{2}}^{\lambda\alpha})^{-1}
    -(a_{i+\frac{1}{2},l+1}^{\lambda\alpha})^{-1}\right]\left(\underset{r=0}{\overset{l}\sum}a_{i+\frac{1}{2},r+\frac{1}{2}}^{\lambda\alpha}
    \delta_{t}u_{j}^{r}\right)^{2}\right\},
   \end{equation}
    \begin{equation*}
    u_{j}^{i}(c\Delta_{0t}^{\lambda}u_{j}^{i+\frac{1}{2}+\alpha})=\frac{1}{2}c\Delta_{0t}^{\lambda}(u_{j}^{i+\frac{1}{2}+\alpha})^{2}-
    \frac{k^{2-\lambda}}{4\Gamma(2-\lambda)}\left\{(a_{i+\frac{1}{2},i+\frac{1}{2}}^{\lambda\alpha})^{-1}(\delta_{t}u_{j}^{i})^{2}-
    \left[a_{i+\frac{1}{2},\frac{1}{2}}^{\lambda\alpha}-(a_{i+\frac{1}{2},1}^{\lambda\alpha})^{-1}
    (a_{i+\frac{1}{2},\frac{1}{2}}^{\lambda\alpha})^{2}\right]*\right.
   \end{equation*}
    \begin{equation}\label{114}
     \left.\left(\delta_{t}u_{j}^{0}\right)^{2}
    -(a_{i+\frac{1}{2},i}^{\lambda\alpha})^{-1}\left(\underset{l=0}{\overset{i-\frac{1}{2}}\sum}
    a_{i+\frac{1}{2},l+\frac{1}{2}}^{\lambda\alpha}\delta_{t}u_{j}^{l}\right)^{2}-\underset{l=\frac{1}{2}}{\overset{i-1}\sum}
    \left[(a_{i+\frac{1}{2},l+\frac{1}{2}}^{\lambda\alpha})^{-1}-(a_{i+\frac{1}{2},l+1}^{\lambda\alpha})^{-1}\right]\left(\underset{r=0}
    {\overset{l}\sum}a_{i+\frac{1}{2},r+\frac{1}{2}}^{\lambda\alpha}\delta_{t}u_{j}^{r}\right)^{2}\right\}.
     \end{equation}
     Furthermore,
    \begin{equation*}
    u_{j}^{i+1}(c\Delta_{0t}^{\lambda}u_{j}^{i+1+\alpha})=\frac{1}{2}c\Delta_{0t}^{\lambda}(u_{j}^{i+1+\alpha})^{2}+
    \frac{k^{2-\lambda}}{4\Gamma(2-\lambda)}\left\{(a_{i+1,i+1}^{\lambda\alpha})^{-1}\left(\underset{l=0}{\overset{i+\frac{1}{2}}\sum}
    a_{i+1,l+\frac{1}{2}}^{\lambda\alpha}\delta_{t}u_{j}^{l}\right)^{2}+\right.
   \end{equation*}
    \begin{equation}\label{115}
    \left.\left[a_{i+1,\frac{1}{2}}^{\lambda\alpha}-(a_{i+1,1}^{\lambda\alpha})^{-1}(a_{i+1,\frac{1}{2}}^{\lambda\alpha})^{2}\right]
    \left(\delta_{t}u_{j}^{0}\right)^{2}+\underset{l=\frac{1}{2}}{\overset{i}\sum}\left[(a_{i+1,l+\frac{1}{2}}^{\lambda\alpha})^{-1}
    -(a_{i+1,l+1}^{\lambda\alpha})^{-1}\right]\left(\underset{r=0}{\overset{l}\sum}a_{i+1,r+\frac{1}{2}}^{\lambda\alpha}
    \delta_{t}u_{j}^{r}\right)^{2}\right\},
   \end{equation}
    \begin{equation*}
    u_{j}^{i+\frac{1}{2}}(c\Delta_{0t}^{\lambda}u_{j}^{i+1+\alpha})=\frac{1}{2}c\Delta_{0t}^{\lambda}(u_{j}^{i+1+\alpha})^{2}-
    \frac{k^{2-\lambda}}{4\Gamma(2-\lambda)}\left\{(a_{i+1,i+1}^{\lambda\alpha})^{-1}(\delta_{t}u_{j}^{i+\frac{1}{2}})^{2}-
    \left[a_{i+1,\frac{1}{2}}^{\lambda\alpha}-(a_{i+1,1}^{\lambda\alpha})^{-1}(a_{i+1,\frac{1}{2}}^{\lambda\alpha})^{2}\right]*\right.
   \end{equation*}
    \begin{equation}\label{116}
     \left.\left(\delta_{t}u_{j}^{0}\right)^{2}-(a_{i+1,i+\frac{1}{2}}^{\lambda\alpha})^{-1}\left(\underset{l=0}{\overset{i}\sum}
    a_{i+1,l+\frac{1}{2}}^{\lambda\alpha}\delta_{t}u_{j}^{l}\right)^{2}-\underset{l=\frac{1}{2}}{\overset{i-\frac{1}{2}}\sum}
    \left[(a_{i+1,l+\frac{1}{2}}^{\lambda\alpha})^{-1}-(a_{i+1,l+1}^{\lambda\alpha})^{-1}\right]\left(\underset{r=0}
    {\overset{l}\sum}a_{i+1,r+\frac{1}{2}}^{\lambda\alpha}\delta_{t}u_{j}^{r}\right)^{2}\right\},
     \end{equation}
    where $"*"$ denotes the usual multiplication in $\mathbb{C}$, $c\Delta_{0t}^{\lambda}u_{j}^{i+\frac{1}{2}+\alpha}$ and
    $c\Delta_{0t}^{\lambda}u_{j}^{i+1+\alpha}$ are defined by equations $(\ref{21})$ and $(\ref{51a})$, respectively.
   \end{lemma}

   \begin{proof}
    Firstly, combining equations $(\ref{21})$, $(\ref{22})$ and $(\ref{95})$, it is not difficult to observe that
    \begin{equation}\label{117}
    c\Delta_{0t}^{\lambda}u_{j}^{i+\frac{1}{2}+\alpha}=\frac{k^{1-\lambda}}{\Gamma(2-\lambda)}\underset{l=0}{\overset{i}\sum}
    a_{i+\frac{1}{2},l+\frac{1}{2}}^{\lambda\alpha}\delta_{t}u_{j}^{l}.
     \end{equation}
     Similarly, a combining equations $(\ref{51a})$, $(\ref{52})$ and $(\ref{96})$ gives
    \begin{equation}\label{118}
    c\Delta_{0t}^{\lambda}u_{j}^{i+1+\alpha}=\frac{k^{1-\lambda}}{\Gamma(2-\lambda)}\underset{l=0}{\overset{i+\frac{1}{2}}\sum}
    a_{i+1,l+\frac{1}{2}}^{\lambda\alpha}\delta_{t}u_{j}^{l},
     \end{equation}
    where the summation index $l$ varies with a step size $\frac{1}{2}.$ Furthermore, using relation $(\ref{117})$, we should prove only equations
    $(\ref{113})$ and $(\ref{114})$. The proof of relations $(\ref{115})$ and $(\ref{116})$ is similar thanks to equation $(\ref{118})$.\\

    Subtracting the quantity $\frac{1}{2}c\Delta_{0t}^{\lambda}(u_{j}^{i+\frac{1}{2}+\alpha})^{2}$ from
    $u_{j}^{i+\frac{1}{2}}(c\Delta_{0t}^{\lambda}u_{j}^{i+\frac{1}{2}+\alpha})$ to get
    \begin{equation*}
    u_{j}^{i+\frac{1}{2}}(c\Delta_{0t}^{\lambda}u_{j}^{i+\frac{1}{2}+\alpha})-\frac{1}{2}c\Delta_{0t}^{\lambda}(u_{j}^{i+\frac{1}{2}+\alpha})^{2}=
    \frac{k^{1-\lambda}}{\Gamma(2-\lambda)}\left[u_{j}^{i+\frac{1}{2}}\underset{l=0}{\overset{i}\sum}
    a_{i+\frac{1}{2},l+\frac{1}{2}}^{\lambda\alpha}\delta_{t}u_{j}^{l}-\frac{1}{2}\underset{l=0}{\overset{i}\sum}
    a_{i+\frac{1}{2},l+\frac{1}{2}}^{\lambda\alpha}\delta_{t}(u_{j}^{l})^{2}\right]=
   \end{equation*}
    \begin{equation*}
    \frac{k^{1-\lambda}}{\Gamma(2-\lambda)}\left[u_{j}^{i+\frac{1}{2}}\underset{l=0}{\overset{i}\sum}a_{i+\frac{1}{2},l+\frac{1}{2}}^{\lambda\alpha}
    \delta_{t}u_{j}^{l}-\frac{1}{2}\underset{l=0}{\overset{i}\sum}a_{i+\frac{1}{2},l+\frac{1}{2}}^{\lambda\alpha}\frac{(u_{j}^{l+\frac{1}{2}})^{2}-
    (u_{j}^{l})^{2}}{k/2}\right]=\frac{k^{1-\lambda}}{\Gamma(2-\lambda)}\left[u_{j}^{i+\frac{1}{2}}\underset{l=0}{\overset{i}\sum}
    a_{i+\frac{1}{2},l+\frac{1}{2}}^{\lambda\alpha}\delta_{t}u_{j}^{l}\right.
   \end{equation*}
   \begin{equation*}
    \left.-\frac{1}{2}\underset{l=0}{\overset{i}\sum}a_{i+\frac{1}{2},l+\frac{1}{2}}^{\lambda\alpha}\delta_{t}u_{j}^{l}(u_{j}^{l+\frac{1}{2}}
    +u_{j}^{l})\right]=\frac{k^{1-\lambda}}{\Gamma(2-\lambda)}\underset{l=0}{\overset{i}\sum}a_{i+\frac{1}{2},l
    +\frac{1}{2}}^{\lambda\alpha}\delta_{t}u_{j}^{l}\left(u_{j}^{i+\frac{1}{2}}-\frac{u_{j}^{l+\frac{1}{2}}+u_{j}^{l}}{2}\right)=
   \end{equation*}
   \begin{equation*}
   \frac{k^{1-\lambda}}{\Gamma(2-\lambda)}\underset{l=0}{\overset{i}\sum}a_{i+\frac{1}{2},l+\frac{1}{2}}^{\lambda\alpha}\delta_{t}u_{j}^{l}
   \left[\frac{1}{2}(u_{j}^{l+\frac{1}{2}}-u_{j}^{l})+\underset{r=l+\frac{1}{2}}{\overset{i}\sum}(u_{j}^{r+\frac{1}{2}}-u_{j}^{r})\right]=
    \frac{k^{2-\lambda}}{4\Gamma(2-\lambda)}\left[\underset{l=0}{\overset{i}\sum}a_{i+\frac{1}{2},l+\frac{1}{2}}^{\lambda\alpha}(\delta_{t}u_{j}^{l})^{2}+
    \right.
   \end{equation*}
   \begin{equation}\label{119}
    \left.2\underset{l=0}{\overset{i}\sum}a_{i+\frac{1}{2},l+\frac{1}{2}}^{\lambda\alpha}\delta_{t}u_{j}^{l}\underset{r=l+\frac{1}{2}}{\overset{i}\sum}
    \delta_{t}u_{j}^{r}\right]=\frac{k^{2-\lambda}}{4\Gamma(2-\lambda)}\left[\underset{l=0}{\overset{i}\sum}a_{i+\frac{1}{2},l+\frac{1}{2}}^{\lambda\alpha}
    (\delta_{t}u_{j}^{l})^{2}+2\underset{r=\frac{1}{2}}{\overset{i}\sum}\delta_{t}u_{j}^{r}\underset{l=0}{\overset{r-\frac{1}{2}}\sum}
    a_{i+\frac{1}{2},l+\frac{1}{2}}^{\lambda\alpha}\delta_{t}u_{j}^{l}\right].
   \end{equation}
   We recall that the summation indices $r$ and $l$ varies in the ranges: $\frac{1}{2},1,\frac{3}{2},...,i$ and $0,\frac{1}{2},1,...,r-\frac{1}{2}$,
   respectively. In order to simplifying computations, we set
   \begin{equation}\label{120}
    W_{i+\frac{1}{2},j}^{r+\frac{1}{2}}=\underset{l=0}{\overset{r}\sum}a_{i+\frac{1}{2},l+\frac{1}{2}}^{\lambda\alpha}\delta_{t}u_{j}^{l},\text{\,\,\,for\,\,\,}
    r\geq0.
   \end{equation}
    Utilizing this, direct calculations yield
    \begin{equation}\label{121}
    \delta_{t}u_{j}^{r}=(a_{i+\frac{1}{2},r+\frac{1}{2}}^{\lambda\alpha})^{-1}(W_{i+\frac{1}{2},j}^{r+\frac{1}{2}}-W_{i+\frac{1}{2},j}^{r}),
    \text{\,\,\,for\,\,\,}r\geq0.
   \end{equation}
   Substituting $(\ref{121})$ into equation $(\ref{119})$ and using Lemma $\ref{l7}$, simple computations provide
      \begin{equation*}
    u_{j}^{i+\frac{1}{2}}(c\Delta_{0t}^{\lambda}u_{j}^{i+\frac{1}{2}+\alpha})-\frac{1}{2}c\Delta_{0t}^{\lambda}(u_{j}^{i+\frac{1}{2}+\alpha})^{2}=
    \frac{k^{2-\lambda}}{4\Gamma(2-\lambda)}\left[a_{i+\frac{1}{2},\frac{1}{2}}^{\lambda\alpha}(\delta_{t}u_{j}^{0})^{2}+
    \underset{l=\frac{1}{2}}{\overset{i}\sum}(a_{i+\frac{1}{2},l+\frac{1}{2}}^{\lambda\alpha})^{-1}
    (W_{i+\frac{1}{2},j}^{l+\frac{1}{2}}-W_{i+\frac{1}{2},j}^{l})^{2}\right.
   \end{equation*}
     \begin{equation*}
    \left.+2\underset{r=\frac{1}{2}}{\overset{i}\sum}(a_{i+\frac{1}{2},r+\frac{1}{2}}^{\lambda\alpha})^{-1}(W_{i+\frac{1}{2},j}^{r+\frac{1}{2}}-
    W_{i+\frac{1}{2},j}^{r})W_{i+\frac{1}{2},j}^{r}\right]=\frac{k^{2-\lambda}}{4\Gamma(2-\lambda)}\left[a_{i+\frac{1}{2},\frac{1}{2}}^{\lambda\alpha}
   (\delta_{t}u_{j}^{0})^{2}+\right.
   \end{equation*}
     \begin{equation*}
    \underset{l=\frac{1}{2}}{\overset{i}\sum}(a_{i+\frac{1}{2},l+\frac{1}{2}}^{\lambda\alpha})^{-1}
    \left[(W_{i+\frac{1}{2},j}^{l+\frac{1}{2}})^{2}-(W_{i+\frac{1}{2},j}^{l})^{2}\right]\overset{(a)}{=}\frac{k^{2-\lambda}}{4\Gamma(2-\lambda)}
    \left[a_{i+\frac{1}{2},\frac{1}{2}}^{\lambda\alpha}(\delta_{t}u_{j}^{0})^{2}+(a_{i+\frac{1}{2},i+\frac{1}{2}}^{\lambda\alpha})^{-1}
    (W_{i+\frac{1}{2},j}^{i+\frac{1}{2}})^{2}-\right.
   \end{equation*}
      \begin{equation*}
    \left.(a_{i+\frac{1}{2},1}^{\lambda\alpha})^{-1}(W_{i+\frac{1}{2},j}^{\frac{1}{2}})^{2}+
    \underset{l=\frac{1}{2}}{\overset{i-\frac{1}{2}}\sum}\left[(a_{i+\frac{1}{2},l+\frac{1}{2}}^{\lambda\alpha})^{-1}-
    (a_{i+\frac{1}{2},l+1}^{\lambda\alpha})^{-1}\right](W_{i+\frac{1}{2},j}^{l+\frac{1}{2}})^{2}\right\}.
   \end{equation*}
    Equality $(a)$ is due to Lemma $\ref{l7}$. This ends the proof of relation $(\ref{113})$ thanks to equation $(\ref{120})$.\\

    Let prove equation $(\ref{114})$. Adding the term $\frac{1}{2}c\Delta_{0t}^{\lambda}(u_{j}^{i+\frac{1}{2}+\alpha})^{2}$ to
    $-u_{j}^{i}(c\Delta_{0t}^{\lambda}u_{j}^{i+\frac{1}{2}+\alpha})$ and performing straightforward calculations, this yields
     \begin{equation*}
    -u_{j}^{i}(c\Delta_{0t}^{\lambda}u_{j}^{i+\frac{1}{2}+\alpha})+\frac{1}{2}c\Delta_{0t}^{\lambda}(u_{j}^{i+\frac{1}{2}+\alpha})^{2}=
     \frac{k^{1-\lambda}}{\Gamma(2-\lambda)}\underset{l=0}{\overset{i}\sum}a_{i+\frac{1}{2},l
    +\frac{1}{2}}^{\lambda\alpha}\delta_{t}u_{j}^{l}\left[\frac{1}{2}(u_{j}^{l+\frac{1}{2}}+u_{j}^{l})-u_{j}^{i}\right]=
   \end{equation*}
   \begin{equation*}
   \frac{k^{1-\lambda}}{\Gamma(2-\lambda)}\underset{l=0}{\overset{i}\sum}a_{i+\frac{1}{2},l+\frac{1}{2}}^{\lambda\alpha}\delta_{t}u_{j}^{l}
   \left[\frac{1}{2}(u_{j}^{l+\frac{1}{2}}-u_{j}^{l})-\underset{r=l}{\overset{i-\frac{1}{2}}\sum}(u_{j}^{r+\frac{1}{2}}-u_{j}^{r})\right]=
    \frac{k^{2-\lambda}}{4\Gamma(2-\lambda)}\left[\underset{l=0}{\overset{i}\sum}a_{i+\frac{1}{2},l+\frac{1}{2}}^{\lambda\alpha}(\delta_{t}u_{j}^{l})^{2}-
    \right.
   \end{equation*}
   \begin{equation*}
    \left.2\underset{l=0}{\overset{i}\sum}a_{i+\frac{1}{2},l+\frac{1}{2}}^{\lambda\alpha}\delta_{t}u_{j}^{l}\underset{r=l}{\overset{i-\frac{1}{2}}\sum}
    \delta_{t}u_{j}^{r}\right]=\frac{k^{2-\lambda}}{4\Gamma(2-\lambda)}\left[\underset{l=0}{\overset{i}\sum}a_{i+\frac{1}{2},l+\frac{1}{2}}^{\lambda\alpha}
    (\delta_{t}u_{j}^{l})^{2}-2\underset{r=0}{\overset{i-\frac{1}{2}}\sum}\delta_{t}u_{j}^{r}\underset{l=0}{\overset{r}\sum}
    a_{i+\frac{1}{2},l+\frac{1}{2}}^{\lambda\alpha}\delta_{t}u_{j}^{l}\right],
   \end{equation*}
   since the sum equals zero if the lower summation index is less than the upper one. Using $(\ref{120})$-$(\ref{121})$ together with Lemma $\ref{l7}$,
    this equation becomes
     \begin{equation*}
    -u_{j}^{i}(c\Delta_{0t}^{\lambda}u_{j}^{i+\frac{1}{2}+\alpha})+\frac{1}{2}c\Delta_{0t}^{\lambda}(u_{j}^{i+\frac{1}{2}+\alpha})^{2}=
        \frac{k^{2-\lambda}}{4\Gamma(2-\lambda)}\left[-a_{i+\frac{1}{2},\frac{1}{2}}^{\lambda\alpha}(\delta_{t}u_{j}^{0})^{2}+
    \underset{l=\frac{1}{2}}{\overset{i-\frac{1}{2}}\sum}(a_{i+\frac{1}{2},l+\frac{1}{2}}^{\lambda\alpha})^{-1}
    (W_{i+\frac{1}{2},j}^{l+\frac{1}{2}}-W_{i+\frac{1}{2},j}^{l})^{2}\right.
   \end{equation*}
     \begin{equation*}
    \left.+a_{i+\frac{1}{2},i+\frac{1}{2}}^{\lambda\alpha}(\delta_{t}u_{j}^{i})^{2}-2\underset{r=\frac{1}{2}}{\overset{i-\frac{1}{2}}\sum}
    (a_{i+\frac{1}{2},r+\frac{1}{2}}^{\lambda\alpha})^{-1}(W_{i+\frac{1}{2},j}^{r+\frac{1}{2}}-W_{i+\frac{1}{2},j}^{r})W_{i+\frac{1}{2},j}^{r+\frac{1}{2}}\right]=
     \frac{k^{2-\lambda}}{4\Gamma(2-\lambda)}\left[a_{i+\frac{1}{2},i+\frac{1}{2}}^{\lambda\alpha}(\delta_{t}u_{j}^{i})^{2}-\right.
   \end{equation*}
     \begin{equation*}
    a_{i+\frac{1}{2},\frac{1}{2}}^{\lambda\alpha}(\delta_{t}u_{j}^{0})^{2}-\underset{l=\frac{1}{2}}{\overset{i-\frac{1}{2}}\sum}
    (a_{i+\frac{1}{2},l+\frac{1}{2}}^{\lambda\alpha})^{-1}\left[(W_{i+\frac{1}{2},j}^{l+\frac{1}{2}})^{2}-(W_{i+\frac{1}{2},j}^{l})^{2}\right]=
     \frac{k^{2-\lambda}}{4\Gamma(2-\lambda)}\left[a_{i+\frac{1}{2},i+\frac{1}{2}}^{\lambda\alpha}(\delta_{t}u_{j}^{i})^{2}-
     a_{i+\frac{1}{2},\frac{1}{2}}^{\lambda\alpha}(\delta_{t}u_{j}^{0})^{2}\right.
   \end{equation*}
      \begin{equation*}
    \left. +(a_{i+\frac{1}{2},1}^{\lambda\alpha})^{-1}(W_{i+\frac{1}{2},j}^{\frac{1}{2}})^{2}-(a_{i+\frac{1}{2},i}^{\lambda\alpha})^{-1}
    (W_{i+\frac{1}{2},j}^{i})^{2}-\underset{l=\frac{1}{2}}{\overset{i-1}\sum}\left[(a_{i+\frac{1}{2},l+\frac{1}{2}}^{\lambda\alpha})^{-1}-
    (a_{i+\frac{1}{2},l+1}^{\lambda\alpha})^{-1}\right](W_{i+\frac{1}{2},j}^{l+\frac{1}{2}})^{2}\right\}.
   \end{equation*}
    The proof of equation $(\ref{114})$ is completed thanks to equality $W_{i+\frac{1}{2},j}^{\frac{1}{2}}=
    a_{i+\frac{1}{2},\frac{1}{2}}^{\lambda\alpha}\delta_{t}u_{j}^{0}$ and equation $(\ref{120})$. In a similar manner, one easily proves relations
    $(\ref{115})$ and $(\ref{116})$. This completes the proof of Lemma $\ref{l8}$.
   \end{proof}

   \begin{lemma}\label{l9}
    Let $(a_{\cdot,l}^{\lambda\alpha})_{l}$ be the generalized sequences defined by equations $(\ref{94})$-$(\ref{96})$, respectively. For every grid
    function $u(\cdot,\cdot)$ defined on the mesh space $\mathcal{Y}_{kh}$, the following estimates hold
     \begin{equation}\label{122a}
    \left(c\Delta_{0t}^{\lambda}u^{i+\frac{1}{2}+\alpha},u^{\alpha_{i}}\right)\geq \frac{h}{2}\underset{j=1}{\overset{M-1}\sum}c\Delta_{0t}^{\lambda}
    (u_{j}^{i+\frac{1}{2}+\alpha})^{2},
   \end{equation}
    if $\alpha=1-\lambda$ satisfies
    \begin{equation}\label{122b}
    4\alpha^{2}-(1+4\alpha)\left((a_{i+\frac{1}{2},i}^{\lambda\alpha})^{-1}a_{i+\frac{1}{2},i+\frac{1}{2}}^{\lambda\alpha}-1\right)\leq
     0,\text{\,\,\,for\,\,\,}i\geq1.
   \end{equation}
    Furthermore,
    \begin{equation}\label{122c}
    \left(c\Delta_{0t}^{\lambda}u^{i+1+\alpha},u^{\theta_{i}}\right)\geq \frac{h}{2}\underset{j=1}{\overset{M-1}\sum}c\Delta_{0t}^{\lambda}
    (u_{j}^{i+1+\alpha})^{2},
   \end{equation}
    whenever
    \begin{equation}\label{122d}
    4\alpha^{2}-(1+4\alpha)\left((a_{i+1,i+\frac{1}{2}}^{\lambda\alpha})^{-1}a_{i+1,i+1}^{\lambda\alpha}-1\right)\leq 0,\text{\,\,\,for\,\,\,}
    i\geq0.
   \end{equation}
   \end{lemma}

    \begin{proof}
    We should prove only estimate $(\ref{122a})$, the proof of inequality $(\ref{122c})$ is similar.\\

    Multiplying both sides of relations $(\ref{113})$ and $(\ref{114})$ by  $1+2\alpha$ and $-2\alpha$, respectively, using relation $(\ref{120})$ and
    summing, it is not hard to see that
    \begin{equation*}
    (1+2\alpha)u^{i+\frac{1}{2}}_{j}c\Delta_{0t}^{\lambda}u_{j}^{i+\frac{1}{2}+\alpha}-2\alpha u^{i}_{j}c\Delta_{0t}^{\lambda}u_{j}^{i+\frac{1}{2}+\alpha}=
    \frac{1}{2}c\Delta_{0t}^{\lambda}(u_{j}^{i+\frac{1}{2}+\alpha})^{2}+\frac{k^{2-\lambda}}{4\Gamma(2-\lambda)}
    \left\{(1+2\alpha)(a_{i+\frac{1}{2},i+\frac{1}{2}}^{\lambda\alpha})^{-1}(W_{i+\frac{1}{2},j}^{i+\frac{1}{2}})^{2}\right.
   \end{equation*}
    \begin{equation*}
    +[a_{i+\frac{1}{2},\frac{1}{2}}^{\lambda\alpha}-(a_{i+\frac{1}{2},1}^{\lambda\alpha})^{-1}(a_{i+\frac{1}{2},\frac{1}{2}}^{\lambda\alpha})^{2}]
    (\delta_{t}u_{j}^{0})^{2}+(1+2\alpha)\left[(a_{i+\frac{1}{2},i}^{\lambda\alpha})^{-1}-(a_{i+\frac{1}{2},i+\frac{1}{2}}^{\lambda\alpha})^{-1}\right]
    (W_{i+\frac{1}{2},j}^{i})^{2}-2\alpha(a_{i+\frac{1}{2},i}^{\lambda\alpha})^{-1}*
   \end{equation*}
    \begin{equation*}
     \left.(W_{i+\frac{1}{2},j}^{i})^{2}+2\alpha(a_{i+\frac{1}{2},i+\frac{1}{2}}^{\lambda\alpha})^{-1}(W_{i+\frac{1}{2},j}^{i+\frac{1}{2}}-
    W_{i+\frac{1}{2},j}^{i})^{2}+\underset{l=\frac{1}{2}}{\overset{i-1}\sum}\left[(a_{i+\frac{1}{2},l+\frac{1}{2}}^{\lambda\alpha})^{-1}-
    (a_{i+\frac{1}{2},l+1}^{\lambda\alpha})^{-1}\right](W_{i+\frac{1}{2},j}^{l+\frac{1}{2}})^{2}\right\},
   \end{equation*}
    where $"*"$ denotes the usual multiplication in $\mathbb{C}$ and $W_{i+\frac{1}{2},j}^{m}$ is given by equation $(\ref{120})$. Since
    $u_{j}^{\alpha_{i}}=(1+2\alpha)u^{i+\frac{1}{2}}_{j}-2\alpha u^{i}_{j}$, expanding and rearranging terms, this is equivalent to
     \begin{equation*}
    u_{j}^{\alpha_{i}}c\Delta_{0t}^{\lambda}u_{j}^{i+\frac{1}{2}+\alpha}=\frac{1}{2}c\Delta_{0t}^{\lambda}(u_{j}^{i+\frac{1}{2}+\alpha})^{2}+
    \frac{k^{2-\lambda}}{4\Gamma(2-\lambda)}\left\{[a_{i+\frac{1}{2},\frac{1}{2}}^{\lambda\alpha}-(a_{i+\frac{1}{2},1}^{\lambda\alpha})^{-1}
    (a_{i+\frac{1}{2},\frac{1}{2}}^{\lambda\alpha})^{2}](\delta_{t}u_{j}^{0})^{2}+\right.
   \end{equation*}
    \begin{equation*}
    (a_{i+\frac{1}{2},i+\frac{1}{2}}^{\lambda\alpha})^{-1}\left[(1+4\alpha)(W_{i+\frac{1}{2},j}^{i+\frac{1}{2}})^{2}-4\alpha
    W_{i+\frac{1}{2},j}^{i+\frac{1}{2}}W_{i+\frac{1}{2},j}^{i}+\left((a_{i+\frac{1}{2},i}^{\lambda\alpha})^{-1}
    a_{i+\frac{1}{2},i+\frac{1}{2}}^{\lambda\alpha}-1\right)(W_{i+\frac{1}{2},j}^{i})^{2}\right]+
   \end{equation*}
    \begin{equation}\label{122e}
     \left.\underset{l=\frac{1}{2}}{\overset{i-1}\sum}\left[(a_{i+\frac{1}{2},l+\frac{1}{2}}^{\lambda\alpha})^{-1}-
    (a_{i+\frac{1}{2},l+1}^{\lambda\alpha})^{-1}\right](W_{i+\frac{1}{2},j}^{l+\frac{1}{2}})^{2}\right\}.
   \end{equation}
    Here, we assume that the sum equals zero if the upper summation index is less than the lower one. Furthermore, since
    $a_{i+\frac{1}{2},l}^{\lambda\alpha}<a_{i+\frac{1}{2},l+\frac{1}{2}}^{\lambda\alpha}$, for $l=\frac{1}{2},1,\frac{3}{2},2,...,i$, it is not difficult
    to observe that
     \begin{equation*}
    [a_{i+\frac{1}{2},\frac{1}{2}}^{\lambda\alpha}-(a_{i+\frac{1}{2},1}^{\lambda\alpha})^{-1}(a_{i+\frac{1}{2},\frac{1}{2}}^{\lambda\alpha})^{2}]
   (\delta_{t}u_{j}^{0})^{2}>0,\text{\,\,\,}\underset{l=\frac{1}{2}}{\overset{i-1}\sum}\left[(a_{i+\frac{1}{2},l+\frac{1}{2}}^{\lambda\alpha})^{-1}-
    (a_{i+\frac{1}{2},l+1}^{\lambda\alpha})^{-1}\right](W_{i+\frac{1}{2},j}^{l+\frac{1}{2}})^{2}>0,
   \end{equation*}
    and
     \begin{equation*}
    (1+4\alpha)(W_{i+\frac{1}{2},j}^{i+\frac{1}{2}})^{2}-4\alpha W_{i+\frac{1}{2},j}^{i+\frac{1}{2}}W_{i+\frac{1}{2},j}^{i}+
   \left((a_{i+\frac{1}{2},i}^{\lambda\alpha})^{-1}a_{i+\frac{1}{2},i+\frac{1}{2}}^{\lambda\alpha}-1\right)(W_{i+\frac{1}{2},j}^{i})^{2}=
   \end{equation*}
    \begin{equation*}
    (1+4\alpha)\left(W_{i+\frac{1}{2},j}^{i+\frac{1}{2}}-\frac{2\alpha}{1+4\alpha}W_{i+\frac{1}{2},j}^{i}\right)^{2}-\frac{1}{1+4\alpha}\left[4\alpha^{2}-
    (1+4\alpha)\left((a_{i+\frac{1}{2},i}^{\lambda\alpha})^{-1}a_{i+\frac{1}{2},i+\frac{1}{2}}^{\lambda\alpha}-1\right)\right](W_{i+\frac{1}{2},j}^{i})^{2}
    \geq 0.
   \end{equation*}
    The last estimate follows from relation $(\ref{122b})$. Using this, equation $(\ref{122e})$ becomes
    \begin{equation*}
    u_{j}^{\alpha_{i}}c\Delta_{0t}^{\lambda}u_{j}^{i+\frac{1}{2}+\alpha}\geq\frac{1}{2}c\Delta_{0t}^{\lambda}(u_{j}^{i+\frac{1}{2}+\alpha})^{2},
    \text{\,\,\,for\,\,\,}j=0,1,2,...,M.
   \end{equation*}
    Multiplying both sides of this estimate by $h$ and summing up from $j=1,2,...,M-1$, to get estimate $(\ref{122a})$. In a similar way, one easily proves
    inequality $(\ref{122c})$. This completes the proof of Lemma $\ref{l9}$.
    \end{proof}

    Using the above results (Lemmas $\ref{l1}$-$\ref{l9}$), we are ready to state and prove the main result (namely Theorem $\ref{tt1}$) of this
    paper.

   \begin{theorem}\label{tt1} (Unconditional stability and Convergence rate).
   Let $U$ be the numerical solution provided by the new approach $(\ref{s1})$-$(\ref{s5})$ and $u$ be the exact solution of the initial-boundary value
   problem $(\ref{1e})$-$(\ref{3e})$. Suppose that $\alpha=1-\lambda$ and the generalized sequences $(a_{\cdot,l}^{\lambda\alpha})_{l}$ satisfy relations
    $(\ref{122b})$ and $(\ref{122d})$. Then it holds
    \begin{equation}\label{122}
     \underset{0\leq i\leq N-1}{\max}\|U^{i+\frac{1}{2}}\|_{L^{2}},\text{\,\,\,} \underset{0\leq i\leq N}{\max}\|U^{i}\|_{L^{2}}\leq
      \underset{0\leq i\leq N}{\max}\|u^{i}\|_{L^{2}}+C\left[k^{2-\frac{\lambda}{2}}+k^{2+\frac{\lambda}{2}}+(1+k^{2})k^{\frac{\lambda}{2}}h^{4}\right].
   \end{equation}
    Furthermore, let $e=u-U$ be the error term, then the following estimates are satisfied
    \begin{equation}\label{123}
     \underset{0\leq i\leq N-1}{\max}\|e^{i+\frac{1}{2}}\|_{L^{2}},\text{\,\,\,} \underset{0\leq i\leq N}{\max}\|e^{i}\|_{L^{2}}\leq
      \widehat{C}\left(k^{2-\frac{\lambda}{2}}+h^{4}\right),
   \end{equation}
    where $C$ and $\widehat{C}$ are positive constants that do not depend on the time step $k$ and mesh size $h$.
   \end{theorem}

   \begin{remark}
    Inequality $(\ref{122})$ shows that the proposed approach $(\ref{s1})$-$(\ref{s5})$ is unconditionally stable whereas estimate $(\ref{123})$ indicates
    that the new method is convergent of order $2-\frac{\lambda}{2}$ in time and spatial fourth-order accurate.
   \end{remark}

    \begin{proof} (of Theorem $\ref{tt1}$).
     Setting $I_{j}^{\alpha,i}=cD_{0t}^{\lambda}u_{j}^{i+\frac{1}{2}+\alpha}-c\Delta_{0t}^{\lambda}u_{j}^{i+\frac{1}{2}+\alpha}$, it comes from
     $(\ref{72})$ that
     \begin{equation}\label{124}
     I_{j}^{\alpha,i}=\frac{1}{\Gamma(1-\lambda)}\left\{\underset{l=0}{\overset{i-1}\sum}\int_{t_{l+\frac{1}{2}}}^{t_{l+\frac{3}{2}}}
     \frac{E^{l}_{\tau,j}(\tau)}{(t_{i+\frac{1}{2}+\alpha}-\tau)^{\lambda}}d\tau+\int_{0}^{t_{\frac{1}{2}}}\frac{u_{\tau,j}(\tau)-
     \delta_{t}u_{j}^{0}}{(t_{i+\frac{1}{2}+\alpha}-\tau)^{\lambda}}d\tau+\int_{t_{i+\frac{1}{2}}}^{t_{i+\frac{1}{2}+\alpha}}\frac{u_{\tau,j}(\tau)
    -\delta_{t}u_{j}^{i}}{(t_{i+\frac{1}{2}+\alpha}-\tau)^{\lambda}}d\tau\right\},
       \end{equation}
     for $i\geq1$. Combining equations $(\ref{43a})$, $(\ref{s3})$ and equation $(\ref{79})$ (case where $\gamma_{i}=\alpha_{i}$), simple
     calculations give
      \begin{equation*}
     c\Delta_{0t}^{\lambda}u_{j}^{i+\frac{1}{2}+\alpha}+I_{j}^{\alpha,i}-c\Delta_{0t}^{\lambda}U_{j}^{i+\frac{1}{2}+\alpha}= L_{h}u_{j}^{\alpha_{i}}-
     L_{h}U_{j}^{\alpha_{i}}+q^{i+\frac{1}{2}+\alpha}\psi_{1,j}^{i}-p^{i+\frac{1}{2}+\alpha}\psi_{2,j}^{i}+\alpha(\frac{1}{2}+\alpha)k^{2}
     g^{i+\frac{1}{2}+\alpha}_{j}H_{1,j}^{i},
   \end{equation*}
    which is equivalent to
    \begin{equation}\label{125}
     c\Delta_{0t}^{\lambda}e_{j}^{i+\frac{1}{2}+\alpha}-L_{h}e_{j}^{\alpha_{i}}=-I_{j}^{\alpha,i}+q^{i+\frac{1}{2}+\alpha}\psi_{1,j}^{i}-p^{i+\frac{1}{2}
     +\alpha}\psi_{2,j}^{i}+\alpha(\frac{1}{2}+\alpha)k^{2}g^{i+\frac{1}{2}+\alpha}_{j}H_{1,j}^{i},
       \end{equation}
    where $H_{1,j}^{i}$, $\psi_{1,j}^{i}$, $\psi_{2,j}^{i}$ and $I_{j}^{\alpha,i}$ are defined by relations $(\ref{28d})$, $(\ref{42})$, $(\ref{43})$
   and $(\ref{124})$, respectively. Multiplying both sides of $(\ref{125})$ by $he_{j}^{\alpha_{i}}$ and summing up the obtained equation from
   $j=1,2,...,M-1$, yields
   \begin{equation*}
    \left(c\Delta_{0t}^{\lambda}e_{j}^{i+\frac{1}{2}+\alpha},e^{\alpha_{i}}\right)+\left(-L_{h}e_{j}^{\alpha_{i}},e^{\alpha_{i}}\right)=
   h\underset{j=1}{\overset{M-1}\sum}\left[-I_{j}^{\alpha,i}+q^{i+\frac{1}{2}+\alpha}\psi_{1,j}^{i}-p^{i+\frac{1}{2}
     +\alpha}\psi_{2,j}^{i}+\alpha(\frac{1}{2}+\alpha)k^{2}g^{i+\frac{1}{2}+\alpha}_{j}H_{1,j}^{i}\right]e^{\alpha_{i}}_{j}.
       \end{equation*}
      Utilizing equation $(\ref{81})$, this becomes
    \begin{equation*}
    \left(c\Delta_{0t}^{\lambda}e_{j}^{i+\frac{1}{2}+\alpha},e^{\alpha_{i}}\right)+\left(-L_{h}e_{j}^{\alpha_{i}},e^{\alpha_{i}}\right)=
   -h\underset{j=1}{\overset{M-1}\sum}I_{j}^{\alpha,i}e^{\alpha_{i}}_{j}+h\underset{j=1}{\overset{M-1}\sum}(L_{h}u_{j}^{i+\frac{1}{2}+\alpha}-
   L_{h}U_{j}^{\alpha_{i}})e^{\alpha_{i}}_{j}.
       \end{equation*}
   Applying the H\"{o}lder inequality to the right side of this equation and using Lemma $\ref{l3}$, this implies
   \begin{equation*}
    \left(c\Delta_{0t}^{\lambda}e_{j}^{i+\frac{1}{2}+\alpha},e^{\alpha_{i}}\right)+\gamma\|\delta_{x}e^{\alpha_{i}}\|_{L^{2}}^{2}+
    \beta\|e^{\alpha_{i}}\|_{L^{2}}^{2}\leq \left[\left(h\underset{j=1}{\overset{M-1}\sum}|I_{j}^{\alpha,i}|^{2}\right)^{\frac{1}{2}}+
    \left(h\underset{j=1}{\overset{M-1}\sum}|L_{h}u_{j}^{i+\frac{1}{2}+\alpha}-L_{h}U_{j}^{\alpha_{i}}|^{2}\right)^{\frac{1}{2}}\right]*
       \end{equation*}
    \begin{equation}\label{125a}
       \left(h\underset{j=1}{\overset{M-1}\sum}|e^{\alpha_{i}}_{j}|^{2}\right)^{\frac{1}{2}}=\left(\|I^{\alpha,i}\|_{L^{2}}+
   \|L_{h}u^{i+\frac{1}{2}+\alpha}-L_{h}U^{\alpha_{i}}\|_{L^{2}}\right)\|e^{\alpha_{i}}\|_{L^{2}}.
       \end{equation}
   By the Poincare-Friedrichs inequality, there is a positive parameter $\widehat{C}_{0}$ so that,
   $\|\delta_{x}e^{\alpha_{i}}\|_{L^{2}}^{2}\geq\widehat{C}_{0}\|e^{\alpha_{i}}\|_{L^{2}}^{2}$. This fact, together with estimate $(\ref{125a})$ and
   Lemmas $\ref{l1}$-$\ref{l2}$ result in
   \begin{equation}\label{126}
    \left(c\Delta_{0t}^{\lambda}e_{j}^{i+\frac{1}{2}+\alpha},e^{\alpha_{i}}\right)+C_{0}\|e^{\alpha_{i}}\|_{L^{2}}^{2}\leq
    \left[C_{\frac{1}{2}}k^{2-\lambda}+C_{3}\left(k^{2}+(1+k^{2})h^{4}\right)\|u\|_{\mathcal{C}^{6,3}_{D}}\right]\|e^{\alpha_{i}}\|_{L^{2}},
       \end{equation}
   where $C_{0}=\gamma\widehat{C}_{0}+\beta$, $C_{\frac{1}{2}}$ is the constant given in estimate $(\ref{l1})$ and all the
   constants in estimate $(\ref{80})$ are absorbed into a positive constant $C_{3}$.\\

     Now, replacing $u$ by $e$ in estimate $(\ref{122a})$ of Lemma $\ref{l9}$ and plugging the new estimate with $(\ref{126})$ provides
     \begin{equation*}
    \frac{h}{2}\underset{j=1}{\overset{M-1}\sum}c\Delta_{0t}^{\lambda}(e_{j}^{i+\frac{1}{2}+\alpha})^{2}+C_{0}\|e^{\alpha_{i}}\|_{L^{2}}^{2}\leq
    \left[C_{\frac{1}{2}}k^{2-\lambda}+C_{3}\left(k^{2}+(1+k^{2})h^{4}\right)\|u\|_{\mathcal{C}^{6,3}_{D}}\right]\|e^{\alpha_{i}}\|_{L^{2}}.
    \end{equation*}
   Combining this together with relation $(\ref{117})$, direct calculations give
    \begin{equation*}
    \frac{h}{2}\frac{k^{1-\lambda}}{\Gamma(2-\lambda)}\underset{j=1}{\overset{M-1}\sum}\underset{l=0}{\overset{i}\sum}
   a_{i+\frac{1}{2},l+\frac{1}{2}}^{\lambda\alpha}\delta_{t}(e_{j}^{l})^{2}+C_{0}\|e^{\alpha_{i}}\|_{L^{2}}^{2}\leq
    \left[C_{\frac{1}{2}}k^{2-\lambda}+C_{3}\left(k^{2}+(1+k^{2})h^{4}\right)\|u\|_{\mathcal{C}^{6,3}_{D}}\right]\|e^{\alpha_{i}}\|_{L^{2}}.
    \end{equation*}
   Performing simple computations, this yields
   \begin{equation*}
    h\underset{j=1}{\overset{M-1}\sum}\underset{l=0}{\overset{i}\sum}a_{i+\frac{1}{2},l+\frac{1}{2}}^{\lambda\alpha}
   \left[(e_{j}^{l+\frac{1}{2}})^{2}-(e_{j}^{l})^{2}\right]+\Gamma(2-\lambda)C_{0}k^{\lambda}\|e^{\alpha_{i}}\|_{L^{2}}^{2}
    \leq\Gamma(2-\lambda)\left[C_{\frac{1}{2}}k^{2}+C_{3}\left(k^{2+\lambda}+\right.\right.
    \end{equation*}
    \begin{equation*}
    \left.\left.(1+k^{2})k^{\lambda}h^{4}\right)\|u\|_{\mathcal{C}^{6,3}_{D}}\right]\|e^{\alpha_{i}}\|_{L^{2}}.
    \end{equation*}
    We recall that the summation index $l$ varies in the range $l=0,\frac{1}{2},1,\frac{3}{2},...,i$. Using the summation by parts, this results in
   \begin{equation*}
    h\underset{j=1}{\overset{M-1}\sum}\left\{a_{i+\frac{1}{2},i+\frac{1}{2}}^{\lambda\alpha}(e_{j}^{i+\frac{1}{2}})^{2}-
   a_{i+\frac{1}{2},\frac{1}{2}}^{\lambda\alpha}(e_{j}^{0})^{2}-\underset{l=0}{\overset{i-\frac{1}{2}}\sum}\left[a_{i+\frac{1}{2},l+1}^{\lambda\alpha}
   -a_{i+\frac{1}{2},l+\frac{1}{2}}^{\lambda\alpha}\right](e_{j}^{l+\frac{1}{2}})^{2}\right\}+\Gamma(2-\lambda)C_{0}k^{\lambda}
   \|e^{\alpha_{i}}\|_{L^{2}}^{2}\leq
    \end{equation*}
    \begin{equation}\label{126a}
    \Gamma(2-\lambda)\left[C_{\frac{1}{2}}k^{2}+C_{3}\left(k^{2+\lambda}+(1+k^{2})k^{\lambda}h^{4}\right)\|u\|_{\mathcal{C}^{6,3}_{D}}\right]
    \|e^{\alpha_{i}}\|_{L^{2}}.
    \end{equation}
   It follows from the initial condition given in $(\ref{s5})$ that $e_{j}^{0}=0$, for $j=0,1,2,...,M$. Using this, inequality $(\ref{126a})$ is
   equivalent to
    \begin{equation}\label{127}
    a_{i+\frac{1}{2},i+\frac{1}{2}}^{\lambda\alpha}\|e^{i+\frac{1}{2}}\|_{L^{2}}^{2}+\Gamma(2-\lambda)C_{0}k^{\lambda}
   \|e^{\alpha_{i}}\|_{L^{2}}^{2}\leq\underset{l=0}{\overset{i-\frac{1}{2}}\sum}\left[a_{i+\frac{1}{2},l+1}^{\lambda\alpha}-
    a_{i+\frac{1}{2},l+\frac{1}{2}}^{\lambda\alpha}\right]\|e^{l+\frac{1}{2}}\|_{L^{2}}^{2}+\Gamma(2-\lambda)F(k,h)\|e^{\alpha_{i}}\|_{L^{2}},
    \end{equation}
   where
    \begin{equation}\label{128}
    F(k,h)=C_{\frac{1}{2}}k^{2}+C_{3}\left(k^{2+\lambda}+(1+k^{2})k^{\lambda}h^{4}\right)\|u\|_{\mathcal{C}^{6,3}_{D}}.
    \end{equation}
    But, it is not hard to observe that
    \begin{equation}\label{129}
    F(k,h)\|e^{\alpha_{i}}\|_{L^{2}}=2(\sqrt{C_{0}k^{\lambda}}\|e^{\alpha_{i}}\|_{L^{2}})(\frac{1}{2\sqrt{C_{0}k^{\lambda}}}F(k,h))\leq
    C_{0}k^{\lambda}\|e^{\alpha_{i}}\|_{L^{2}}^{2}+\frac{1}{4C_{0}k^{\lambda}}F(k,h)^{2}.
    \end{equation}
    Substituting $(\ref{129})$ into $(\ref{127})$ and after simplifying provides
     \begin{equation}\label{131}
    a_{i+\frac{1}{2},i+\frac{1}{2}}^{\lambda\alpha}\|e^{i+\frac{1}{2}}\|_{L^{2}}^{2}\leq\underset{l=0}{\overset{i-\frac{1}{2}}\sum}
    \left[a_{i+\frac{1}{2},l+1}^{\lambda\alpha}-a_{i+\frac{1}{2},l+\frac{1}{2}}^{\lambda\alpha}\right]\|e^{l+\frac{1}{2}}\|_{L^{2}}^{2}
    +\frac{\Gamma(2-\lambda)}{4C_{0}k^{\lambda}}F(k,h)^{2},\text{\,\,\,for\,\,\,}i\geq1.
    \end{equation}
    Similarly, one easily shows that
    \begin{equation}\label{132}
    a_{i+1,i+1}^{\lambda\alpha}\|e^{i+1}\|_{L^{2}}^{2}\leq\underset{l=0}{\overset{i}\sum}\left[a_{i+1,l+1}^{\lambda\alpha}-a_{i+1,
    l+\frac{1}{2}}^{\lambda\alpha}\right]\|e^{l+\frac{1}{2}}\|_{L^{2}}^{2}+\frac{\Gamma(2-\lambda)}{4C_{0}k^{\lambda}}F(k,h)^{2},
    \text{\,\,\,for\,\,\,}i\geq0.
    \end{equation}
    We should prove by mathematical induction that for $i=0,1,2,...$, the following estimates are satisfied
    \begin{equation}\label{133}
    \|e^{i+\frac{1}{2}}\|_{L^{2}},\text{\,\,\,}\|e^{i+1}\|_{L^{2}}\leq \beta_{0}\gamma_{0}k^{\frac{-\lambda}{2}}\left[1+\gamma_{0}^{2}
    \left(a_{1,1}^{\lambda\alpha}-a_{1,\frac{1}{2}}^{\lambda\alpha}\right)\right]^{\frac{1}{2}}F(k,h),
    \end{equation}
     where
    \begin{equation}\label{133a}
    \beta_{0}=\frac{1}{2}\sqrt{\frac{\Gamma(2-\lambda)}{C_{0}}}\text{\,\,\,and\,\,\,}\gamma_{0}=\left(\frac{2\alpha^{\lambda}(2-\lambda)}
     {(1-\lambda)(2-3\lambda)}\right)^{\frac{1}{2}}.
    \end{equation}
    Firstly, for $i=0$, we should find a relation satisfied by $a_{1,1}^{\lambda\alpha}\|e^{\frac{1}{2}}\|_{L^{2}}^{2}$, where $a_{1,1}^{\lambda\alpha}$
     is given by $(\ref{22a})$. It comes from equation $(\ref{81})$ that $Lu_{j}^{\frac{1}{2}+\alpha}-L_{h}u_{j}^{\alpha_{0}}=
     q^{\frac{1}{2}+\alpha}\psi_{1,j}^{0}-p^{\frac{1}{2}+\alpha}\psi_{2,j}^{0}+\alpha(\frac{1}{2}+\alpha)k^{2}
     g_{j}^{\frac{1}{2}+\alpha}H_{1,j}^{0}$. Setting $I_{j}^{\alpha,0}=cD_{0t}^{\lambda}u_{j}^{\frac{1}{2}+\alpha}-c\Delta_{0t}^{\lambda}
     u_{j}^{\frac{1}{2}+\alpha}$, using equation $(\ref{18})$ and subtracting equation $(\ref{44aa})$ from $(\ref{44a})$, it is easy to see that
    \begin{equation}\label{136}
    f_{\frac{1}{2},0}^{\lambda\alpha}\delta_{t}e^{0}_{j}=L_{h}e_{j}^{\alpha_{0}}+Lu_{j}^{\frac{1}{2}+\alpha}-L_{h}u_{j}^{\alpha_{0}}+I_{j}^{\alpha,0},
    \end{equation}
    where $\psi_{1,j}^{0}$, $\psi_{2,j}^{0}$ and $H_{1,j}^{0}$ are given by $(\ref{42})$, $(\ref{43})$ and $(\ref{28d})$, respectively. Since
    $e^{0}_{j}=0$, for $j=0,1,2,...,M$, and $a_{1,1}^{\lambda\alpha}=\widetilde{f}_{\frac{1}{2},0}^{\lambda\alpha}$, multiplying both sides of
    $(\ref{136})$ by $(1+2\alpha)he_{j}^{\alpha_{0}}$ to get
    \begin{equation*}
    \frac{2h}{k}\frac{k^{1-\lambda}}{\Gamma(2-\lambda)}a_{\frac{1}{2},\frac{1}{2}}^{\lambda\alpha}(e_{j}^{\alpha_{0}})^{2}=(1+2\alpha)h\left[
    L_{h}e_{j}^{\alpha_{0}}+Lu_{j}^{\frac{1}{2}+\alpha}-L_{h}u_{j}^{\alpha_{0}}+I_{j}^{\alpha,0}\right]e_{j}^{\alpha_{0}}.
    \end{equation*}
   Summing this from $j=1,2,...,M-1$, it is not difficult to observe that
   \begin{equation*}
    a_{\frac{1}{2},\frac{1}{2}}^{\lambda\alpha}\|e^{\alpha_{0}}\|_{L^{2}}^{2}=\frac{(1+2\alpha)k^{\lambda}\Gamma(2-\lambda)}{2}\left[
    \left(L_{h}e^{\alpha_{0}},e^{\alpha_{0}}\right)+\left(Lu^{\frac{1}{2}+\alpha}-L_{h}u^{\alpha_{0}},e^{\alpha_{0}}\right)
     +\left(I^{\alpha,0},e^{\alpha_{0}}\right)\right].
    \end{equation*}
    Applying Lemma $\ref{l3}$ and the Poincare-Friedrichs inequality, this implies
    \begin{equation}\label{137}
    a_{\frac{1}{2},\frac{1}{2}}^{\lambda\alpha}\|e^{\alpha_{0}}\|_{L^{2}}^{2}+\frac{C_{0}(1+2\alpha)k^{\lambda}\Gamma(2-\lambda)}{2}
    \|e^{\alpha_{0}}\|_{L^{2}}^{2}\leq\frac{(1+2\alpha)k^{\lambda}\Gamma(2-\lambda)}{2}\left[\left(Lu^{\frac{1}{2}+\alpha}-L_{h}u^{\alpha_{0}},
    e^{\alpha_{0}}\right)+\left(I^{\alpha,0},e^{\alpha_{0}}\right)\right].
    \end{equation}
    where $C_{0}$ is the positive constant given in $(\ref{126})$. But it holds
    \begin{equation*}
    \left(I^{\alpha,0},e^{\alpha_{0}}\right)=2\left(\frac{1}{\sqrt{2C_{0}}}I^{\alpha,0},\sqrt{\frac{C_{0}}{2}}e^{\alpha_{0}}\right)\leq
    \frac{1}{2C_{0}}\|I^{\alpha,0}\|_{L^{2}}^{2}+\frac{C_{0}}{2}\|e^{\alpha_{0}}\|_{L^{2}}^{2},
    \end{equation*}
    and
    \begin{equation*}
    \left(Lu^{\frac{1}{2}+\alpha}-L_{h}u^{\alpha_{0}},e^{\alpha_{0}}\right)\leq\frac{1}{2C_{0}}\|Lu^{\frac{1}{2}+\alpha}-
    L_{h}u^{\alpha_{0}}\|_{L^{2}}^{2}+\frac{C_{0}}{2}\|e^{\alpha_{0}}\|_{L^{2}}^{2}.
    \end{equation*}
    Substituting this into estimate $(\ref{137})$ and rearranging terms yields
    \begin{equation*}
    a_{\frac{1}{2},\frac{1}{2}}^{\lambda\alpha}\|e^{\alpha_{0}}\|_{L^{2}}^{2}\leq\frac{(1+2\alpha)k^{\lambda}\Gamma(2-\lambda)}{4C_{0}}
   \left[\|I^{\alpha,0}\|_{L^{2}}^{2}+\|Lu^{\frac{1}{2}+\alpha}-L_{h}u^{\alpha_{0}}\|_{L^{2}}^{2}\right].
    \end{equation*}
   Utilizing Lemmas $\ref{l1}$-$\ref{l2}$, we obtain
   \begin{equation*}
    a_{\frac{1}{2},\frac{1}{2}}^{\lambda\alpha}\delta_{t}\|e^{\alpha_{0}}\|_{L^{2}}^{2}\leq\frac{(1+2\alpha)k^{\lambda}\Gamma(2-\lambda)}{4C_{0}}
    \left[C_{\frac{1}{2}}^{2}k^{4-2\lambda}+C_{3}^{2}\left(k^{2}+(1+k^{2})h^{4}\right)^{2}\|u\|_{\mathcal{C}^{6,3}_{D}}^{2}\right],
    \end{equation*}
   where we absorbed all the constants in estimate $(\ref{80})$ into a positive constant $C_{3}$. Since $e^{\alpha_{0}}=(1+2\alpha)e^{\frac{1}{2}}$
   and $a^{2}+b^{2}\leq(a+b)^{2}$, for any nonnegative real numbers $a$ and $b$, this becomes
   \begin{equation*}
    a_{\frac{1}{2},\frac{1}{2}}^{\lambda\alpha}\|e^{\frac{1}{2}}\|_{L^{2}}^{2}\leq\frac{(1+2\alpha)^{-1}k^{-\lambda}\Gamma(2-\lambda)}{4C_{0}}
    F(k,h)^{2},
    \end{equation*}
    $F(k,h)$ is defined by $(\ref{128})$. Using relation $(\ref{100})$ and performing direct calculations to obtain
    \begin{equation}\label{137a}
    \|e^{\frac{1}{2}}\|_{L^{2}}\leq\left(\frac{\Gamma(2-\lambda)}{4C_{0}}\right)^{\frac{1}{2}}
    \left(\frac{2\alpha^{\lambda}(2-\lambda)}{(1-\lambda)(2-3\lambda)}\right)^{\frac{1}{2}}k^{-\frac{\lambda}{2}}F(k,h)\leq\beta_{0}\gamma_{0}
    k^{-\frac{\lambda}{2}}\left[1+\gamma_{0}^{2}\left(a_{1,1}^{\lambda\alpha}-a_{1,\frac{1}{2}}^{\lambda\alpha}\right)\right]^{\frac{1}{2}}F(k,h),
    \end{equation}
    since $(1+2\alpha)^{-1}<1$. The first estimate in $(\ref{137a})$ together with inequality $(\ref{132})$, for $i=0$, result in
    \begin{equation}\label{138}
    a_{1,1}^{\lambda\alpha}\|e^{1}\|_{L^{2}}^{2}\leq\left(a_{1,1}^{\lambda\alpha}-a_{1,\frac{1}{2}}^{\lambda\alpha}\right)\|e^{\frac{1}{2}}\|_{L^{2}}^{2}+
    \beta_{0}^{2}k^{-\lambda}F(k,h)^{2}\leq \beta_{0}^{2}k^{-\lambda}\left[1+\gamma_{0}^{2}\left(a_{1,1}^{\lambda\alpha}-
    a_{1,\frac{1}{2}}^{\lambda\alpha}\right)\right]F(k,h)^{2}.
    \end{equation}
     But it comes from $(\ref{100})$ that $a_{1,1}^{\lambda\alpha}>\gamma_{0}^{-2}$, where $\gamma_{0}$ is given in relation $(\ref{133a})$. Dividing each
    side of $(\ref{138})$ by $\gamma_{0}^{-2}$ and taking the square root to get the desired result.\\

   Now, for any positive integer $i$, we assume that estimates $(\ref{133})$ holds for $j=0,1,2,...,i-1$. Using the assumption, relation $(\ref{131})$
   implies
    \begin{equation*}
    a_{i+\frac{1}{2},i+\frac{1}{2}}^{\lambda\alpha}\|e^{i+\frac{1}{2}}\|_{L^{2}}^{2}\leq\underset{l=0}{\overset{i-\frac{1}{2}}\sum}
    \left[a_{i+\frac{1}{2},l+1}^{\lambda\alpha}-a_{i+\frac{1}{2},l+\frac{1}{2}}^{\lambda\alpha}\right]\|e^{l+\frac{1}{2}}\|_{L^{2}}^{2}
    +\frac{\Gamma(2-\lambda)}{4C_{0}k^{\lambda}}F(k,h)^{2}\leq\beta_{0}^{2}\gamma_{0}^{2}k^{-\lambda}\left[1+\right.
    \end{equation*}
    \begin{equation*}
    \left.\gamma_{0}^{2}\left(a_{1,1}^{\lambda\alpha}-a_{1,\frac{1}{2}}^{\lambda\alpha}\right)\right]F(k,h)^{2}\underset{l=0}{\overset{i-\frac{1}{2}}\sum}
   \left[a_{i+\frac{1}{2},l+1}^{\lambda\alpha}-a_{i+\frac{1}{2},l+\frac{1}{2}}^{\lambda\alpha}\right]+\beta_{0}^{2}k^{-\lambda}F(k,h)^{2}
   =\beta_{0}^{2}\gamma_{0}^{2}k^{-\lambda}\left[1+\gamma_{0}^{2}\left(a_{1,1}^{\lambda\alpha}-a_{1,\frac{1}{2}}^{\lambda\alpha}\right)\right]*
    \end{equation*}
   \begin{equation*}
     \left[a_{i+\frac{1}{2},i+\frac{1}{2}}^{\lambda\alpha}-a_{i+\frac{1}{2},\frac{1}{2}}^{\lambda\alpha}\right]F(k,h)^{2}
   +\beta_{0}^{2}k^{-\lambda}F(k,h)^{2}=\beta_{0}^{2}k^{-\lambda}\left\{1+\gamma_{0}^{2}\left[1+\gamma_{0}^{2}\left(a_{1,1}^{\lambda\alpha}
   -a_{1,\frac{1}{2}}^{\lambda\alpha}\right)\right]a_{i+\frac{1}{2},i+\frac{1}{2}}^{\lambda\alpha}\right\}F(k,h)^{2}
    \end{equation*}
   \begin{equation}\label{139}
    -\beta_{0}^{2}\gamma_{0}^{2}k^{-\lambda}\left[1+\gamma_{0}^{2}\left(a_{1,1}^{\lambda\alpha}
   -a_{1,\frac{1}{2}}^{\lambda\alpha}\right)\right]a_{i+\frac{1}{2},\frac{1}{2}}^{\lambda\alpha}F(k,h)^{2}.
    \end{equation}
   Utilizing relation $(\ref{100})$, it holds $a_{i+\frac{1}{2},\frac{1}{2}}^{\lambda\alpha},a_{i+\frac{1}{2},i+\frac{1}{2}}^{\lambda\alpha}>
   \gamma_{0}^{-2}$. So $-\beta_{0}^{2}\gamma_{0}^{2}k^{-\lambda}\left[1+\gamma_{0}^{2}\left(a_{1,1}^{\lambda\alpha}
   -a_{1,\frac{1}{2}}^{\lambda\alpha}\right)\right]a_{i+\frac{1}{2},\frac{1}{2}}^{\lambda\alpha}F(k,h)^{2}<-\beta_{0}^{2}k^{-\lambda}
   \left[1+\gamma_{0}^{2}\left(a_{1,1}^{\lambda\alpha}-a_{1,\frac{1}{2}}^{\lambda\alpha}\right)\right]F(k,h)^{2}<-\beta_{0}^{2}k^{-\lambda}
   F(k,h)^{2}$. This fact, together with estimate $(\ref{139})$ gives
   \begin{equation*}
    a_{i+\frac{1}{2},i+\frac{1}{2}}^{\lambda\alpha}\|e^{i+\frac{1}{2}}\|_{L^{2}}^{2}\leq\beta_{0}^{2}\gamma_{0}^{2}k^{-\lambda}
    \left[1+\gamma_{0}^{2}\left(a_{1,1}^{\lambda\alpha}-a_{1,\frac{1}{2}}^{\lambda\alpha}\right)\right]
   a_{i+\frac{1}{2},i+\frac{1}{2}}^{\lambda\alpha}F(k,h)^{2}
    \end{equation*}
   Multiplying both sides of this estimate by $(a_{i+\frac{1}{2},i+\frac{1}{2}}^{\lambda\alpha})^{-1}$ and taking the square root to obtain
   \begin{equation}\label{140}
    \|e^{i+\frac{1}{2}}\|_{L^{2}}\leq\beta_{0}\gamma_{0}k^{-\frac{\lambda}{2}}
    \left[1+\gamma_{0}^{2}\left(a_{1,1}^{\lambda\alpha}-a_{1,\frac{1}{2}}^{\lambda\alpha}\right)\right]^{\frac{1}{2}}F(k,h).
    \end{equation}
   Analogously, using relation $(\ref{140})$ together with the assumption and performing simple simple computations, estimate $(\ref{132})$ becomes
   \begin{equation*}
    a_{i+1,i+1}^{\lambda\alpha}\|e^{i+1}\|_{L^{2}}^{2}\leq\beta_{0}^{2}k^{-\lambda}\left\{1+\gamma_{0}^{2}\left[1+\gamma_{0}^{2}\left(a_{1,1}^{\lambda\alpha}
   -a_{1,\frac{1}{2}}^{\lambda\alpha}\right)\right]a_{i+1,i+1}^{\lambda\alpha}\right\}F(k,h)^{2}
    \end{equation*}
   \begin{equation*}
    -\beta_{0}^{2}\gamma_{0}^{2}k^{-\lambda}\left[1+\gamma_{0}^{2}\left(a_{1,1}^{\lambda\alpha}-a_{1,\frac{1}{2}}^{\lambda\alpha}\right)\right]
    a_{i+1,\frac{1}{2}}^{\lambda\alpha}F(k,h)^{2}.
    \end{equation*}
    Since $-\gamma_{0}^{2}a_{i+1,\frac{1}{2}}^{\lambda\alpha}<-1$, straightforward computations yield
    \begin{equation*}
    \|e^{i+1}\|_{L^{2}}\leq\beta_{0}\gamma_{0}k^{-\frac{\lambda}{2}}\left[1+\gamma_{0}^{2}\left(a_{1,1}^{\lambda\alpha}-
    a_{1,\frac{1}{2}}^{\lambda\alpha}\right)\right]^{\frac{1}{2}}F(k,h).
    \end{equation*}
    This ends the proof of estimates in $(\ref{133})$. Substituting equation $(\ref{128})$ into relation $(\ref{133})$ and absorbing all the constants into
    a positive constant $C$ gives
     \begin{equation}\label{141}
    \|e^{i+\frac{1}{2}}\|_{L^{2}},\text{\,\,\,}\|e^{i+1}\|_{L^{2}}\leq C\left[k^{2-\frac{\lambda}{2}}+k^{2+\frac{\lambda}{2}}
     +(1+k^{2})k^{\frac{\lambda}{2}}h^{4}\right].
    \end{equation}
    But $\left|\|U^{l}\|_{L^{2}}-\|u^{l}\|_{L^{2}}\right|\leq\|e^{l}\|_{L^{2}}$, for $l=i+\frac{1}{2},i+1$. So,
     \begin{equation*}
    \|U^{l}\|_{L^{2}}\leq\|u^{l}\|_{L^{2}}+C\left[k^{2-\frac{\lambda}{2}}+k^{2+\frac{\lambda}{2}}+(1+k^{2})k^{\frac{\lambda}{2}}h^{4}\right].
    \end{equation*}
    The proof of estimate $(\ref{122})$ is completed by taking the maximum over $i$. Now, since $k<1$ then, $(1+k^{2})k^{\frac{\lambda}{2}}\leq2$ and
    $k^{2+\frac{\lambda}{2}}\leq k^{2-\frac{\lambda}{2}}$, utilizing this, inequality $(\ref{141})$ implies
    \begin{equation*}
    \|e^{i+\frac{1}{2}}\|_{L^{2}},\text{\,\,\,}\|e^{i+1}\|_{L^{2}}\leq 2C\left[k^{2-\frac{\lambda}{2}}+h^{4}\right].
    \end{equation*}
    Taking the maximum over $i$, this completes the proof of Theorem $\ref{tt1}$.
   \end{proof}

   \section{Numerical experiments and Convergence rate}\label{sec4}
   This section deals with a two-level fourth-order scheme for time-fractional convection-diffusion-reaction equation. We carry out numerical experiments
    on the new approach to illustrate our theoretical statements. Two examples described in \cite{aaa} and \cite{maam} are considered to demonstrate the 
    effectiveness and utility of the proposed technique. We observe from each case satisfactory results. Thus, the proposed approach provides better 
   performances than a large class of numerical scheme widely studied in the literature \cite{18zzy,16zzy,maam,20zzy,14zzy} for the initial-boundary value 
   problem $(\ref{1e})$-$(\ref{3e})$. We confirm
   the predicted convergence rate from the theory (see Section $\ref{sec3}$, Theorem $\ref{tt1}$). More precisely, \textbf{Tables} $1$-$4$ and Figures
    $\ref{figure1}$-$\ref{figure4}$ present the exact solution, the approximate ones and the errors between the computed solution and the analytical one
    with different values of time step $k$ and space step $h$ satisfying $k=h^{\frac{4}{2-\frac{\lambda}{2}}}.$ In addition, we look at the error
    estimates of the proposed method for the parameters $\lambda\in\{6.6\times10^{-1},9\times10^{-1}\}$, $\alpha=1-\lambda$ and $T=L_{1}=1$.\\

    Finally, to analyze the stability and convergence rate of our numerical scheme, we take the mesh size $h=2^{-l},$ $l=3,4,5,6$ and time step
   $k\in\left\{2^{-\frac{l}{2-\frac{\lambda}{2}}},\text{\,\,}l=2,3,..,40\right\}$, by a mid-point refinement. We set
   $k=h^{\frac{4}{2-\frac{\lambda}{2}}}$ and we compute the numerical solutions $\||U|\|_{L^{2}(0,T;L^{2})}$, the exact ones
   $\||u|\|_{L^{2}(0,T;L^{2})}$, the error estimates $\||E(h)|\|_{L^{2}(0,T;L^{2})}$ related to the two-level scheme and the convergence rate using
   the formula $CR=log_{2}(r(h))$, where $r(h)=\||E(2h)|\|_{L^{2}(0,T;L^{2})}/\||E(h)|\|_{L^{2}(0,T;L^{2})}$, to see that the new algorithm is
   unconditionally stable, convergent of order $2-\frac{\lambda}{2}$ in time and spatial fourth-order accuracy. In addition, we plot the computed
   solutions, the analytical ones and the error versus $i.$ We observe from this study that the proposed method is both efficient and effective than a 
   broad range of numerical methods \cite{11mc,31mc,5jrs,39jrs,9zzy,8zzy} applied to the considered problem.\\

   $\bullet$ \textbf{Example $1.$} Let $\Omega=(0,1)$ be the unit interval and $T=1$ be the final time. The parameters $\lambda$ and $\alpha$
     are given by $0.66$, $0.9$ and $\alpha=1-\lambda$. In \cite{aaa}, the functions $s$, $q$, $p$, $g$ are defined as:
     $s(x,t)=\Gamma(2-\lambda)^{-1}t^{1-\lambda}\sin x+t[\sin x+\cos x]$, $q(t)=1,$ $p(t)=1$, $g(x,t)=0$ and the exact solution $u$ is given by
      \begin{equation*}
     u(x,t)=t\sin(x).
     \end{equation*}
      The initial and boundary conditions are directly obtained from this analytical solution.\\

       \textbf{Table 1} $\label{t1}$. Unconditional stability and convergence rate $O(h^{4}+k^{2-\frac{\lambda}{2}})$ for the two-level
       fourth-order approach with $\log_{2}(r(h)),$ varying spacing $h=\Delta x$ and time step $k=\Delta t$. In this test we take
        $\lambda=0.9$, $\alpha=1-\lambda=0.1$ and $k=h^{\frac{4}{2-\frac{\lambda}{2}}}.$

             \begin{equation*}
            \begin{tabular}{|c|c|c|c|c|}
            \hline
            % after \\: \hline or \cline{col1-col2} \cline{col3-col4} ...
            $k$ & $\||u|\|_{L^{2}}$ &$\||U|\|_{L^{2}}$ & $\||E(h)|\|_{L^{2}}$ & RC\\
            \hline
            $h^{-1}$ & $2.4192\times10^{-1}$ & $2.4192\times10^{-1}$  &  $3.2480\times10^{-2}$ &   --  \\
            \hline
            $h^{-2}$ & $1.3890\times10^{-1}$  & $1.3889\times10^{-1}$ & $2.0131\times10^{-3}$ & 4.0121 \\
            \hline
            $h^{-3}$ & $7.1804\times10^{-2}$ & $7.1803\times10^{-2}$  &  $1.2685\times10^{-4}$ & 3.9882 \\
            \hline
           $h^{-4}$ & $3.6520\times10^{-2}$ & $3.6521\times10^{-2}$   &  $7.9282\times10^{-6}$ & 4.0000 \\
           \hline
           $h^{-5}$ & $1.8436\times10^{-2}$ & $1.8436\times10^{-2}$   &  $4.9514\times10^{-7}$ & 4.0011 \\
            \hline
          \end{tabular}
            \end{equation*}

            \textbf{Table 2} $\label{t2}$. Stability and Convergence rate $O(h^{4}+k^{2-\frac{\lambda}{2}})$ of the new technique with
        $\log_{2}(r(h)),$ varying spacing $h=\Delta x$ and time step $k=\Delta t$. Here we take $\lambda=0.66$, $\alpha=1-\lambda=0.34$
        and $k=h^{\frac{4}{2-\frac{\lambda}{2}}}.$

         \begin{equation*}
            \begin{tabular}{|c|c|c|c|c|}
            \hline
            % after \\: \hline or \cline{col1-col2} \cline{col3-col4} ...
            $k$ & $\||u|\|_{L^{2}}$ &$\||U|\|_{L^{2}}$ & $\||E(h)|\|_{L^{2}}$ & RC\\
            \hline
            $h^{-1}$ & $2.4211\times10^{-1}$ & $2.4212\times10^{-1}$  &  $3.2389\times10^{-2}$ &   --  \\
            \hline
            $h^{-2}$ & $1.3894\times10^{-1}$ & $1.3893\times10^{-1}$  & $2.03254\times10^{-3}$ & 3.9982 \\
            \hline
            $h^{-3}$ & $7.2046\times10^{-2}$ & $7.1923\times10^{-2}$  &  $1.2665\times10^{-4}$ & 4.0044 \\
            \hline
           $h^{-4}$ & $3.6541\times10^{-2}$ & $3.6541\times10^{-2}$   &  $7.8599\times10^{-6}$ & 4.0102 \\
           \hline
           $h^{-5}$ & $1.8458\times10^{-2}$ & $1.8458\times10^{-2}$   &  $4.9075\times10^{-7}$ & 4.0016 \\
            \hline
          \end{tabular}
            \end{equation*}

      $\bullet$ \textbf{Example $2.$} Suppose $\Omega$ be the open interval $(0,1)$ and $T$ be the final time, $T=1.$ We assume that the parameters
    $\lambda\in\{0.66,0.9\}$ and $\alpha=1-\lambda$. We choose the function $q(t)=e^{t},$ $p(t)=0$, $g(x,t)=1-\sin(2t)$ and
    $s(x,t)=[\pi^{2}t^{2}e^{t}+t^{2}(1-\sin(2t))+2\Gamma(3-\lambda)^{-1}t^{2-\lambda}]\sin(\pi x)$ such that the analytical solution is given in
    \cite{aaa} by
    \begin{equation*}
    u(x,t)=t^{2}\sin(\pi x).
    \end{equation*}
    The initial and boundary conditions are determined from the exact solution $u$.\\

       \textbf{Table 3} $\label{t3}$. Unconditional stability and convergence rate $O(k^{2-\frac{\lambda}{2}}+h^{4})$ for the two-level
       approach with $\log_{2}(r(h)),$ varying time step $k=\Delta t$ and spacing $h=\Delta x$. In this example we take
        $\lambda=0.9$, $\alpha=1-\lambda=0.1$ and $k=h^{\frac{4}{2-\frac{\lambda}{2}}}.$

      \begin{equation*}
            \begin{tabular}{|c|c|c|c|c|}
            \hline
            % after \\: \hline or \cline{col1-col2} \cline{col3-col4} ...
            $k$ & $\||u|\|_{L^{2}}$ &$\||U|\|_{L^{2}}$ & $\||E(h)|\|_{L^{2}}$ & RC\\
            \hline
            $h^{-1}$ & $9.2363\times10^{-1}$ & $9.2361\times10^{-1}$  &  $4.4483\times10^{-3}$ &   --  \\
            \hline
            $h^{-2}$ & $ 3.7291\times10^{-1}$ & $ 3.7292\times10^{-1}$  & $3.2732\times10^{-4}$ & 3.7645 \\
            \hline
            $h^{-3}$ & $1.1238\times10^{-1}$ & $1.1237\times10^{-1}$  &  $2.3565\times10^{-5}$ & 3.7960 \\
            \hline
           $h^{-4}$ & $2.9527\times10^{-2}$ & $2.9527\times10^{-2}$   &  $1.6832\times10^{-6}$ & 3.8074 \\
           \hline
           $h^{-5}$ & $7.5431\times10^{-3}$ & $7.5431\times10^{-3}$   &  $1.1829\times10^{-7}$ & 3.8308 \\
            \hline
          \end{tabular}
            \end{equation*}

            \textbf{Table 4} $\label{t4}$. Stability and Convergence rate $O(k^{2-\frac{\lambda}{2}}+h^{4})$ of the new technique with
       $\log_{2}(r(h)),$ varying spacing $h=\Delta x$ and time step $k=\Delta t$. Here we take $\lambda=0.66$, $\alpha=1-\lambda=0.34$
       and $k=h^{\frac{4}{2-\frac{\lambda}{2}}}.$

          \begin{equation*}
            \begin{tabular}{|c|c|c|c|c|}
            \hline
            % after \\: \hline or \cline{col1-col2} \cline{col3-col4} ...
            $k$ & $\||u|\|_{L^{2}}$ &$\||U|\|_{L^{2}}$ & $\||E(h)|\|_{L^{2}}$ & RC\\
            \hline
            $h^{-1}$ & $8.7304\times10^{-1}$ & $8.7302\times10^{-1}$  &  $1.0445\times10^{-2}$ &   --  \\
            \hline
            $h^{-2}$ & $3.7291\times10^{-1}$ & $3.7292\times10^{-1}$  & $7.4499\times10^{-4}$ & 3.8094 \\
            \hline
            $h^{-3}$ & $1.1234\times10^{-1}$ & $1.1234\times10^{-1}$  &  $5.2836\times10^{-5}$ & 3.8176 \\
            \hline
           $h^{-4}$ & $2.9519\times10^{-2}$ & $2.9519\times10^{-2}$   &  $3.7499\times10^{-6}$ & 3.8166 \\
           \hline
           $h^{-5}$ & $7.5327\times10^{-3}$ & $7.5327\times10^{-3}$   &  $2.6315\times10^{-7}$ & 3.8329 \\
            \hline
          \end{tabular}
            \end{equation*}
          We observe from this table that the proposed method is temporal second order convergent and spatial fourth order accurate.

     \section{General conclusions and future works}\label{sec5}
     In this paper we proposed a two-level fourth-order approach for solving the time-fractional convection-diffusion-reaction equation with 
     variable coefficients and source terms. Both stability analysis and error estimates of the numerical scheme have been deeply analyzed. The
     theory has suggested that the proposed method is unconditionally stable, convergence with order $O(k^{2-\frac{\lambda}{2}})$ in time and fourth
     accurate in space (see Theorem $\ref{tt1}$). The numerical tests are performed for values of the parameter $\lambda$ lying in the interval $(0,1)$
     and they confirmed the theoretical result provided in Section $\ref{sec3}$ (see Theorem $\ref{tt1}$, \textbf{Tables 1-4} and Figures 
    $\ref{figure1}$-$\ref{figure4}$). Especially, the graphs (Figures $\ref{figure1}$-$\ref{figure4}$) show that the new method is both unconditionally 
    stable and convergent whereas \textbf{Tables 1-4} indicate the convergence rate (accurate of order $O(k^{2-\frac{\lambda}{2}})$ in time and 
    fourth-order convergent in space) of the algorithm. Our future works will consider the numerical solution
     of the two-dimensional time-fractional convection-diffusion-reaction equation with variable coefficients using the new approach. Furthermore, the 
     development of more efficient algorithms to reducing computational costs will be also subjected of other works.\\

     \textbf{Acknowledgment.} The work of the first two authors has been partially supported by the Deanship of Scientific Research of Imam Mohammad Ibn 
     Saud Islamic University (IMSIU) under the Grant No. $331203.$\\

          \begin{figure}
         \begin{center}
          Stability and convergence rate of a two-level fourth-order approach for time-fractional advection-diffusion with $\lambda=0.9$ and
       $k=h^{\frac{4}{2-\frac{\lambda}{2}}}.$
         \begin{tabular}{c c}
         \psfig{file=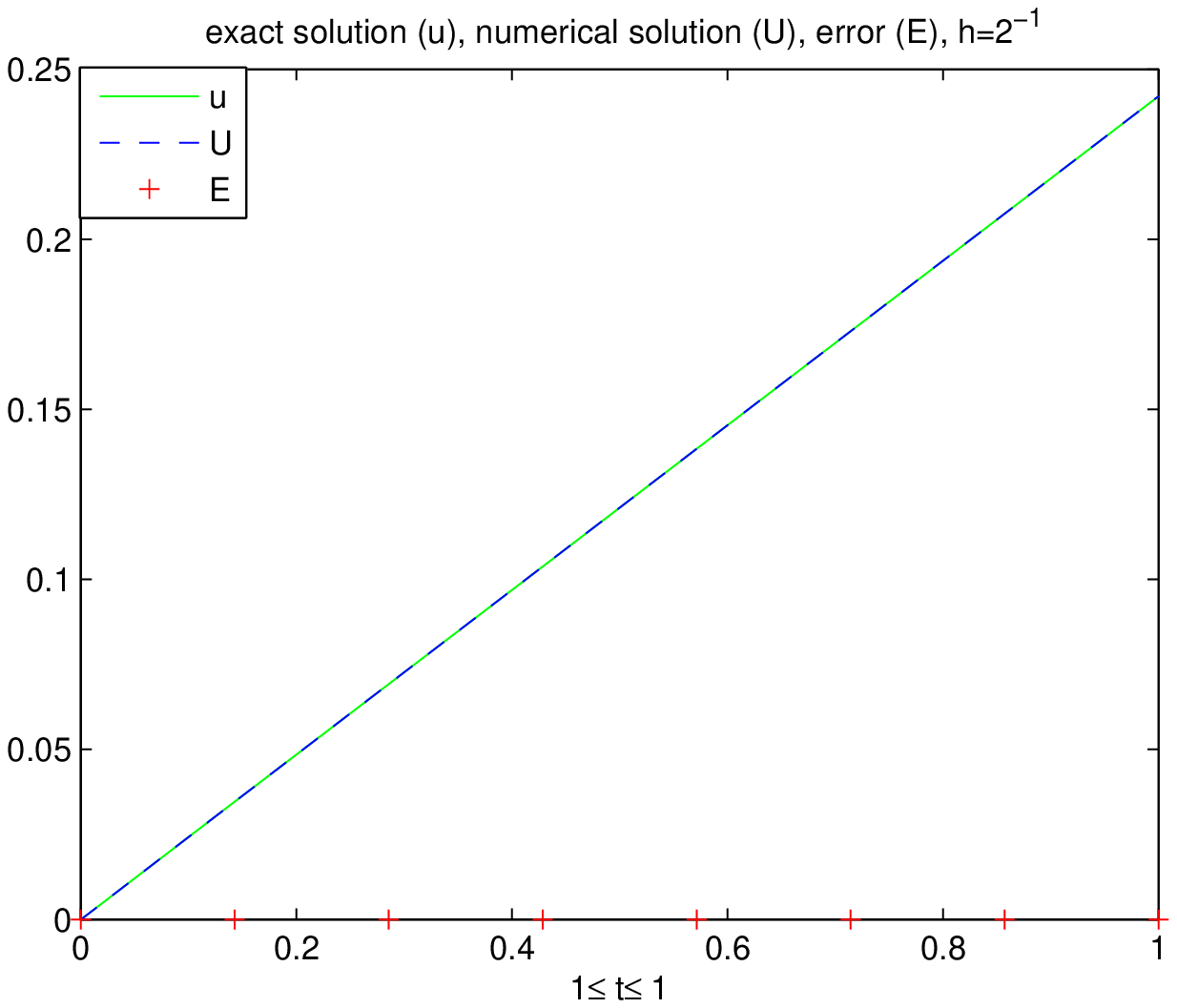,width=7cm} & \psfig{file=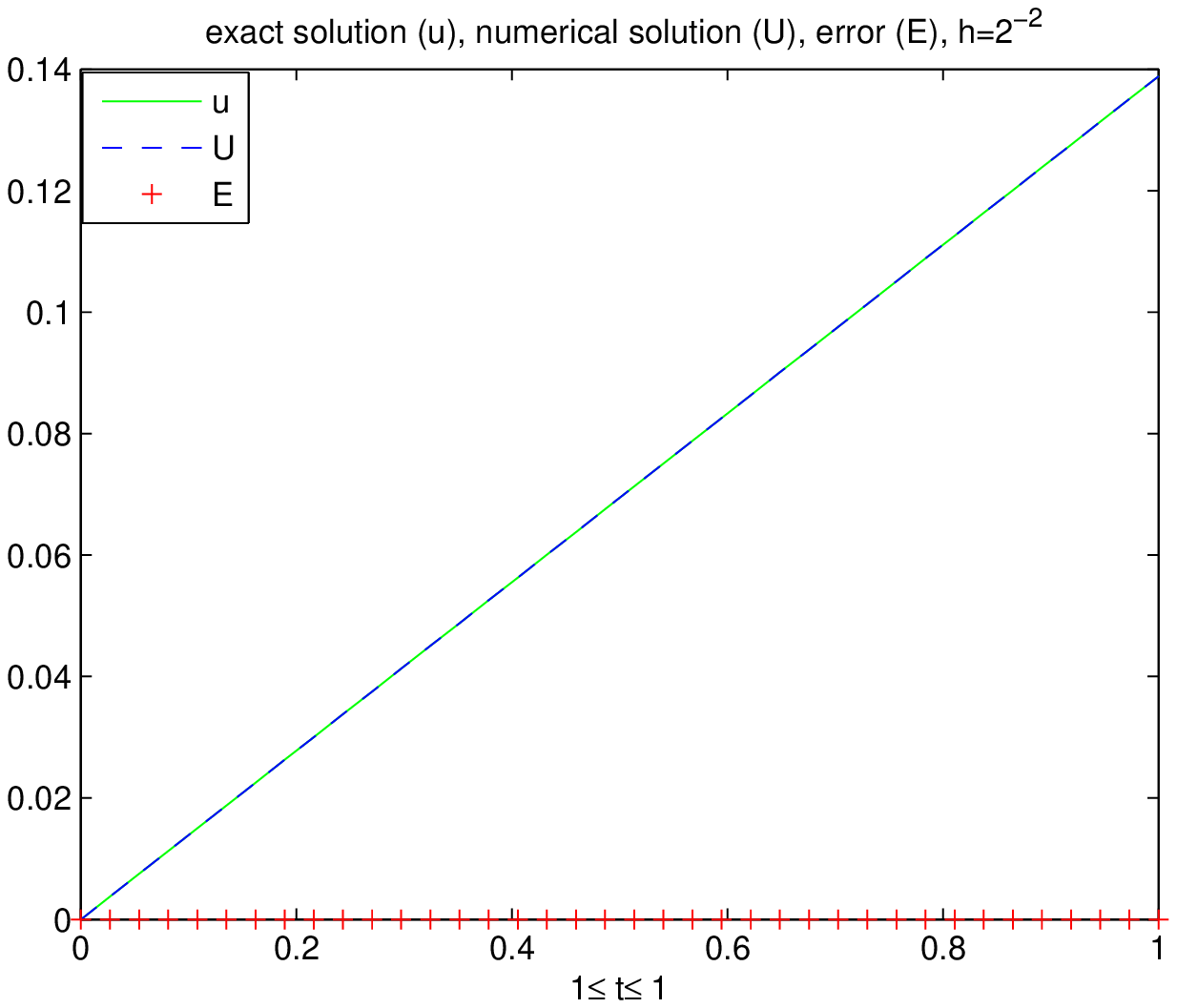,width=7cm}\\
         \psfig{file=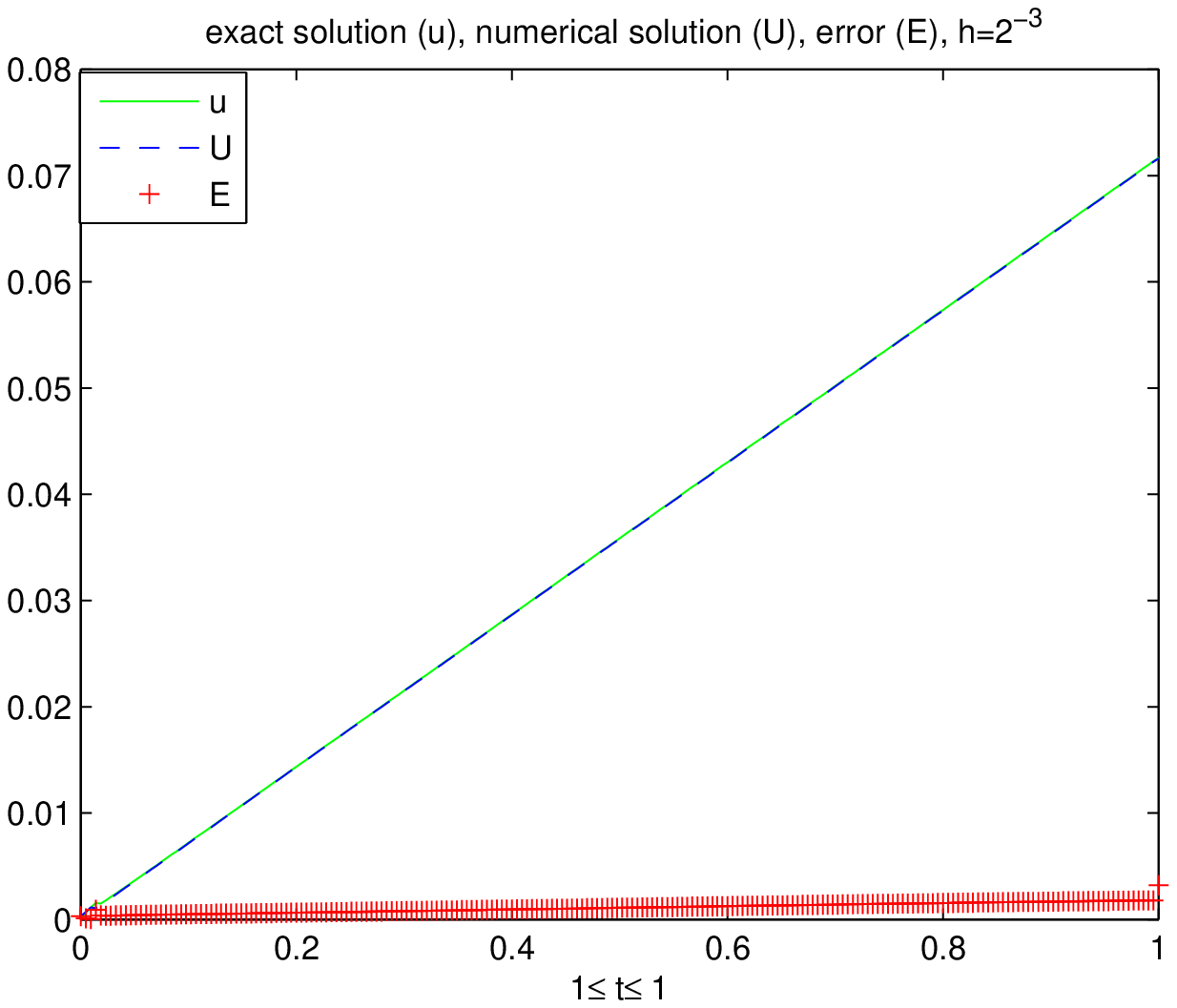,width=7cm} & \psfig{file=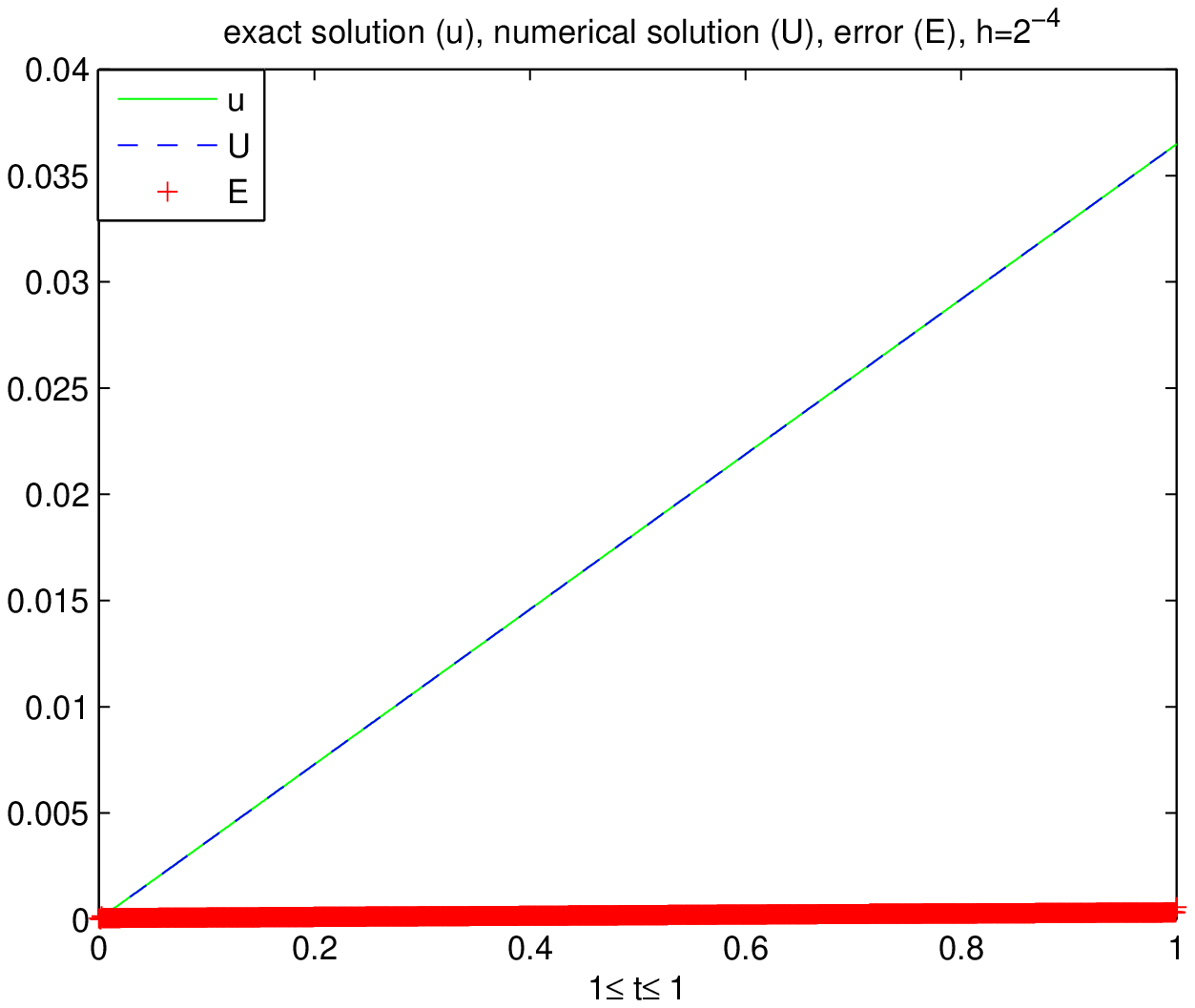,width=7cm}\\
         \end{tabular}
        \end{center}
        \caption{Exact solution (u: in green), Numerical solution (U: in blue) and Error (E: in red) for Problem 1}
        \label{figure1}
        \end{figure}

           \begin{figure}
         \begin{center}
          Analysis of stability and convergence rate of a two-level fourth-order method for time-fractional advection-diffusion with
      $\lambda=0.66$ and $k=h^{\frac{4}{2-\frac{\lambda}{2}}}.$
         \begin{tabular}{c c}
         \psfig{file=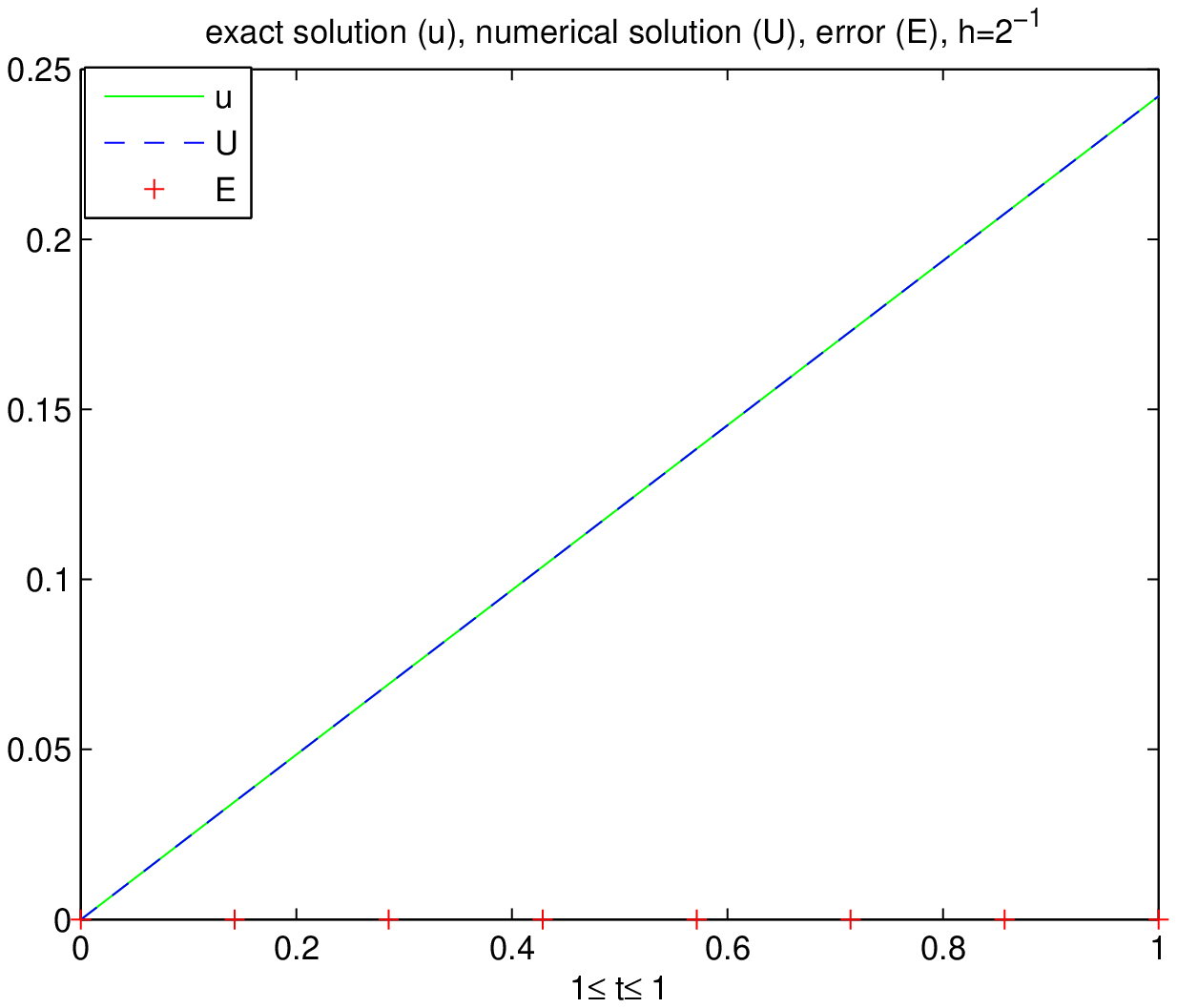,width=7cm} & \psfig{file=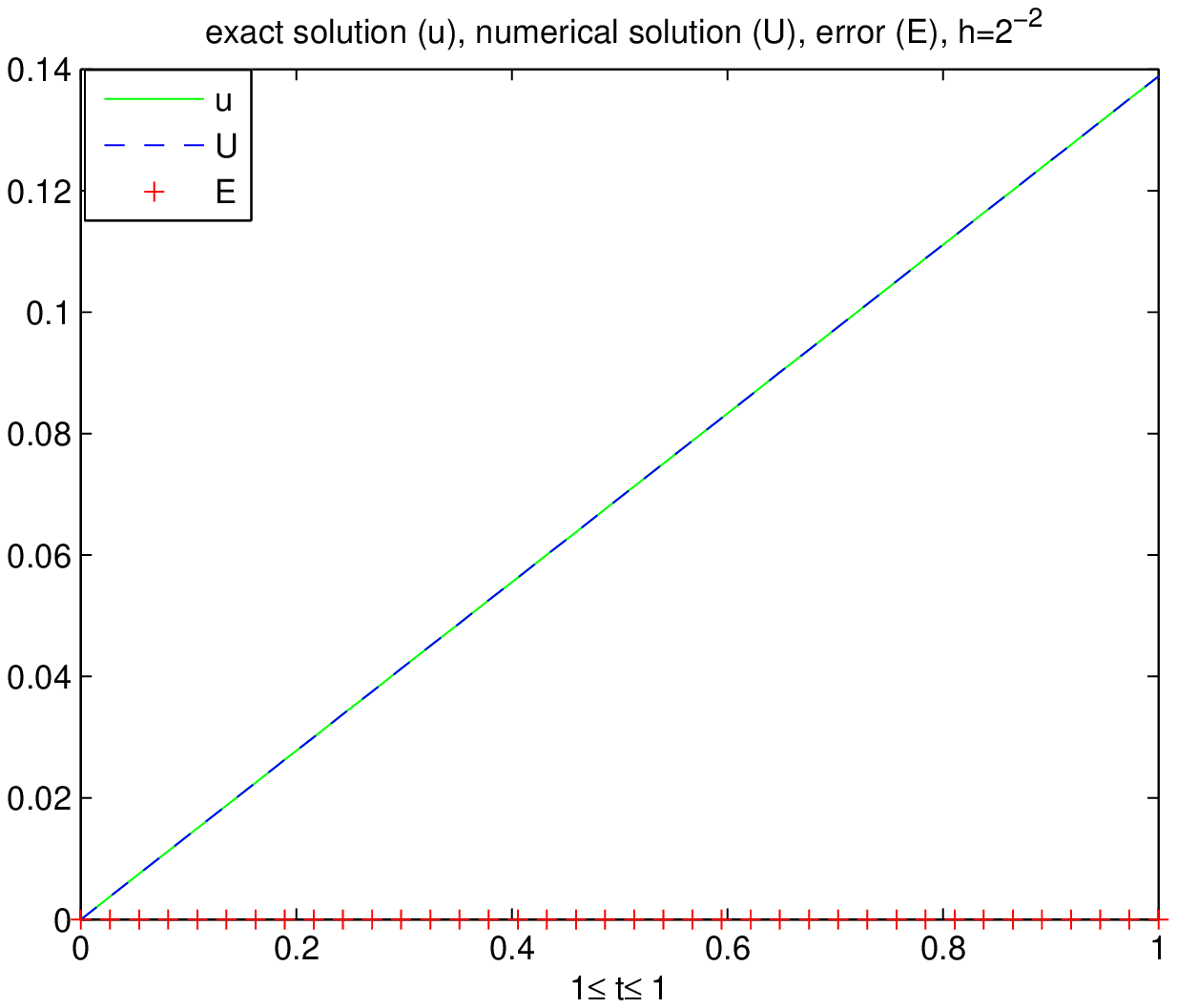,width=7cm}\\
         \psfig{file=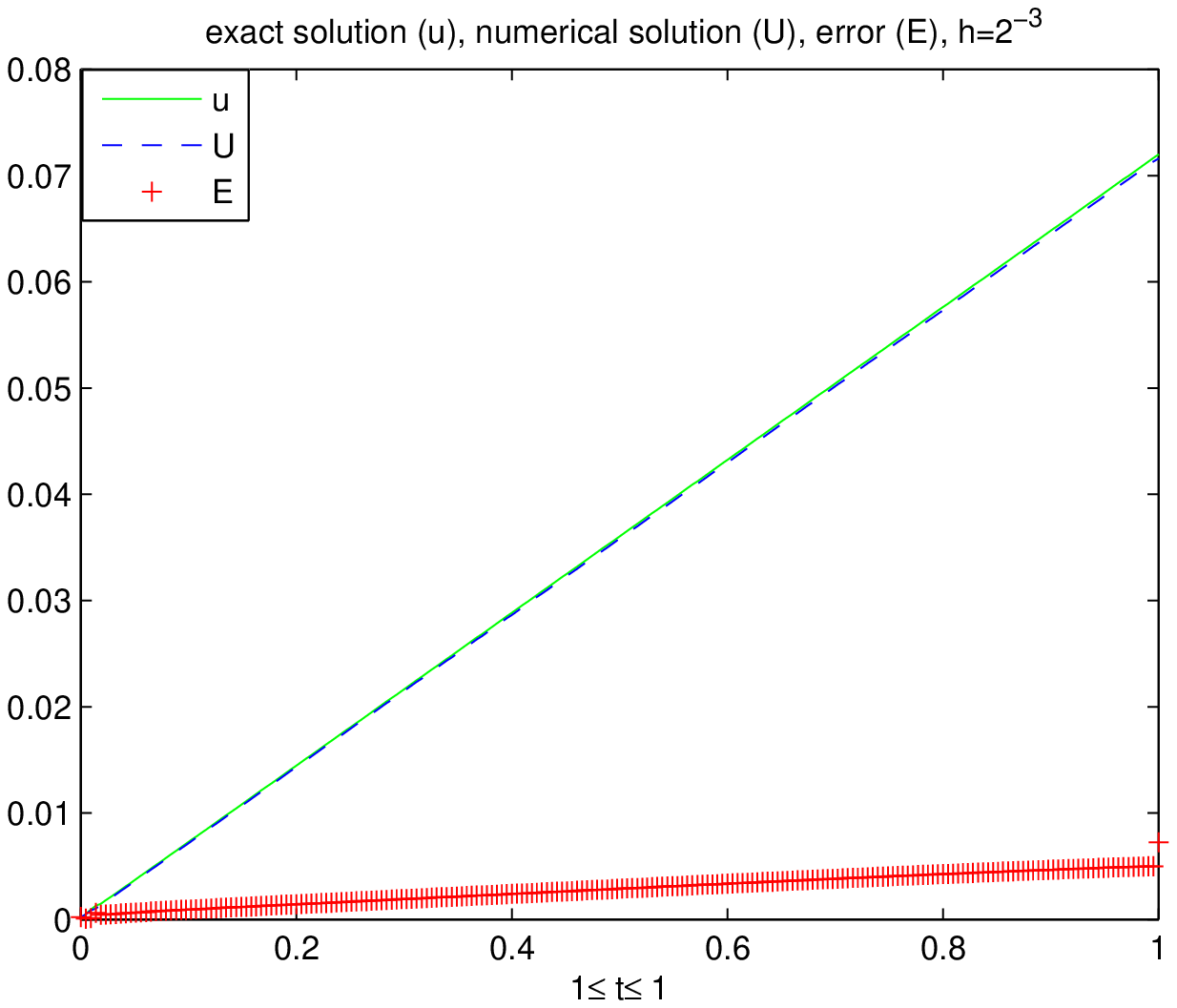,width=7cm} & \psfig{file=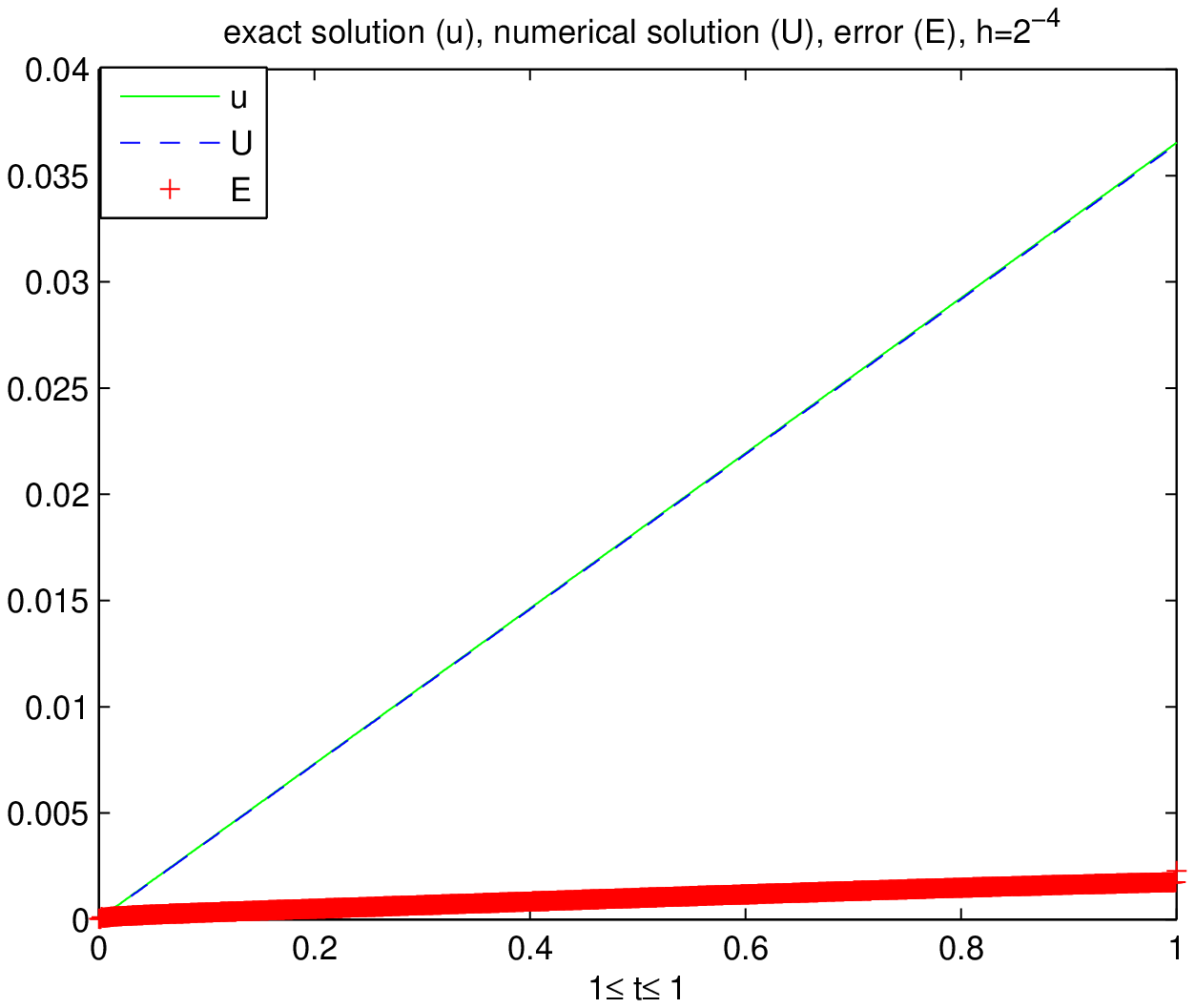,width=7cm}\\
         \end{tabular}
        \end{center}
        \caption{Exact solution (u: in green), Numerical solution (U: in blue) and Error (E: in red) for Problem 1}
        \label{figure2}
        \end{figure}

       \begin{figure}
         \begin{center}
          Stability and convergence rate of a two-level fourth-order numerical scheme for time-fractional advection-diffusion with $\lambda=0.9$ and
      $k=h^{\frac{4}{2-\frac{\lambda}{2}}}.$
         \begin{tabular}{c c}
         \psfig{file=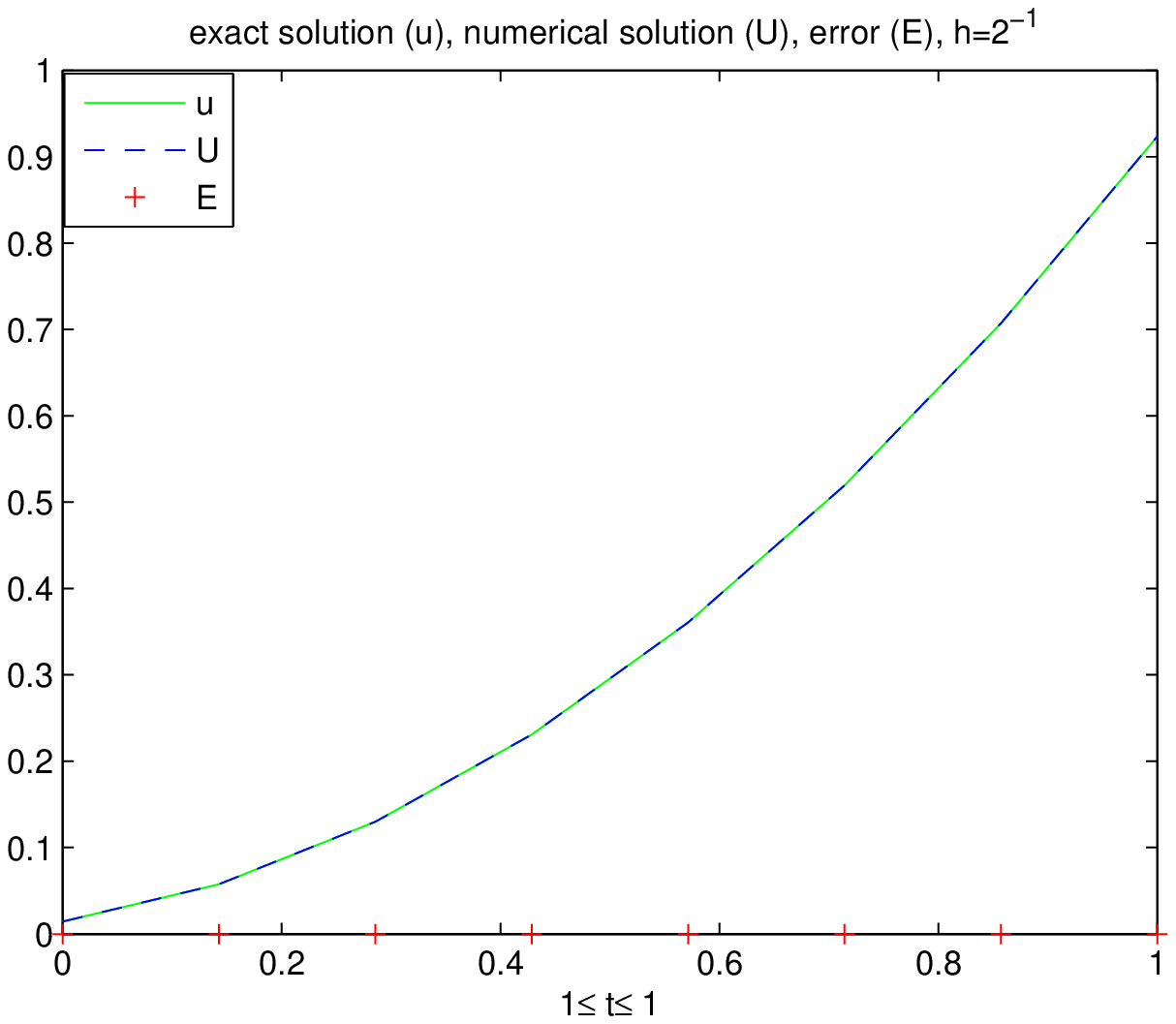,width=7cm} & \psfig{file=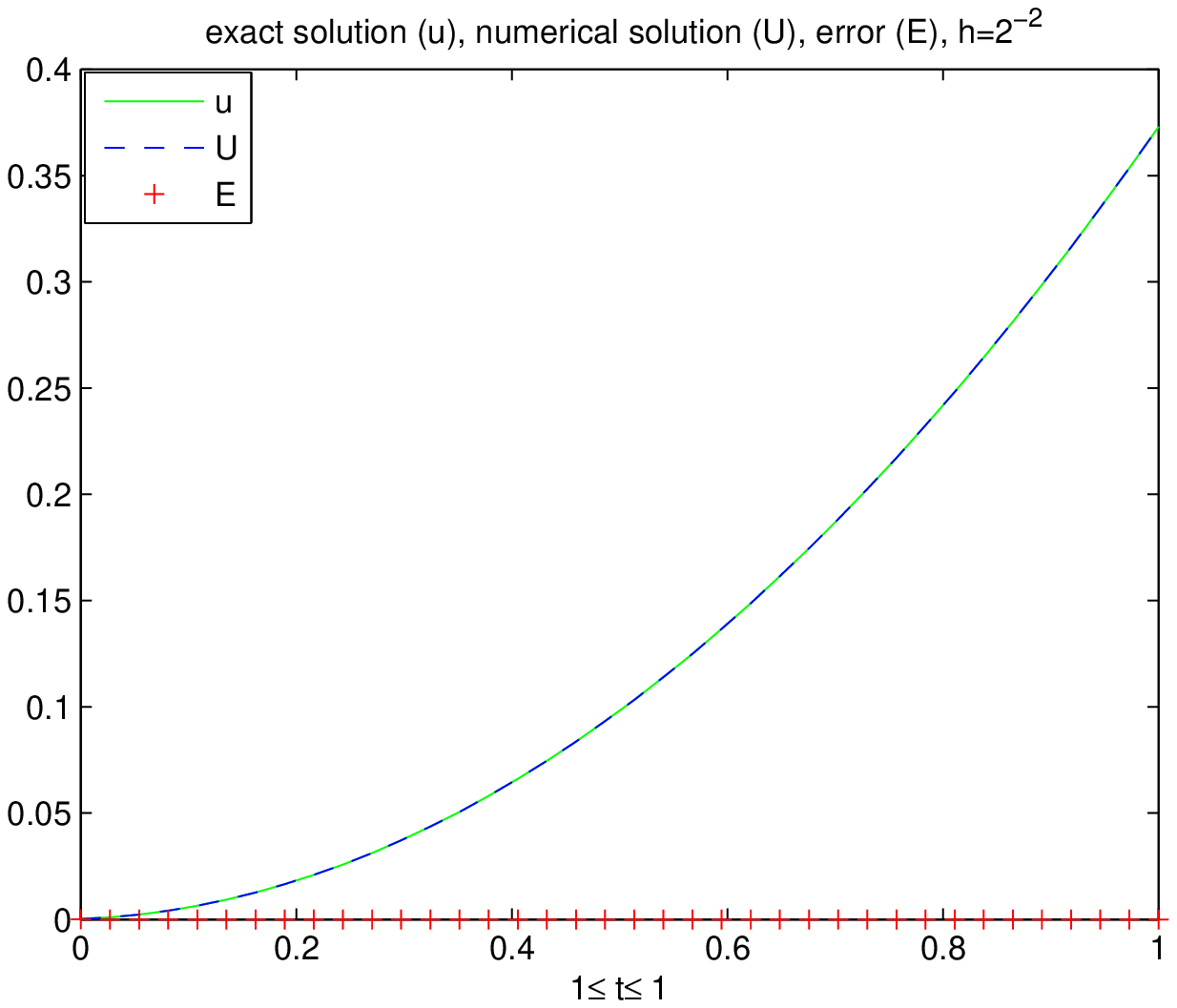,width=7cm}\\
         \psfig{file=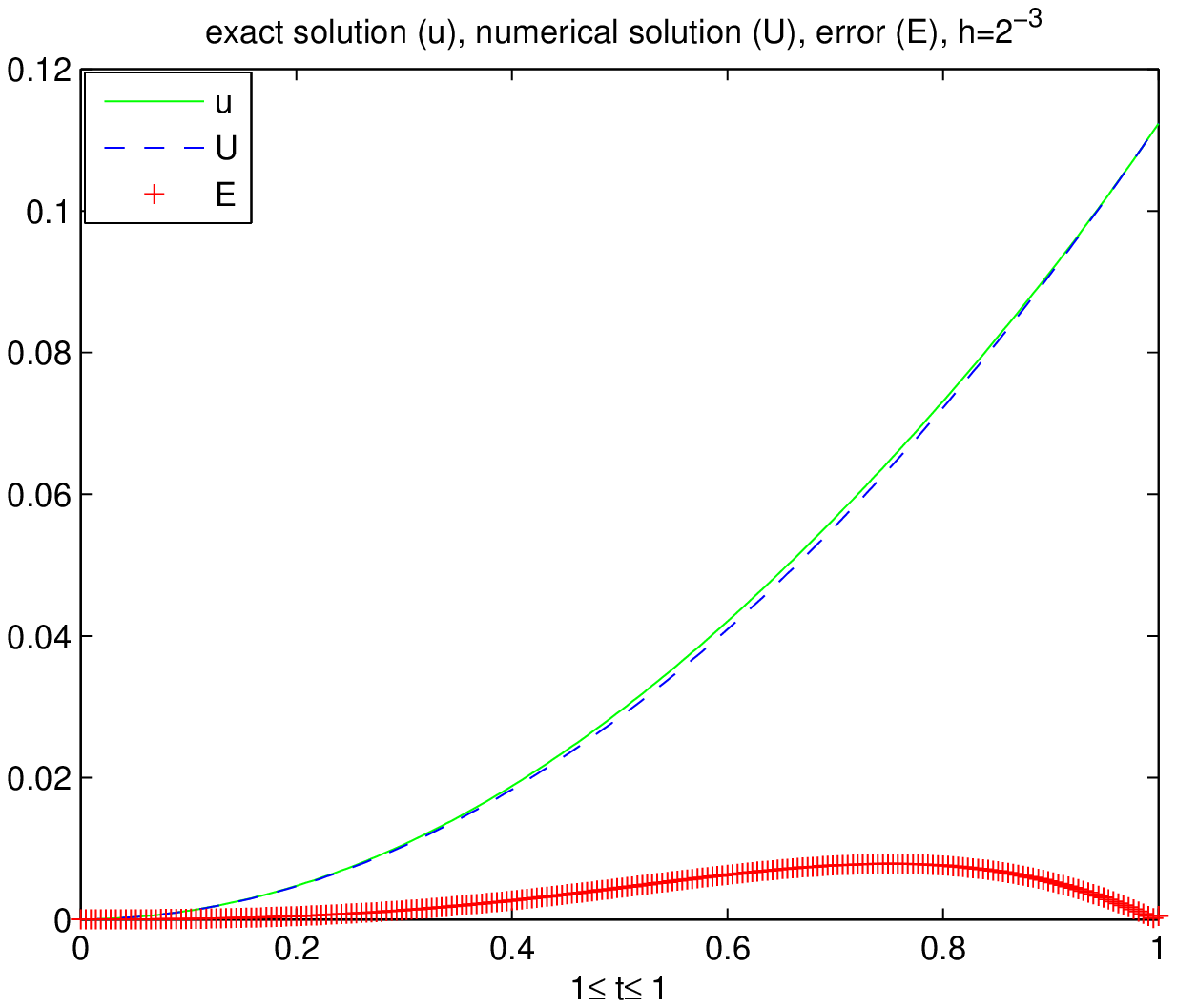,width=7cm} & \psfig{file=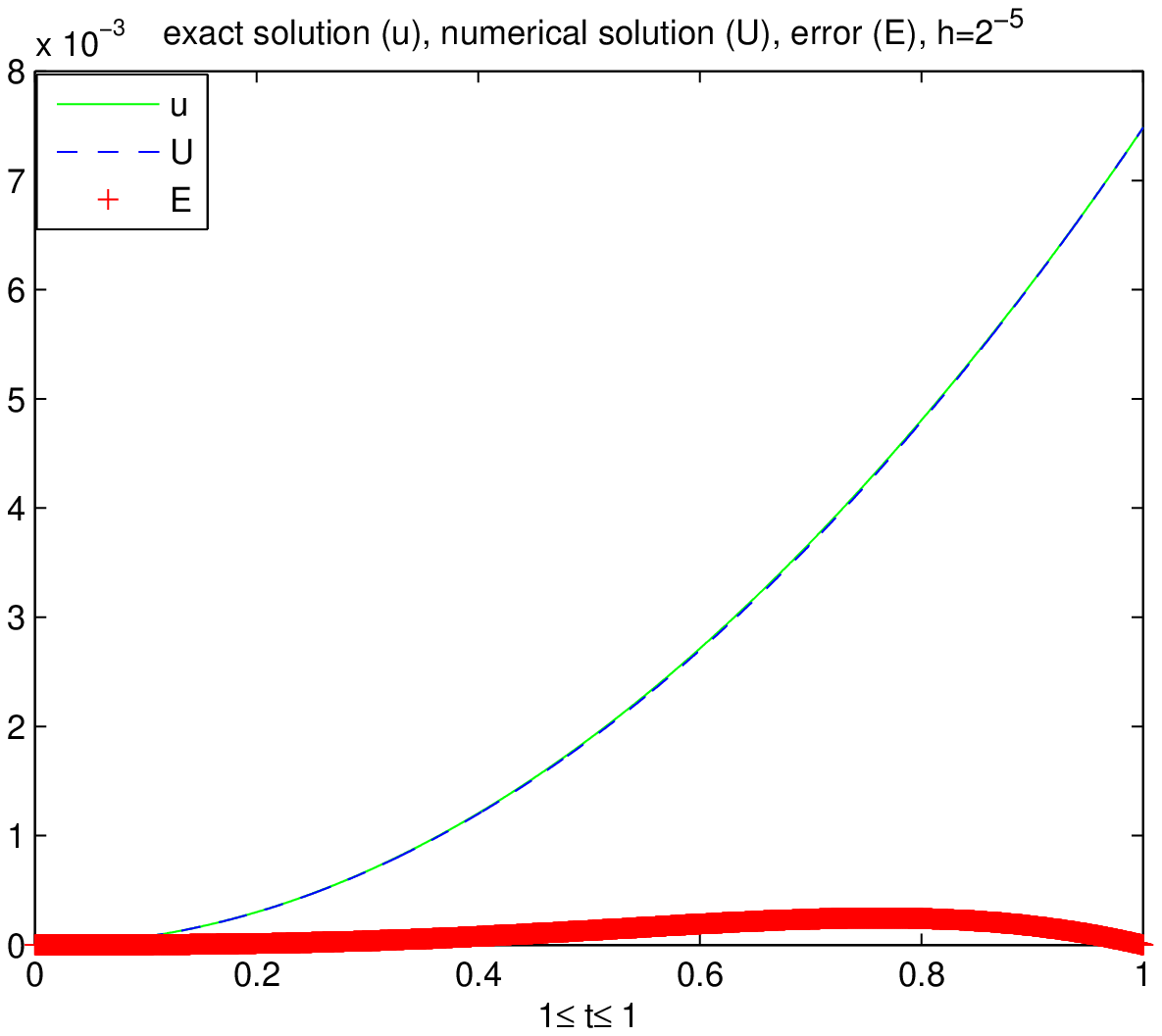,width=7cm}\\
         \end{tabular}
        \end{center}
        \caption{Exact solution (u: in green), Numerical solution (U: in blue) and Error (E: in red) for Problem 2}
        \label{figure3}
        \end{figure}

         \begin{figure}
         \begin{center}
          Analysis of stability and convergence rate of a two-level fourth-order approach for for time-fractional advection-diffusion with $\lambda=0.66$
        and $k=h^{\frac{4}{2-\frac{\lambda}{2}}}.$
         \begin{tabular}{c c}
         \psfig{file=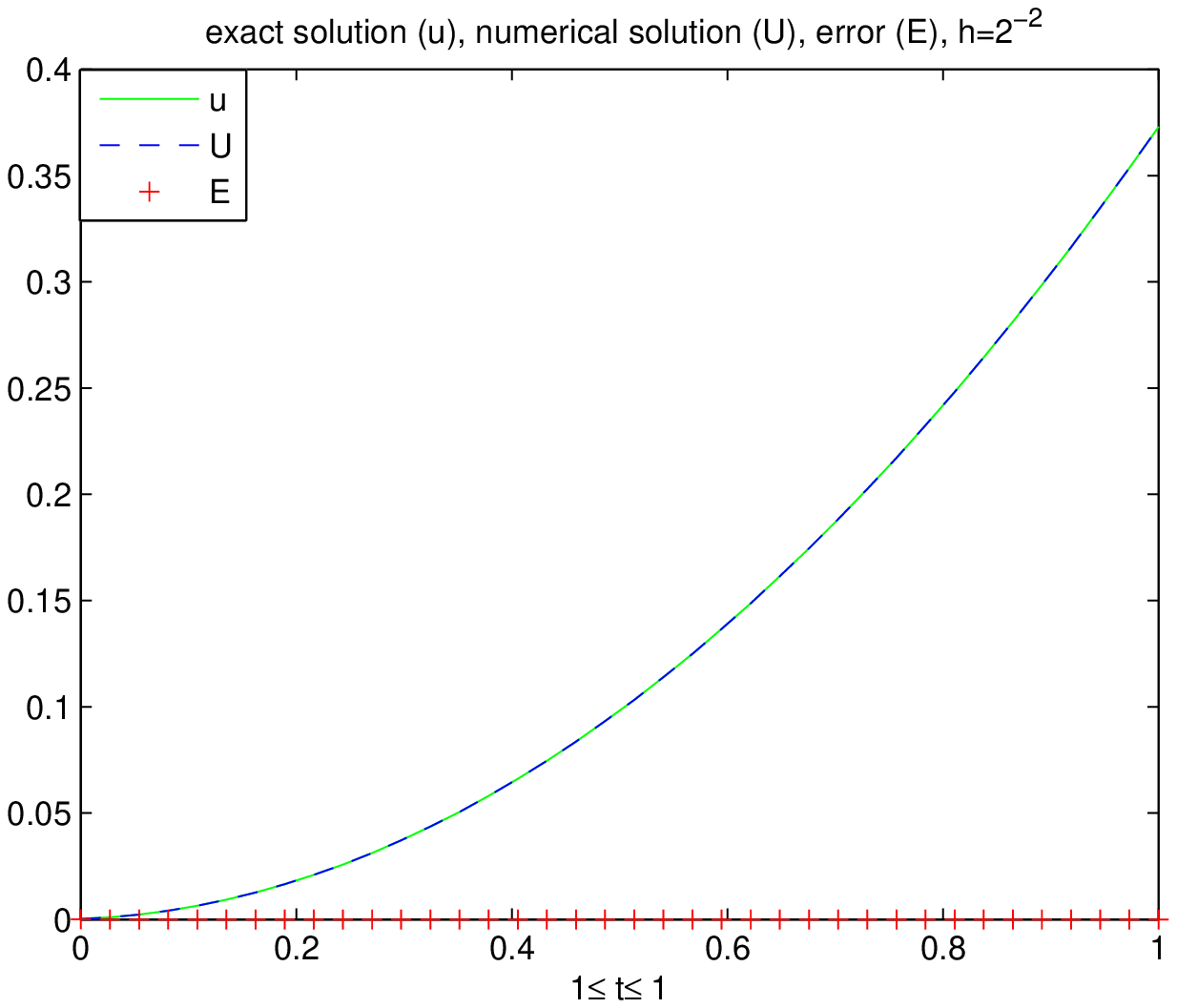,width=7cm} & \psfig{file=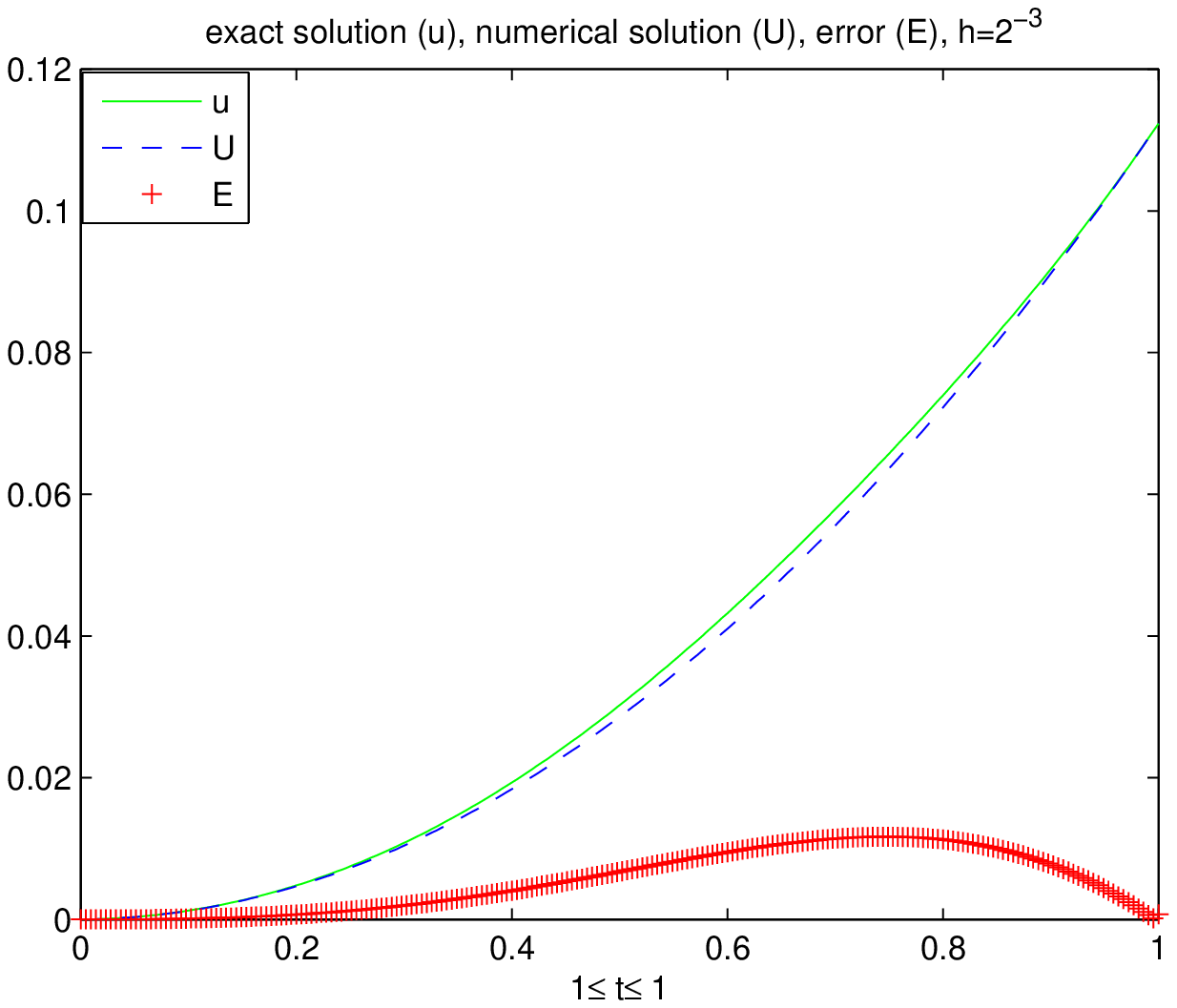,width=7cm}\\
         \psfig{file=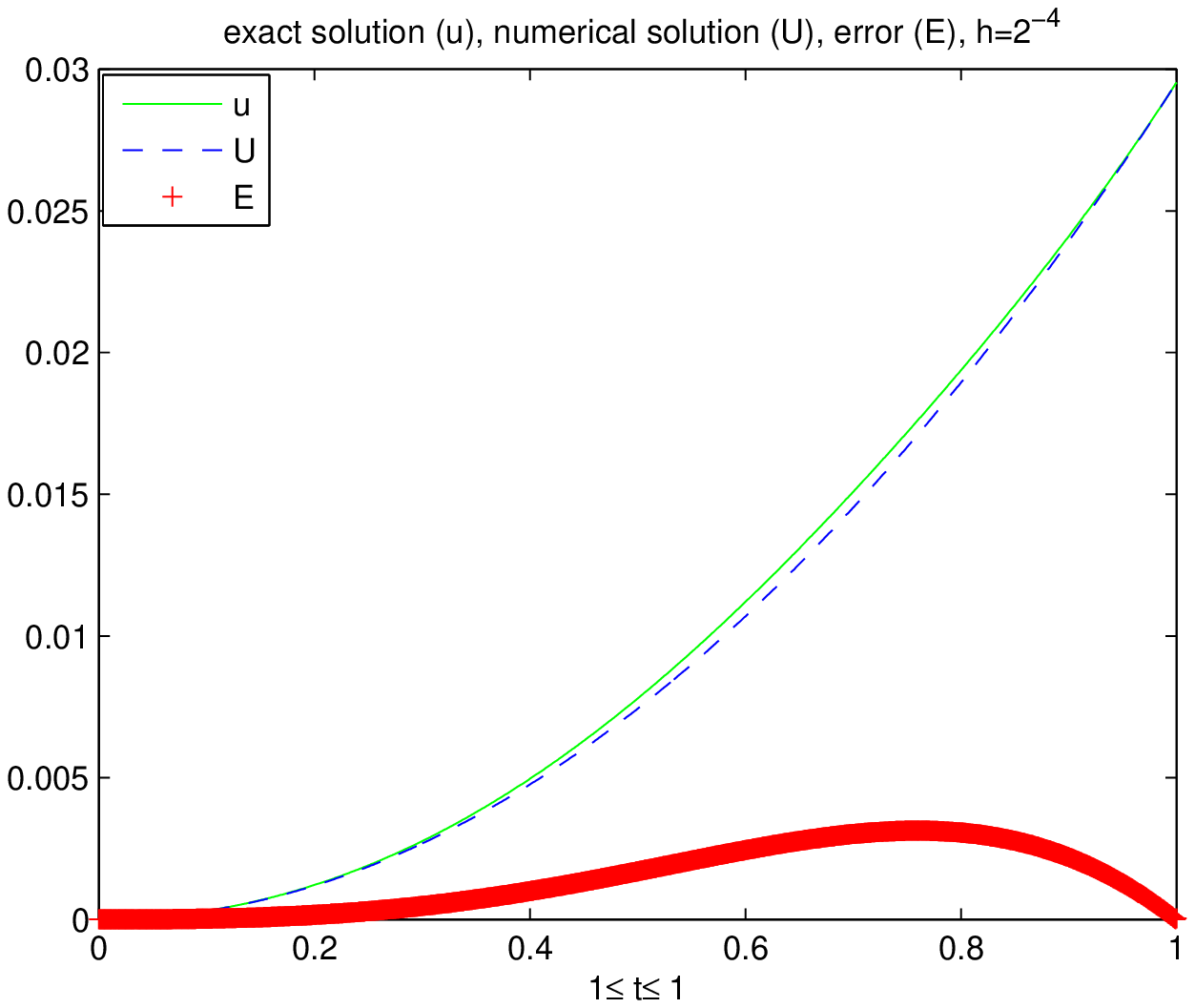,width=7cm} & \psfig{file=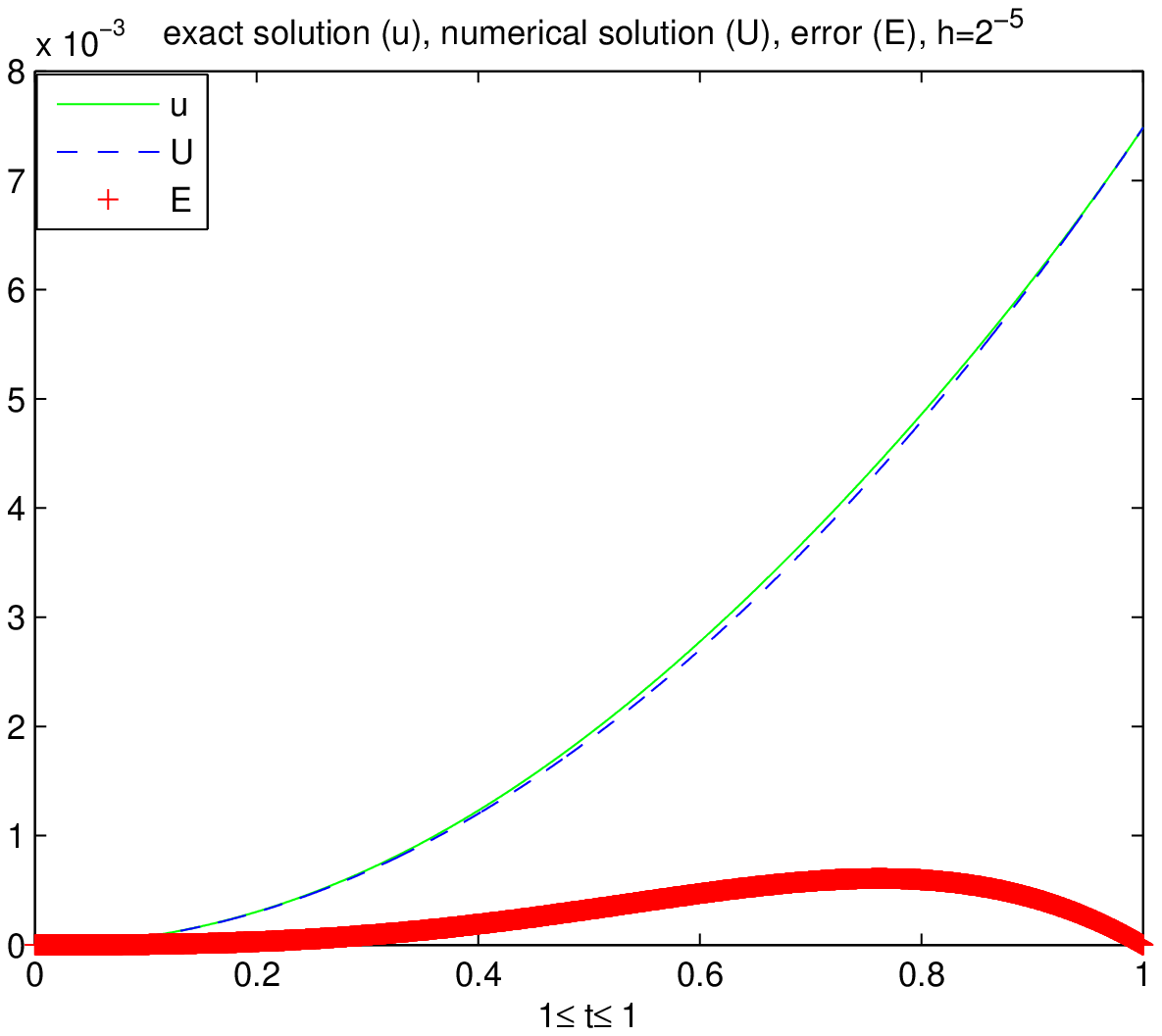,width=7cm}\\
         \end{tabular}
        \end{center}
        \caption{Exact solution (u: in green), Numerical solution (U: in blue) and Error (E: in red) for Problem 2}
        \label{figure4}
        \end{figure}
     \end{document}